\documentclass[12pt]{article}
\usepackage{amssymb,amsmath,amsthm,tikz,multirow, mathdots}
\usepackage{wasysym}
\usepackage{enumerate} 
\usepackage{hyperref} 
\usepackage [autostyle, english = british]{csquotes}
\usepackage [english]{babel}
\usepackage{mathtools}
\usepackage{bm}
\usepackage{authblk} 
\usepackage{geometry}

\usepackage{comment} 

\usetikzlibrary{calc,arrows, arrows.meta, math} 

\usetikzlibrary{shapes.geometric}

\newtheorem*{theorem*}{Theorem}

\theoremstyle{definition}
\newtheorem*{definition*}{Definition}
\newtheorem*{case*}{Case}
\newtheorem*{subcase*}{Subcase}
\newtheorem*{subsubcase*}{Subsubcase}

\theoremstyle{plain}
\newtheorem{thm}{Theorem}[section]
\newtheorem{lem}[thm]{Lemma}
\newtheorem{prop}[thm]{Proposition}

\theoremstyle{definition}

\theoremstyle{remark}

\numberwithin{equation}{section}

\newcommand{\newpart}{\subsubsection*}

\newcommand\floor[1]{\lfloor#1\rfloor} 
\newcommand{\AVC}{\text{AVC}} 
\newcommand\sbullet[1][.5]{\mathbin{\vcenter{\hbox{\scalebox{#1}{$\bullet$}}}}} 
\newcommand{\uline}{\rule[0pt]{10pt}{0.4pt}}

\newcommand{\quotes}[1]{``#1''} 

\newcommand{\arcThroughThreePoints}[4][]{
\coordinate (middle1) at ($(#2)!.5!(#3)$);
\coordinate (middle2) at ($(#3)!.5!(#4)$);
\coordinate (aux1) at ($(middle1)!1!90:(#3)$);
\coordinate (aux2) at ($(middle2)!1!90:(#4)$);
\coordinate (center) at ($(intersection of middle1--aux1 and middle2--aux2)$);
\draw[#1] 
 let \p1=($(#2)-(center)$),
      \p2=($(#4)-(center)$),
      \n0={veclen(\p1)},       
      \n1={atan2(\y1,\x1)}, 
      \n2={atan2(\y2,\x2)},
      \n3={\n2>\n1?\n2:\n2+360}
    in (#2) arc(\n1:\n3:\n0);
}

\newcommand{\bvert}{\vrule width 2pt}

\newcommand\Pentagon{\scaleobj{1}{\pentagon}}

\providecommand{\keywords}[1]{\noindent \textit{Keywords:} #1}
\providecommand{\subject}[1]{\noindent \textit{Mathematics Subject Classification:} #1}

\title{Dihedral Tilings of the Sphere by Regular Polygons and Quadrilaterals II: Regular Polygons with High Gonality and Rhombi}
\author[1]{Ho Man CHEUNG}
\author[2]{Hoi Ping LUK}
\affil[2]{The Hong Kong University of Science \& Technology} 
\affil[1]{email: hmcheungae@connect.ust.hk}
\affil[2]{email: hoi@connect.ust.hk}

\begin{document}
\maketitle

\begin{abstract} We classify the dihedral edge-to-edge tilings of the sphere by regular polygons with gonality $m\ge5$ and rhombi. \\

 \\

\subject{05B45, 52C20, 51M09, 51M20}
\end{abstract}

\section{Introduction}

This paper is the second of the series to classify dihedral tilings of the sphere, where one prototile is a regular polygon and the other is a rhombus. The two prototiles in the first \cite{luk2} of the series are one square and one rhombus. The two prototiles in this paper are one regular polygon ($m$-gon with $m\ge5$ and edge combination $x^m$ and angles $\alpha$) and one rhombus (with edge combination $x^4$ and angles $\beta,\gamma$). The prototiles are illustrated in Figure \ref{Fig-a5-a4-angles} where the regular polygon is shaded and the rhombus is unshaded. Throughout this paper, the shaded tiles are always regular $m$-gons. We assume that the degree of a vertex is $\ge3$.

\begin{figure}[h!] 
\centering
\begin{tikzpicture}

\tikzmath{
\s=1;
\r=0.8;
\g=4;
\th=360/\g;
\x=\r*cos(\th/2);
\R=0.8;
\G=5;
\ph=360/\G;
\hx=360/6;
}

\begin{scope}[]

\foreach \a in {0,...,3} {

\draw[rotate=\a*\th]
	(1.5*\th:\r) -- (0.5*\th:\r)
;

\node at (\th+\th*\a: 0.95*\r) {\small $x$}; 

}

\foreach \a in {0,1} {

\node at (1.5*\th+\a*2*\th:0.625*\r) {\small $\beta$};
\node at (0.5*\th+\a*2*\th:0.65*\r) {\small $\gamma$};

}

\end{scope}

\begin{scope}[xshift=3*\s cm] 

\fill[gray!40]
	(90:\R) -- (90+\ph:\R) -- (90+2*\ph:\R) -- (90+3*\ph:\R) -- (90+4*\ph:\R) -- cycle
;

\foreach \a in {0,...,4} {

\draw[rotate=\a*\ph]
	(90-1*\ph:\R) -- (90:\R)
;

\node at (90+\a*\ph:0.7*\R) {\small $\alpha$};

\node at (\th-0.5*\ph+\ph*\a: 1.15*\r) {\small $x$}; 

}

\end{scope}

\begin{scope}[xshift=6*\s cm] 

\fill[gray!40]
	(90:\R) -- (90+\hx:\R) -- (90+2*\hx:\R) -- (90+3*\hx:\R) -- (90+4*\hx:\R) -- (90+5*\hx:\R) -- cycle
;

\foreach \a in {0,...,5} {

\draw[rotate=\a*\hx]
	(90-1*\hx:\R) -- (90:\R)
;

\node at (90+\a*\hx:0.7*\R) {\small $\alpha$};

\node at (\th-0.5*\hx+\hx*\a: 1.15*\r) {\small $x$}; 

}

\end{scope}

\node at (8*\s, 0) {\Large $\cdots$};

\end{tikzpicture}
\caption{Rhombus and regular polygons}
\label{Fig-a5-a4-angles}
\end{figure}
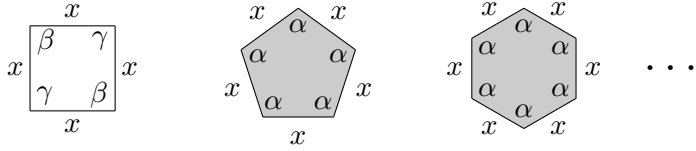


The main result is given below, where $f$ denotes the number of tiles. 

\begin{theorem*} The dihedral tilings of the sphere by regular polygons with gonality $m\ge5$ and rhombi are 
\begin{enumerate}[I.]
\item Earth map type: one infinite family of tilings with $f=8c-2$, where $c\ge2$ and $m=5$, and vertices $\{ \beta^2\gamma, \alpha\beta\gamma^c  \}$; 
\item Prism type: one infinite family of tilings with $f=m+2$, and vertex $\{ \alpha\beta\gamma \}$;
\item Archimedean type: 
\begin{itemize}
\item three triangular fusions of the snub dodecahedron with $m=5$, and $f=52$, and vertices $\{ \alpha\beta^2, \alpha\beta\gamma^2 \}$;
\item one quadrilateral subdivision of the truncated icosahedron with $m=5$, $f=72$, and vertices $\{ \beta^3, \alpha\beta\gamma^2 \}$.
\end{itemize}
\end{enumerate}
\end{theorem*}

The family of earth map type are derived from the monohedral tilings $E_{\square}4$ in \cite{cly}. The first picture of Figure \ref{Fig-a5-a4-EMTs} shows a part of $E_{\square}4$ with $\gamma^3$ at one end and $\gamma^2$ at the other. Let $\mathcal{T}$ denote the general version of it consisting of $2c-1$ rhombi with $\gamma^c$ at one end and $\gamma^{c-1}$ at the other (the first picture shows $c=3$). The earth map type tilings are constructed by four copies of $\mathcal{T}$ and $2$ regular pentagons. The four $\mathcal{T}$'s are glued together between the two pentagons as illustrated in the second picture.

\begin{figure}[h!]
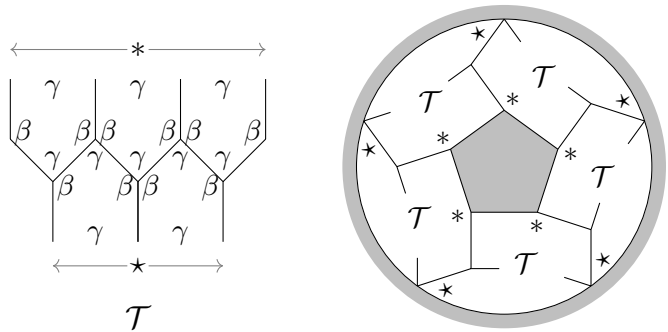
 
\centering

\caption{The earth map type, an infinite family of tilings with $\beta^2\gamma, \alpha\beta\gamma^c$}
\label{Fig-a5-a4-EMTs}
\end{figure}


The family of prism type are derived from prisms with two congruent $m$-gons on top and at the bottom and the cylindrical middle consists of $m$ rhombi. We remark that the corresponding tilings with the regular $m$-gon as regular triangle or square respectively also belong to the same family in Figure \ref{Fig-am-a4-Tilings-prism}.

\begin{figure}[h!]
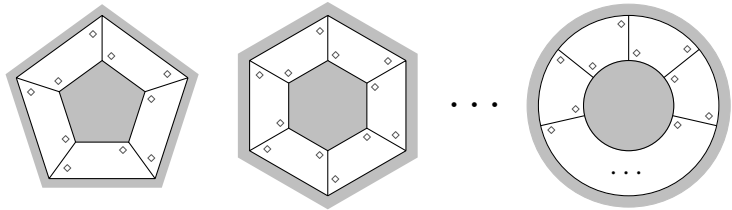
 
\centering
%
\caption{The earth map type: prisms consist of two regular $m$-gons and $m$ rhombi, $m\ge5$ and $\diamond=\beta$}
\label{Fig-am-a4-Tilings-prism}
\end{figure}


\begin{figure}[h!]
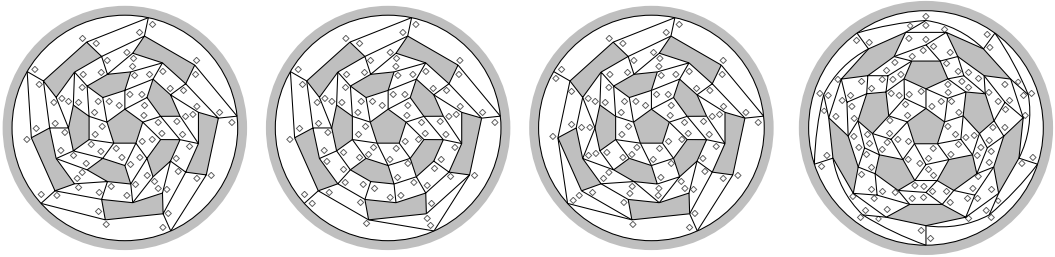
 
\centering
%
\caption{The four Archimedean type tilings, by regular pentagons and rhombi, $\diamond=\beta$}
\label{Fig-a5-a4-Archimedean-Tilings}
\end{figure}

The tilings of Archimedean type in Figure \ref{Fig-a5-a4-Archimedean-Tilings}. The first three tilings are triangular fusions of the snub dodecahedron and the fourth one is a quadrilateral subdivision of a (deformed) truncated icosahedron (or commonly known as the Adidas Telstar football).

The three triangular fusions of the snub dodecahedron are explained in Figure \ref{Fig-a5-a4-snub-dodecahedron}. The first picture illustrates the snub dodecahedron, which is a dihedral tiling of the sphere by regular triangles and regular pentagons. In general, if the triangles in a dihedral tiling can be grouped into adjacent pairs, then by fusing all the pairs we get quadrilaterals. In the snub dodecahedron, all triangles are regular. Therefore the fusion gives  congruent rhombi. The dashed line in Figure \ref{Fig-a5-a4-snub-dodecahedron} represent the choice of adjacent pairs. The second, third and fourth picture are all possible ways to group the pairs.

\begin{figure}[h!]
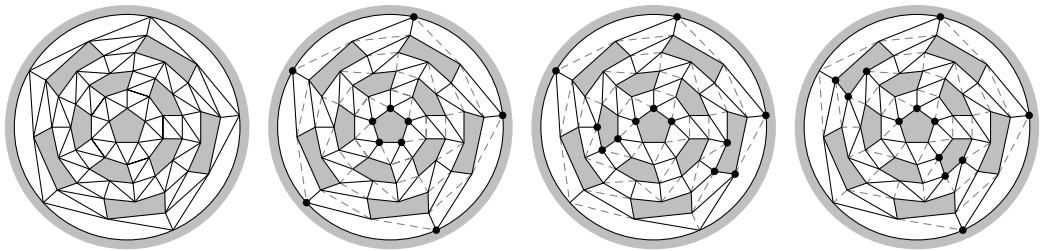
 
\centering

\caption{Three triangular fusions of the snub dodecahedron}
\label{Fig-a5-a4-snub-dodecahedron}
\end{figure}

To distinguish the last three tilings in Figure \ref{Fig-a5-a4-snub-dodecahedron}, we denote the vertices at the pentagons with exterior edge configuration \quotes{dashed-solid-solid} by $\bullet$'s. To highlight the difference, we only display them at pentagons with at least three such vertices. The distribution of $\bullet$'s shows that the the second tiling is different from the third and the fourth in Figure \ref{Fig-a5-a4-snub-dodecahedron}. The third and the fourth tiling are differentiated in Figure \ref{Fig-a5-a4-diff-snub-dodecahedron}. We look at the (shortest) path through rhombi from the middle $\bullet$ of a trio to the opposite edge to a middle $\bullet$ of another trio. In the left picture, from the trio at the centre pentagon to the specified edge of its immediate neighbouring pentagon, the path goes through three rhombi. In the right picture, the corresponding path goes through two rhombi. 


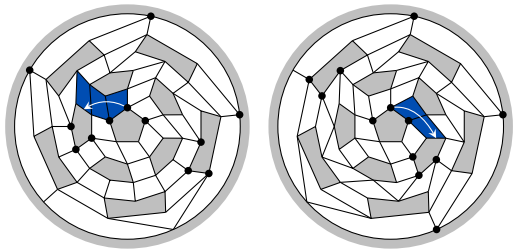
\begin{figure}[h!] 
\centering
\begin{tikzpicture}

\tikzmath{
\s=1;
\r=0.5;
\rr=0.1*\r;
\th=360/5;
}

\begin{scope}[xshift=3.5*\s cm] 

\fill[gray!50]
	(0,0) circle (3.25*\r)
;

\fill[white]
	(0,0) circle (3*\r)
;

\fill[gray!50]
	(90:0.5*\r) -- (90+\th:0.5*\r) -- (90+2*\th:0.5*\r) -- (90+3*\th:0.5*\r) -- (90+4*\th:0.5*\r)
;

\foreach \a in {0,...,4} {

\fill[gray!50, rotate=\a*\th]
	(90-\th:1*\r) -- (\th/6:1.5*\r) -- (0.5*\th:1.5*\r) -- (5*\th/6:1.5*\r)  -- (90-0.5*\th:1*\r) -- cycle
;

\fill[gray!50, rotate=\a*\th]
	(90-\th/12:2*\r) --  (5*\th/6:2*\r)  --  (0.5*\th:2*\r)  -- (42:2.5*\r) --  (90-\th/6:2.5*\r)  -- cycle
;

}

\fill[teal!60!blue, opacity=0.6]
	(90:0.5*\r) -- (90:1*\r) -- (90+0.5*\th:1*\r) -- (-\th/6+2*\th:1.5*\r) -- (-\th/6+2*\th:2*\r) -- (\th/6+2*\th:1.5*\r) -- (90+\th:1*\r) -- (90+\th:0.5*\r) -- cycle
;

	(90-\th:0.5*\r) -- (90-0.5*\th:1*\r) 
	(90:0.5*\r) -- (90+0.5*\th:1*\r) 
	(90+\th:0.5*\r) -- (90+1.5*\th:1*\r) 
	(90+2*\th:0.5*\r) -- (90+2*\th:1*\r) 
	(90-2*\th:0.5*\r) -- (90-2.5*\th:1*\r) 
	(90-2*\th:0.5*\r) -- (90-1.5*\th:1*\r) 
	(90:1*\r) -- (90-0.5*\th:1*\r) 
	(90-\th:1*\r) -- (90-1.5*\th:1*\r) 
	(90+\th:1*\r) -- (90+0.5*\th:1*\r) 
	(90-2*\th:1*\r) -- (5*\th/6-2*\th:1.5*\r) 
	(90+2*\th:1*\r) -- (5*\th/6+2*\th:1.5*\r) 
	(\th/6:1.5*\r) -- (-\th/6:1.5*\r) 
	(\th/6+\th:1.5*\r) -- (-\th/6+\th:1.5*\r) 
	(\th/6+2*\th:1.5*\r) -- (-\th/6+2*\th:1.5*\r) 
	(5*\th/6-1*\th:2*\r) -- (0.5*\th-1*\th:1.5*\r) 
	(\th/6:1.5*\r) -- (\th/6:2*\r) 
	(5*\th/6:2*\r) -- (0.5*\th:1.5*\r) 
	(\th/6:2*\r) -- (0.5*\th:2*\r) 
	(5*\th/6+\th:2*\r) -- (0.5*\th+\th:1.5*\r) 
	(\th/6+\th:1.5*\r) -- (\th/6+\th:2*\r) 
	(\th/6+\th:2*\r) -- (0.5*\th+\th:2*\r) 
	(\th/6+2*\th:1.5*\r) -- (\th/6+2*\th:2*\r) 
	(\th/6+2*\th:2*\r) -- (0.5*\th+2*\th:2*\r) 
	(5*\th/6+2*\th:2*\r) -- (0.5*\th+2*\th:1.5*\r) 
	(90-\th/12+2*\th:1.5*\r) --  (5*\th/6+2*\th:2*\r) 
	(0.5*\th+3*\th:1.5*\r) -- (\th/6+3*\th:2*\r) 
	(5*\th/6+3*\th:2*\r) -- (0.5*\th+3*\th:1.5*\r) 
	(90-\th/12+3*\th:1.5*\r) --  (5*\th/6+3*\th:2*\r) 
	(0.5*\th-1*\th:1.5*\r) -- (\th/6-1*\th:2*\r) 
	(42:2.5*\r) -- (\th/12:2.5*\r) 
	(42+\th:2.5*\r) -- (\th/12+\th:2.5*\r) 
	(42+2*\th:2.5*\r) -- (\th/12+2*\th:2.5*\r) 
	(0.5*\th+3*\th:2*\r) -- (\th/12+3*\th:2.5*\r) 
	(0.5*\th-1*\th:2*\r) -- (\th/12-1*\th:2.5*\r) 
	(90-\th/6-2*\th:2.5*\r) -- (90-\th/6-2*\th:3*\r) 
	(42:2.5*\r) -- (90-\th/6:3*\r) 
	(42+\th:2.5*\r) -- (90-\th/6+\th:3*\r) 
	(42+2*\th:2.5*\r) -- (90-\th/6+2*\th:3*\r) 
	(42-\th:2.5*\r) -- (90-\th/6-\th:3*\r) 
	(42+3*\th:2.5*\r) -- (\th/12+3*\th:3*\r) 
;

\draw[]
	(90+2*\th:0.5*\r) -- (90+2.5*\th:1*\r) 
	(42+3*\th:2.5*\r) -- (90-\th/6+3*\th:3*\r)
;

\foreach \a in {0,1,2,3} {

\draw[rotate=\a*\th]
	(90-2*\th:0.5*\r) -- (90-2*\th:1*\r)
	(90-\th:0.5*\r) -- (90-1.5*\th:1*\r)
	(90-\th/6-\th:2.5*\r) -- (90-\th/6-\th:3*\r)
	(42-\th:2.5*\r) -- (\th/12-\th:3*\r)
;


}

\foreach \a in {0,1,2} {

\draw[rotate=\a*\th]
	(90-\th:1*\r) -- (5*\th/6-\th:1.5*\r)
	(90-\th/12-\th:1.5*\r) --  (5*\th/6-\th:2*\r) 
	(0.5*\th:1.5*\r) -- (\th/6:2*\r) 
	(0.5*\th:2*\r) -- (\th/12:2.5*\r) 
;








}





%

\foreach \a in {0,1} {

\draw[rotate=\a*\th]
	(90+2*\th:1*\r) -- (90+1.5*\th:1*\r) 
	(\th/6+3*\th:2*\r) -- (0.5*\th+3*\th:2*\r) 
	(\th/6+3*\th:1.5*\r) -- (\th/6+3*\th:2*\r) 
	(\th/6+3*\th:1.5*\r) -- (-\th/6+3*\th:1.5*\r)
	(42+3*\th:2.5*\r) -- (\th/12+3*\th:2.5*\r) 
;






}

\foreach \a in {0,...,4}
{

\tikzset{rotate=\a*\th}

\draw
	(90:0.5*\r) -- (90-\th:0.5*\r)  
	(90-0.5*\th:1*\r) -- (90-\th:1*\r) 
	(\th/6:1.5*\r) -- (0.5*\th:1.5*\r) 
	(0.5*\th:1.5*\r) -- (5*\th/6:1.5*\r)  
	(90-0.5*\th:1*\r) -- (5*\th/6:1.5*\r) 
	(90:1*\r) -- (90-\th/12:1.5*\r) 
	%
	(0.5*\th:1.5*\r) -- (0.5*\th:2*\r) 
	(-\th/6:2*\r) -- (\th/6:2*\r) 
	(0.5*\th:2*\r) -- (5*\th/6:2*\r) 
	(90-\th/6:2.5*\r) -- (42:2.5*\r) 
	(0.5*\th:2*\r) -- (42:2.5*\r) 
	(\th/6:2*\r) -- (\th/12:2.5*\r) 
	%
	(-\th/6:1.5*\r) -- (-\th/6:2*\r) 
	%
	%
	%
;


}

\draw[-stealth, white]
	(90:0.5*\r) to[out=150, in=30] ($(90+\th:1*\r) !1/2! (\th/6+2*\th:1.5*\r)$)
;

\draw
	(0,0) circle (3*\r)
;

\fill 
	(90:0.5*\r) circle (\rr)
	(90-\th:0.5*\r) circle (\rr)
	(90+\th:0.5*\r) circle (\rr)
	(90+1.5*\th:1*\r) circle (\rr)
	(5*\th/6+2*\th:1.5*\r) circle (\rr)
	(0.5*\th+2*\th:1.5*\r) circle (\rr)
	(-\th/6:2*\r) circle (\rr)
	(0.5*\th-1*\th:2*\r) circle (\rr)
	(42-\th:2.5*\r) circle (\rr)
	(\th/12:3*\r) circle (\rr)
	(\th/12+\th:3*\r) circle (\rr)
	(\th/12+2*\th:3*\r) circle (\rr)
;

\end{scope} 

\begin{scope}[xshift=7*\s cm] 

\fill[gray!50]
	(0,0) circle (3.25*\r)
;

\fill[white]
	(0,0) circle (3*\r)
;

\fill[gray!50]
	(90:0.5*\r) -- (90+\th:0.5*\r) -- (90+2*\th:0.5*\r) -- (90+3*\th:0.5*\r) -- (90+4*\th:0.5*\r)
;

\foreach \a in {0,...,4} {

\fill[gray!50, rotate=\a*\th]
	(90-\th:1*\r) -- (\th/6:1.5*\r) -- (0.5*\th:1.5*\r) -- (5*\th/6:1.5*\r)  -- (90-0.5*\th:1*\r) -- cycle
;

\fill[gray!50, rotate=\a*\th]
	(90-\th/12:2*\r) --  (5*\th/6:2*\r)  --  (0.5*\th:2*\r)  -- (42:2.5*\r) --  (90-\th/6:2.5*\r)  -- cycle
;

}

\fill[teal!60!blue, opacity=0.6]
	(90:0.5*\r) --  (90-0.5*\th:1*\r) -- (90-\th:1*\r) -- (5*\th/6-\th:1.5*\r) -- (90-1.5*\th:1*\r) -- (90-\th:0.5*\r) -- cycle
;

	(90-\th:0.5*\r) -- (90-0.5*\th:1*\r) 
	(90:0.5*\r) -- (90+0.5*\th:1*\r) 
	(90+\th:0.5*\r) -- (90+1.5*\th:1*\r) 
	(90+2*\th:0.5*\r) -- (90+2*\th:1*\r) 
	(90-2*\th:0.5*\r) -- (90-2.5*\th:1*\r) 
	(90-2*\th:0.5*\r) -- (90-1.5*\th:1*\r) 
	(90:1*\r) -- (90-0.5*\th:1*\r) 
	(90-\th:1*\r) -- (90-1.5*\th:1*\r) 
	(90+\th:1*\r) -- (90+0.5*\th:1*\r) 
	(90-2*\th:1*\r) -- (5*\th/6-2*\th:1.5*\r) 
	(90+2*\th:1*\r) -- (5*\th/6+2*\th:1.5*\r) 
	(\th/6:1.5*\r) -- (-\th/6:1.5*\r) 
	(\th/6+\th:1.5*\r) -- (-\th/6+\th:1.5*\r) 
	(\th/6+2*\th:1.5*\r) -- (-\th/6+2*\th:1.5*\r) 
	(5*\th/6-1*\th:2*\r) -- (0.5*\th-1*\th:1.5*\r) 
	(\th/6:1.5*\r) -- (\th/6:2*\r) 
	(5*\th/6:2*\r) -- (0.5*\th:1.5*\r) 
	(\th/6:2*\r) -- (0.5*\th:2*\r) 
	(42:2.5*\r) -- (\th/12:2.5*\r) 
	(5*\th/6+\th:2*\r) -- (0.5*\th+\th:1.5*\r) 
	(\th/6+\th:1.5*\r) -- (\th/6+\th:2*\r) 
	(\th/6+\th:2*\r) -- (0.5*\th+\th:2*\r) 
	(42+\th:2.5*\r) -- (\th/12+\th:2.5*\r) 
	(42:2.5*\r) -- (90-\th/6:3*\r) 
	(42+\th:2.5*\r) -- (90-\th/6+\th:3*\r) 
	(42-\th:2.5*\r) -- (90-\th/6-\th:3*\r) 
	(\th/6+2*\th:1.5*\r) -- (\th/6+2*\th:2*\r) 
	(42+2*\th:2.5*\r) -- (\th/12+2*\th:3*\r) 
	(0.5*\th+2*\th:2*\r) -- (\th/12+2*\th:2.5*\r) 
	(0.5*\th+2*\th:1.5*\r) -- (0.5*\th+2*\th:2*\r) 
	(90-\th/6+2*\th:2.5*\r) -- (90-\th/6+2*\th:3*\r) 
	(-\th/6+3*\th:1.5*\r) -- (-\th/6+3*\th:2*\r) 
	(\th/6+3*\th:1.5*\r) -- (\th/6+3*\th:2*\r)  
	(0.5*\th+3*\th:1.5*\r) -- (0.5*\th+3*\th:2*\r) 
	(0.5*\th+3*\th:2*\r) -- (\th/12+3*\th:2.5*\r) 
	(-\th/6-1*\th:1.5*\r) -- (-\th/6-1*\th:2*\r) 
	(\th/6-1*\th:1.5*\r) -- (\th/6-1*\th:2*\r)  
	(\th/6-\th:2*\r) -- (0.5*\th-\th:2*\r) 
	(42-1*\th:2.5*\r) -- (\th/12-1*\th:2.5*\r)  
	(42+3*\th:2.5*\r) -- (90-\th/6+3*\th:3*\r) 
;

\draw[]
	(90+2*\th:0.5*\r) -- (90+2.5*\th:1*\r) 
	(42+2*\th:2.5*\r) -- (90-\th/6+2*\th:3*\r)
;

\foreach \a in {0,1,2,3} {

\draw[rotate=\a*\th]
	(90-2*\th:0.5*\r) -- (90-2*\th:1*\r)
	(90-\th:0.5*\r) -- (90-1.5*\th:1*\r)
	(90-\th/6-2*\th:2.5*\r) -- (90-\th/6-2*\th:3*\r)
	(42-2*\th:2.5*\r) -- (\th/12-2*\th:3*\r) 
;


}

\foreach \a in {0,1,2} {

\draw[rotate=\a*\th]
	(90-\th:1*\r) -- (5*\th/6-\th:1.5*\r)
	(90-\th/12-\th:1.5*\r) --  (5*\th/6-\th:2*\r) 
	(0.5*\th-\th:1.5*\r) -- (0.5*\th-\th:2*\r) 
	(0.5*\th:1.5*\r) -- (\th/6:2*\r) 
	(-\th/6:1.5*\r) -- (-\th/6:2*\r) 
	(0.5*\th-\th:2*\r) -- (\th/12-\th:2.5*\r)
;







}

\foreach \a in {0,1} {

\draw[rotate=\a*\th]
	(90+2*\th:1*\r) -- (90+1.5*\th:1*\r) 
	%
	%
	(90-\th/12+2*\th:1.5*\r) --  (5*\th/6+2*\th:2*\r)
	(0.5*\th+3*\th:1.5*\r) -- (\th/6+3*\th:2*\r) 
	(\th/6+3*\th:1.5*\r) -- (-\th/6+3*\th:1.5*\r)
	(5*\th/6+2*\th:2*\r) -- (0.5*\th+2*\th:1.5*\r)
	(\th/6+2*\th:2*\r) -- (0.5*\th+2*\th:2*\r) 
	(42+2*\th:2.5*\r) -- (\th/12+2*\th:2.5*\r) 
;








%
}






\foreach \a in {0,...,4}
{

\draw[rotate=\a*\th]
	(90:0.5*\r) -- (90-\th:0.5*\r)  
	(90-0.5*\th:1*\r) -- (90-\th:1*\r) 
	(\th/6:1.5*\r) -- (0.5*\th:1.5*\r) 
	(0.5*\th:1.5*\r) -- (5*\th/6:1.5*\r)  
	(90-0.5*\th:1*\r) -- (5*\th/6:1.5*\r) 
	(90:1*\r) -- (90-\th/12:1.5*\r) 
	%
	%
	(-\th/6:2*\r) -- (\th/6:2*\r) 
	(0.5*\th:2*\r) -- (5*\th/6:2*\r) 
	(90-\th/6:2.5*\r) -- (42:2.5*\r) 
	(0.5*\th:2*\r) -- (42:2.5*\r) 
	(\th/6:2*\r) -- (\th/12:2.5*\r) 
	%
	%
	%
	%
;


}

\draw
	(0,0) circle (3*\r)
;

\draw[-stealth, white]
	(90:0.5*\r) to[out=10, in=120] ($(5*\th/6-\th:1.5*\r) !1/2! (90-1.5*\th:1*\r)$) 
;

\fill 
	(90:0.5*\r) circle (\rr)
	(90-\th:0.5*\r) circle (\rr)
	(90+\th:0.5*\r) circle (\rr)
	(90-2*\th:1*\r) circle (\rr)
	(0.5*\th-1*\th:1.5*\r) circle (\rr)
	(90-\th/12-2*\th:1.5*\r) circle (\rr)
	(-\th/6+2*\th:2*\r) circle (\rr)
	(\th/6+2*\th:2*\r) circle (\rr) 
	(\th/12+2*\th:2.5*\r) circle (\rr)
	(\th/12:3*\r) circle (\rr)
	(\th/12+\th:3*\r) circle (\rr)
	(\th/12-\th:3*\r) circle (\rr)
;

\end{scope} 

\end{tikzpicture}
\caption{Difference between two triangular fusions of the snub dodecahedron}
\label{Fig-a5-a4-diff-snub-dodecahedron}
\end{figure}

It is interesting to note that, the three triangular fusions of the snub dodecahedron can also be derived from the $1$-factors of the graph of dodecahedron. The explanation of such connection can be seen at the end of the proof of Proposition \ref{Prop-a5-a4-albe2}.

The fourth tiling in Figure \ref{Fig-a5-a4-Archimedean-Tilings} can be derived from a truncated icosahedron. Each hexagon in Figure \ref{Fig-a5-a4-truncated-icosahedron} is divided into three kites (by red dotted lines), and form a one parameter family of tiling with regular pentagon and kite. For one special value, the kite becomes rhombus.

\begin{figure}[h!]
\centering
\begin{tikzpicture}[>=latex,scale=1]

\tikzmath{
\r=1;
\th=360/5;
}

\fill[gray!50] (0,0) circle (2.55*\r);
	
\fill[white] (0,0) circle (2.4*\r);

\fill[gray!50]
	(90:0.4*\r) -- (90+\th:0.4*\r) -- (90+2*\th:0.4*\r) -- (90-2*\th:0.4*\r) -- (90-\th:0.4*\r) -- cycle
;

\foreach \a in {0,...,4}{

\fill[gray!50, rotate=\a*\th]
	(90:0.8*\r) -- (66:1.1*\r) -- (78:1.4*\r) -- (102:1.4*\r) -- (114:1.1*\r) -- cycle
;

\fill[gray!50, rotate=\a*\th]
	(66:1.6*\r) -- (78:2*\r) -- (54:2*\r) -- (30:2*\r) -- (42:1.6*\r)		
;

}


	(0,0) -- (18:0.4) -- (90:0.4)
	(66:1.1) -- (78:1.4) -- (102:1.4) -- (114:1.1) -- (90:0.8)
	(66:1.6) -- (78:2) -- (54:2) -- (30:2) -- (42:1.6)		
	(54:2.4) arc (54:126:2.4) -- (126:2.6) arc (126:54:2.6);

\foreach \a in {0,...,4} {

\tikzset{rotate=72*\a}

\draw[red, densely dotted, line width=0.6]
	(18:0.4*\r) -- (54:0.75*\r) -- (90:0.8*\r)
	(54:0.75*\r) -- (42:1.1*\r) 
	(66:1.1*\r) -- (54:1.32*\r) -- (30:1.4*\r)
	(54:1.32*\r) -- (66:1.6*\r)
	(114:1.6*\r) -- (90:1.7*\r) -- (78:2*\r)
	(90:1.7*\r) -- (78:1.4*\r)
	(90:2.175*\r) -- (102:2*\r)
;

\arcThroughThreePoints[red, densely dotted, line width=0.6]{90:2.2*\r}{118:2.3*\r}{126:2.4*\r};
\arcThroughThreePoints[red, densely dotted, line width=0.6]{54:2*\r}{72:2.1*\r}{90:2.2*\r};

}

\foreach \a in {0,...,4}
{
\tikzset{rotate=72*\a}

\draw
	(18:0.4*\r) -- (90:0.4*\r) -- (90:0.8*\r) -- (66:1.1*\r) -- (42:1.1*\r) -- (18:0.8*\r)
	(66:1.1*\r) -- (78:1.4*\r) -- (102:1.4*\r) -- (114:1.1*\r)
	(78:1.4*\r) -- (66:1.6*\r) -- (42:1.6*\r) -- (30:1.4*\r)
	(66:1.6*\r) -- (78:2*\r) -- (102:2*\r) -- (114:1.6*\r)
	(78:2*\r) -- (54:2*\r) -- (30:2*\r)
	(54:2*\r) -- (54:2.4*\r)
;

	(90:0.4*\r) -- (54:0.75*\r) -- (66:1.1*\r)
	(54:0.75*\r) -- (18:0.8*\r)
	(42:1.1*\r) -- (54:1.32*\r) -- (78:1.4*\r)
	(54:1.32*\r) -- (42:1.6*\r)
	(66:1.6*\r) -- (90:1.7*\r) -- (102:2*\r)
	(90:1.7*\r) -- (102:1.4*\r)
	(78:2*\r) -- (90:2.2*\r)
	(54:2.4*\r) to[out=165,in=10] 
	(90:2.2*\r) to[out=190,in=50] (126:2*\r);

}

\draw[] (0,0) circle (2.4*\r);
	
\end{tikzpicture}
\caption{A quadrilateral subdivion of the truncated icosahedron (football)}
\label{Fig-a5-a4-truncated-icosahedron}
\end{figure}
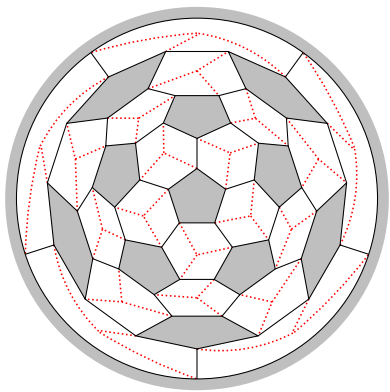

Throughout this paper, a pentagon refers to a regular pentagon and an $m$-gon refers to a regular $m$-gon with $m\ge6$ unless otherwise specified.

This paper is organised as follows. In Section \ref{Sec-basic}, we explain the basic terminology and tools. In Section \ref{Sec-a5-a4-tilings}, we classify the tilings by pentagons and rhombi. In Section \ref{Sec-a6+-a4-tilings}, we classify the tilings by $m$-gons with $m\ge6$ and rhombi. The key to classification is to find all the vertices in the tilings. The sporadic tilings are obtained in Propositions \ref{Prop-a5-a4-be3}, \ref{Prop-a5-a4-albe2}. The first infinite family is obtained in Proposition \ref{Prop-a5-a4-be2ga} and the second infinite family is obtained in Propositions \ref{Prop-a5-a4-albega}, \ref{Prop-am-a4-albega}. 

\section{Basics} \label{Sec-basic}

The {\em vertex angle sum} of a vertex $\alpha^a\beta^b\gamma^c$, consisting of $a$ copies of $\alpha$ and $b$ copies of $\beta$ and $c$ copies of $\gamma$, is
\begin{align}\label{rational-VertexAngSum}
a \alpha + b \beta + c \gamma = 2\pi.
\end{align} 
In a vertex notation, $a,b,c$ are assumed to be $>0$ unless otherwise specified. That is, we only express the angles appearing at a vertex whenever possible. For example, $\alpha\beta^2$ is a vertex with $a=1, b=2$ and $c=0$. The notation $\alpha\beta^2\cdots$ means a vertex with at least one $\alpha$ and two $\beta$'s, i.e., $a\ge1$ and $b\ge2$. The angle combination in $\cdots$ is called the {\em remainder} of the vertex. The value of the remainder is denoted by $R$. For example, $R(\alpha\beta^2) = 2\pi - \alpha - 2\beta$.

There are various constraints on the angle combinations at vertices in a tiling. Examples of such constraints are the vertex angle sum and the quadrilateral angle sum. A collection of all vertices in a tiling satisfying various constraints is called an {\em anglewise vertex combination} ($\AVC$). The following $\AVC$ is from \eqref{Eq-AVC-be2ga-begac-albegac},
\begin{align*}
\AVC = \{ \beta^2\gamma, \beta\gamma^c, \alpha\beta\gamma^c \}.
\end{align*}
The generic $c$ may take different values at different vertex. The tilings constructed from \eqref{Eq-AVC-be2ga-begac-albegac} do not have $\beta\gamma^c$. We use \quotes{$\equiv$} in place of \quotes{$=$} to denote the set of all vertices which actually appear in a tiling. For example, we have the following for the tilings in Figure \ref{Fig-a5-a4-Tilings-be2ga-albegac}
\begin{align*}
\AVC \equiv \{ \beta^2\gamma, \alpha\beta\gamma^c \}.
\end{align*}

To obtain the vertices, it is necessary and convenient to have notations for studying various angle arrangements. For example, $\alpha_1\gamma_2\cdots$ denotes the vertex where $T_1$ contributes $\alpha$ and $T_2$ contributes $\gamma$ in the first picture of Figure \ref{Fig-adj-am-a4}. To emphasize $\alpha_1$ is adjacent to $\gamma_2$ along an edge \quotes{ $\vert$ }, we use $\alpha_1 \vert \gamma_2 \cdots$ to denote the vertex. In addition, the same picture shows that $\alpha\vert\gamma\cdots$ is a vertex if and only if $\alpha\vert\beta\cdots$ is a vertex. Similarly, $T_1, T_2$ in the second picture of Figure \ref{Fig-adj-am-a4} show that $\gamma\vert\gamma\cdots$ is a vertex if and only if $\beta\vert\beta\cdots$ is also a vertex. For a full vertex, such as $\alpha^3$ in the third picture, we use $\vert \alpha \vert \alpha \vert \alpha \vert$ to denote its angle arrangement.

\begin{figure}[h!] 
\centering
\begin{tikzpicture}

\tikzmath{
\s=1;
\r=1;
\g=5;
\ph=360/\g;
\x=\r*cos(\ph/2);
\y=\r*sin(\ph/2);
\rr=2*\y/sqrt(2);
\h=4;
\th=360/\h;
\xx=\rr*cos(\th/2);
}

\begin{scope}

\foreach \a in {0,1,4} {

\draw[rotate=\a*\ph]
	(0.5*\ph:\r) -- (-0.5*\ph:\r)
;

}

\foreach \a in {0,...,3} {

\tikzset{shift={(\x+\xx,0)}}

\draw[rotate=\a*\th]
	(1.5*\th:\rr) -- (0.5*\th:\rr)
;

}

\foreach \a in {0,1} {

\tikzset{shift={(\x+\xx,0)}}

\node at (1.5*\th+\a*2*\th:0.65*\rr) {\small $\beta$};
\node at (0.5*\th+\a*2*\th:0.65*\rr) {\small $\gamma$};

}

\foreach \a in {0,4} {

\node at (0.5*\ph+\a*\ph:0.75*\r) {\small $\alpha$};

}

\node at (175:0.45*\r) {$\vdots$};

\node[inner sep=1,draw,shape=circle] at (0,0) {\small $1$};
\node[inner sep=1,draw,shape=circle] at (\x+\xx,0) {\small $2$};

\end{scope}

\begin{scope}[xshift=4.5*\s cm] 

\foreach \aa in {-1,1} {

\tikzset{shift={(\aa*\xx,0)}, xscale=\aa}

\foreach \a in {0,...,3} {

\draw[rotate=\th*\a]
	(0.5*\th:\rr) -- (1.5*\th:\rr)
;

}

\foreach \a in {0,2} {

\node at (1.5*\th+\th*\a: 0.625*\rr) {\small $\beta$}; 
\node at (0.5*\th+\th*\a: 0.65*\rr) {\small $\gamma$}; 

}

}

\node[inner sep=1,draw,shape=circle] at (-\xx,0) {\small $1$};
\node[inner sep=1,draw,shape=circle] at (\xx,0) {\small $2$};

\end{scope} 

\begin{scope}[xshift=7.5*\s cm, yshift=-0.15*\s cm] 

\foreach \a in {0,1,2} {

\draw[rotate=120*\a]
	(0:0) -- (90:\rr) 
;

\node at (270+\a*120:0.3*\rr) {\small $\alpha$};

}

\end{scope} 

\end{tikzpicture}
\caption{The arrangements of $\alpha \vert \gamma$ and $\gamma\vert\gamma$ and $\beta\vert\beta$ and $\alpha^3$}
\label{Fig-adj-am-a4}
\end{figure}
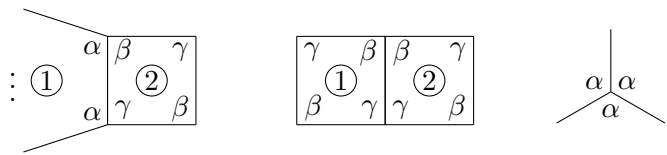


The following lemma follows from a well-known fact on triangle-free graphs on the sphere with vertex degree $\ge3$.

\begin{lem}\label{Lem-deg3} There is a degree $3$ vertex in a dihedral tiling by quadrilaterals and $m$-gons where $m\ge5$.
\end{lem}

\begin{lem}[Counting Lemma, {\cite[Lemma 2.3]{luk2}}] \label{Lem-Counting} In a dihedral tiling of the sphere by regular $m$-gons and rhombi, If at every vertex the number of $\beta$ is no more than the number of $\gamma$, then at every vertex these two numbers are equal.
\end{lem}

The angle sum of the regular $m$-gon is $m\alpha > (m-2)\pi$ and the rhombus angle sum is $2\beta+2\gamma > 2\pi$. Then we get
\begin{align*}
\alpha > (1 - \tfrac{2}{m})\pi, \quad
\beta+\gamma > \pi.
\end{align*}
Up to symmetry between $\beta, \gamma$, we may assume $\beta>\gamma$. This implies $\beta>\frac{1}{2}\pi$.

In a dihedral tiling by regular $m$-gons and rhombi, there is a pair of adjacent tiles consisted of one $m$-gon and one rhombus in the first picture of Figure \ref{Fig-adj-am-a4}. The next lemma follows from this observation.

\begin{lem}\label{Lem-albe-alga} In a dihedral tiling by regular $m$-gons and rhombi, $\alpha\beta\cdots, \alpha\gamma\cdots$ are vertices. 
\end{lem}

An immediate consequence of Lemma \ref{Lem-albe-alga} is that one of $\alpha^a\beta^b, \alpha\beta\gamma\cdots$ is a vertex. Then we have one of the inequalities, $2\alpha + \beta \le 2\pi$ and $\alpha+2\beta \le 2\pi$ and $\alpha+\beta+\gamma\le2\pi$. Combined with $\alpha>(1 - \tfrac{2}{m})\pi$ and $\beta>\gamma$ and $\beta+\gamma>\pi$, we know $\alpha, \gamma<\pi$. Hence the regular $m$-gon is convex. 

By \cite[Lemma 5.2]{cl} or \cite[Lemma 18]{cly}, we have 
\begin{align}
\cos x = \cot \tfrac{1}{2}\beta \cot \tfrac{1}{2}\gamma.
\end{align}
On the other hand, the spherical cosine law for angles on the regular $m$-gon implies
\begin{align}
\cos x = \cot^2 \tfrac{1}{2}\alpha + \frac{\cos \frac{2}{m}\pi}{ \sin^2 \frac{1}{2}\alpha }. 
\end{align}
Combining the two identities, we get
\begin{align} \label{Eq-cot-al-be-ga}
\cot^2 \tfrac{1}{2}\alpha + \cos \tfrac{2}{m}\pi \csc^2 \tfrac{1}{2}\alpha = \cot \tfrac{1}{2}\beta \cot \tfrac{1}{2}\gamma,
\end{align}
Then for $m\ge5$ and $\gamma \in (0,\pi)$, we get $\cot\frac{1}{2}\beta >0$, which implies $\beta<\pi$. Meanwhile, by $m\ge5$ and $\beta>\gamma$ and \eqref{Eq-cot-al-be-ga}, we get $\cot  \tfrac{1}{2}\beta \cot \tfrac{1}{2}\gamma > \cot^2 \frac{1}{2}\alpha$. By $\alpha, \beta, \gamma \in (0,\pi)$ and $\cot \theta$ is strictly decreasing on $(0, \pi)$, this implies $\alpha>\gamma$. So $\gamma$ is the smallest angle. Hence we have $\alpha,\beta>\gamma$ throughout this paper. 



By Lemma \ref{Lem-deg3}, there is a degree $3$ vertex. By $\alpha\beta\cdots$, we have one of the inequalities $2\alpha+\beta \le 2\pi$ and $\alpha+2\beta \le 2\pi$ and $\alpha+\beta+\gamma\le 2\pi$. Then $\alpha, \beta > \gamma$ imply
\begin{align}\label{Eq-al+be+ga<=2pi}
\alpha+\beta+\gamma\le 2\pi.
\end{align}
It further implies that $\gamma^3, \alpha\gamma^2, \beta\gamma^2$ are not vertices. For $\alpha^3, \beta^3, \alpha^2\beta, \alpha\beta^2$, the inequality $\beta \ge \alpha$ implies $\beta \ge \frac{2}{3}\pi \ge \alpha$. Combined with $\alpha>(1 - \tfrac{2}{m})\pi$, we get $m<6$. This means that $\alpha^3, \beta^3, \alpha^2\beta, \alpha\beta^2$ are not vertices for $m$-gons with $m\ge6$. Hence one of the following degree $3$ vertices must appear in a dihedral tiling for the given gonality $m\ge5$,
\begin{align}
\label{Eq-a5-a4-deg3-list}
&m=5: \quad
\alpha^3, \alpha^2\gamma, \beta^3, \alpha^2\beta, \alpha\beta^2, \beta^2\gamma, \alpha\beta\gamma; \\
\label{Eq-am-a4-deg3-list}
&m \ge 6: \quad \alpha^2\gamma,
\beta^2\gamma, \alpha\beta\gamma.
\end{align}

\section{Tilings by Rhombi and Pentagons} \label{Sec-a5-a4-tilings}

\begin{prop}\label{Prop-a5-a4-albega} The dihedral tiling with vertex $\alpha\beta\gamma$ is in Figure \ref{Fig-a5-a4-Tiling-albega}.
\end{prop}

The tiling is given by the prism with $2$ pentagons and $5$ rhombi. The tiling is the first one in Figure \ref{Fig-am-a4-Tilings-prism}.


We remark that $\alpha\ge\beta>\gamma$ implies $\alpha\beta\cdots=\alpha\beta\gamma$ for the following reason. If $\alpha\ge\beta$, then by $\beta+\gamma>\pi$ we get $R(\alpha\beta) \le R(\beta^2)<2\gamma$. By $\alpha,\beta>\gamma$, we have $\alpha\beta\cdots=\alpha\beta\gamma$. Since Lemma \ref{Lem-albe-alga} implies that $\alpha\beta\cdots$ is a vertex, for the studies where $\alpha\beta\gamma$ is not a vertex, it suffices to consider $\beta > \alpha > \gamma$. 

\begin{proof} By $\alpha\beta\gamma$, we have $R(\alpha\beta)=\gamma$. Then by $\alpha, \beta>\gamma$, we get $\alpha\beta\cdots=\alpha\beta\gamma$. Meanwhile, by $\beta+\gamma > \pi$, we also have $R(\beta^2)<2\gamma$. Then we get $\beta^2\cdots=\beta^3, \beta^2\gamma$. 

By $\alpha\beta\cdots=\alpha\beta\gamma$ and $\beta^2\cdots=\beta^3, \beta^2\gamma$, we get $\beta\cdots=\alpha\beta\gamma, \beta^3, \beta^2\gamma, \beta\gamma^c$. The other vertices consist of only $\alpha, \gamma$, which are $\alpha^a, \alpha^a\gamma^c, \gamma^c$. By $\alpha>\frac{3}{5}\pi$, we further get $\alpha^a=\alpha^3$. Hence we have the vertices below
\begin{align*}
\alpha^3, \beta^3, \beta^2\gamma, \alpha\beta\gamma, \gamma^c, \alpha^a\gamma^c, \beta\gamma^c.
\end{align*}

The arrangement $\alpha\vert\gamma$ determines tiles $T_1, T_2$ in the first picture of Figure \ref{Fig-a5-a4-Tiling-albega}. Then $\alpha_1\beta_2\cdots=\alpha\beta\gamma$ determines $T_3$. By the same argument, we further determine $T_4, T_5, T_6$, which implies $\alpha\vert\gamma\cdots=\alpha\beta\gamma$. Then $\alpha^a\gamma^c$ is not a vertex and $\alpha^2\cdots=\alpha^3$.

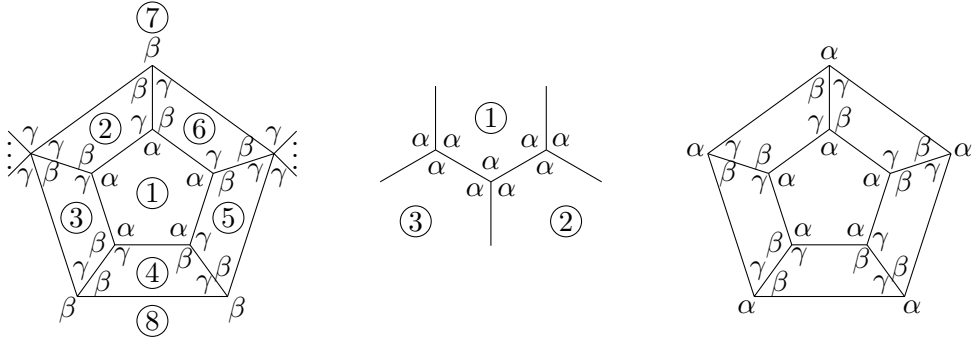
\begin{figure}[h!] 
\centering
\begin{tikzpicture}

\tikzmath{
\s=1;
\r=0.85;
\g=5;
\ph=360/\g;
\ps=360/6;
}

\begin{scope}[]

\foreach \p in {0,...,4} {

\draw[rotate=\p*\ph]
	(90-1*\ph:\r) -- (90:\r)
	(90:\r) -- (90:2*\r)
	(90-1*\ph:2*\r) -- (90:2*\r)
;

\node at (90+\p*\ph:0.7*\r) {\small $\alpha$};

\node at (90-0.15*\ph+\p*\ph:1.175*\r) {\small $\beta$};
\node at (90+0.15*\ph+\p*\ph:1.1*\r) {\small $\gamma$};

\node at (90-0.09*\ph+\p*\ph:1.65*\r) {\small $\gamma$};
\node at (90+0.1*\ph+\p*\ph:1.615*\r) {\small $\beta$};


}

\foreach \aa in {-1, 1} {

\draw[xscale=\aa]
	(90-\ph:2*\r) -- ([shift={(90-\ph:2*\r)}]45:0.5*\r)
	(90-\ph:2*\r) -- ([shift={(90-\ph:2*\r)}]-45:0.5*\r)
;

\node at (90- \aa*0.9*\ph:2.1*\r) {\small $\gamma$};
\node at (90- \aa*1.15*\ph:2*\r) {\small $\gamma$};

\node at (90-\aa*\ph:2.35*\r) {\small $\vdots$};

}

\node at (90:2.25*\r) {\small $\beta$};
\node at (90-2*\ph:2.25*\r) {\small $\beta$};
\node at (90+2*\ph:2.25*\r) {\small $\beta$};

\node[inner sep=1,draw,shape=circle] at (0,0) {\small $1$};
\node[inner sep=1,draw,shape=circle] at (90+0.5*\ph:1.25*\r) {\small $2$};
\node[inner sep=1,draw,shape=circle] at (90+1.5*\ph:1.25*\r) {\small $3$};
\node[inner sep=1,draw,shape=circle] at (90+2.5*\ph:1.25*\r) {\small $4$};
\node[inner sep=1,draw,shape=circle] at (90+3.5*\ph:1.25*\r) {\small $5$};
\node[inner sep=1,draw,shape=circle] at (90+4.5*\ph:1.25*\r) {\small $6$};

\node[inner sep=1,draw,shape=circle] at (90:2.75*\r) {\small $7$};
\node[inner sep=1,draw,shape=circle] at (90+2.5*\ph:2*\r) {\small $8$};

\end{scope}

\begin{scope}[xshift=4.5*\s cm, yshift=1*\s cm] 

\foreach \p in {2,3,4,5} {

\draw[rotate=\p*\ps]
	(90-1*\ps:\r) -- (90:\r)
;

}

\foreach \p in {-1,0,1} {

\draw[rotate=\p*\ps]
	(270:\r) -- (270:2*\r)
;

\node at (282.5+\p*\ps:1.15*\r) {\small $\alpha$};
\node at (257.5+\p*\ps:1.15*\r) {\small $\alpha$};

}

\node at (270:0.7*\r) {\small $\alpha$};
\node at (270-\ps:0.7*\r) {\small $\alpha$};
\node at (270+\ps:0.7*\r) {\small $\alpha$};

\node[inner sep=1,draw,shape=circle] at (0,0) {\small $1$};
\node[inner sep=1,draw,shape=circle] at (270+0.5*\ph:2*\r) {\small $2$};
\node[inner sep=1,draw,shape=circle] at (270-0.5*\ph:2*\r) {\small $3$};

\end{scope} 

\begin{scope}[xshift=9*\s cm] 

\foreach \p in {0,...,4} {

\draw[rotate=\p*\ph]
	(90-1*\ph:\r) -- (90:\r)
	(90:\r) -- (90:2*\r)
	(90-1*\ph:2*\r) -- (90:2*\r)
;

\node at (90+\p*\ph:0.7*\r) {\small $\alpha$};

\node at (90-0.15*\ph+\p*\ph:1.175*\r) {\small $\beta$};
\node at (90+0.15*\ph+\p*\ph:1.1*\r) {\small $\gamma$};

\node at (90-0.09*\ph+\p*\ph:1.65*\r) {\small $\gamma$};
\node at (90+0.1*\ph+\p*\ph:1.615*\r) {\small $\beta$};

\node at (90+\p*\ph:2.2*\r) {\small $\alpha$};

}

\end{scope}

\end{tikzpicture}
\caption{Construction of the tiling with $\alpha\beta\gamma$}
\label{Fig-a5-a4-Tiling-albega}
\end{figure}

By $\alpha^2\cdots=\alpha^3$, starting at an $\alpha^3$, we get tiles $T_1, T_2, T_3$ in the second picture of Figure \ref{Fig-a5-a4-Tiling-albega} and the starting vertex is $\alpha_1\alpha_2\alpha_3$. Then by $\alpha^2\cdots=\alpha^3$, we also know the two adjacent vertices $\alpha_1\alpha_2\cdots, \alpha_1\alpha_3\cdots = \alpha^3$. Repeating this process, we determine a monohedral tiling given by the dodecahedron. 

So $\alpha^3$ is not a vertex of dihedral tiling and hence it remains to discuss the following,
\begin{align*}
\beta^3, \beta^2\gamma, \alpha\beta\gamma, \gamma^c, \beta\gamma^c.
\end{align*}
Since $\alpha$ appears at some vertex, we know that $\alpha\beta\gamma$ is a vertex and $\alpha\beta\cdots=\alpha\gamma\cdots=\alpha\beta\gamma$.

Starting at an $\alpha\beta\gamma$, by $\alpha\beta\cdots=\alpha\gamma\cdots=\alpha\beta\gamma$ we determine $T_1, T_2, ..., T_6$ in the first picture of Figure \ref{Fig-a5-a4-Tiling-albega}. If the undetermined $\beta_2\gamma_6\cdots=\beta^2\gamma$, then we determine $T_7$. By $\beta_3\gamma_2\gamma_7\cdots, \beta_6\gamma_5\gamma_7\cdots=\beta\gamma^c$, we further determine $T_8$ with two adjacent $\beta$'s, a contradiction. This means that none of the undetermined $\beta\vert\gamma\cdots$ at $T_2, ..., T_6$ is $\beta^2\gamma$. Hence they must be $\alpha\beta\gamma$. This determines a tiling in the third picture of Figure \ref{Fig-a5-a4-Tiling-albega} with
\begin{align*}
\AVC \equiv \{ \alpha\beta\gamma \}.
\end{align*}
We remark that $\beta^3, \beta^2\gamma, \gamma^c, \beta\gamma^c$ are not vertices for the dihedral tiling. 

\subsubsection*{Geometric Realisation}

It remains to show that the tiling exists. It suffices to show that for the two regular pentagons with edge length $x$ placed equal distance (denoted by $h$) away from the equator in the north and south hemispheres respectively, one can subdivide the cylinder into five rhombi with length $x$.

Consider a lune defined by the north pole $N$ and the south pole $S$ and $A_1, A_2, B_1,$ $B_2$ where the arcs $A_1A_2, B_1B_2$ have lengths $x$ and $A_1, A_2$ (and respectively $B_1, B_2$) are equal distance from the equator (dotted line) in Figure \ref{Fig-a5-a4-Geom-albega-Lune}, i.e., $NA_1 = NA_2=SB_1=SB_2=r$ for some $r \in (0,\frac{1}{2}\pi)$. Then we have $h=\pi-2r$.

\begin{figure}[h!] 
\centering
\begin{tikzpicture}[>=latex,scale=1]

\begin{scope}[xshift=0cm] 

\tikzmath{
\r=2; \R=sqrt(3);
\a=72; \aa=\a/2;
\xP=0; \yP=\R;
\aOP=60; 
\aPOne=\aOP;
\mPOne=tan(90+\aPOne);
\aPTwo=\aOP;
\mPTwo=tan(-90-\aPTwo);
\xCQ=-1.15; \yCQ=0.8;
\xCOne = 1; \yCOne=0;
\xCTwo = -1; \yCTwo=0;
\b=120; \bb=\b/2;
\aOPP=120; 
\aPThree=\aOPP;
\mPThree=tan(90+\aPThree);
\aPFour=\aOPP;
\mPFour=tan(-90-\aPFour);
\xCQQ=-4.5; \yCQQ=-0.2;
}

\coordinate (O) at (0,0);
\coordinate (C1) at (1,0);
\coordinate (C2) at (-1,0);


	(O) circle ({sqrt(3)});
	
	(C2) circle (\r)
	(C1) circle (\r);

\draw[gray!50]
	([shift={(-60:2)}]-1,0) arc (-60:60:2);

\draw[gray!50]
	([shift={(-60+180:2)}]1,0) arc (-60+180:60+180:2);

\pgfmathsetmacro{\xPOne}{ ( 2*((\mPOne)*(\yCOne)+(\xCOne)) - sqrt( ( 2*( (\mPOne)*(\yCOne)+\xCOne )  )^2 - 4*( (\mPOne)^2 + 1 )*( (\xCOne)^2 + (\yCOne)^2 - \r^2 ) ) )/( 2*( (\mPOne)^2+1 ) ) };
\pgfmathsetmacro{\yPOne}{ \mPOne*\xPOne };

\pgfmathsetmacro{\xPTwo}{ ( 2*((\mPTwo)*(\yCTwo)+(\xCTwo) ) + sqrt( ( 2*( (\mPTwo)*\yCTwo+\xCTwo )  )^2 - 4*( (\mPTwo)^2 + 1 )*( (\xCTwo)^2 + (\yCTwo)^2 - \r^2 ) ) )/( 2*( (\mPTwo)^2+1 ) ) };
\pgfmathsetmacro{\yPTwo}{ \mPTwo*\xPTwo };

\pgfmathsetmacro{\xPPOne}{ ( 2*((\mPOne)*(\yCOne)+(\xCOne)) + sqrt( ( 2*( (\mPOne)*(\yCOne)+\xCOne )  )^2 - 4*( (\mPOne)^2 + 1 )*( (\xCOne)^2 + (\yCOne)^2 - \r^2 ) ) )/( 2*( (\mPOne)^2+1 ) ) };
\pgfmathsetmacro{\yPPOne}{ \mPOne*\xPPOne };

\pgfmathsetmacro{\xPPTwo}{ ( 2*((\mPTwo)*(\yCTwo)+(\xCTwo) ) - sqrt( ( 2*( (\mPTwo)*\yCTwo+\xCTwo )  )^2 - 4*( (\mPTwo)^2 + 1 )*( (\xCTwo)^2 + (\yCTwo)^2 - \r^2 ) ) )/( 2*( (\mPTwo)^2+1 ) ) };
\pgfmathsetmacro{\yPPTwo}{ \mPTwo*\xPPTwo };

\pgfmathsetmacro{\xPQRef}{\xPTwo};
\pgfmathsetmacro{\yPQRef}{\yPTwo};

\pgfmathsetmacro{\dPRefP}{ sqrt( (\xPQRef - \xP)^2 + (\yPQRef - \yP)^2 ) };
\pgfmathsetmacro{\aCPRef}{ acos( (2*\r^2 - (\dPRefP)^2 )/(2*\r^2) ) };
\pgfmathsetmacro{\l}{ \aCPRef/\a };
\pgfmathsetmacro{\rQ}{ sqrt( \R^2 + (\xCQ)^2  )  };
\pgfmathsetmacro{\aPCQ}{ acos( (\xCQ)/(\rQ) ) }
\pgfmathsetmacro{\aQ}{ -( 360 -  2*\aPCQ )*(1-\l) };


\pgfmathsetmacro{\xPThree}{ ( 2*((\mPThree)*(\yCOne)+(\xCOne)) - sqrt( ( 2*( (\mPThree)*(\yCOne)+\xCOne )  )^2 - 4*( (\mPThree)^2 + 1 )*( (\xCOne)^2 + (\yCOne)^2 - \r^2 ) ) )/( 2*( (\mPThree)^2+1 ) ) };
\pgfmathsetmacro{\yPThree}{ \mPThree*\xPThree};

\pgfmathsetmacro{\xPFour}{ ( 2*((\mPFour)*(\yCTwo)+(\xCTwo) ) + sqrt( ( 2*( (\mPFour)*\yCTwo+\xCTwo )  )^2 - 4*( (\mPFour)^2 + 1 )*( (\xCTwo)^2 + (\yCTwo)^2 - \r^2 ) ) )/( 2*( (\mPFour)^2+1 ) ) };
\pgfmathsetmacro{\yPFour}{ \mPFour*\xPFour};

\pgfmathsetmacro{\xPPThree}{ ( 2*((\mPThree)*(\yCOne)+(\xCOne)) + sqrt( ( 2*( (\mPThree)*(\yCOne)+\xCOne )  )^2 - 4*( (\mPThree)^2 + 1 )*( (\xCOne)^2 + (\yCOne)^2 - \r^2 ) ) )/( 2*( (\mPThree)^2+1 ) ) };
\pgfmathsetmacro{\yPPThree}{ \mPThree*\xPPThree };

\pgfmathsetmacro{\xPPFour}{ ( 2*((\mPFour)*(\yCTwo)+(\xCTwo) ) - sqrt( ( 2*( (\mPFour)*\yCTwo+\xCTwo )  )^2 - 4*( (\mPFour)^2 + 1 )*( (\xCTwo)^2 + (\yCTwo)^2 - \r^2 ) ) )/( 2*( (\mPFour)^2+1 ) ) };
\pgfmathsetmacro{\yPPFour}{ \mPFour*\xPPFour };

\pgfmathsetmacro{\xPQRef}{\xPTwo};
\pgfmathsetmacro{\yPQRef}{\yPTwo};


\coordinate (CQ) at (\xCQ, \yCQ);


\coordinate (P) at (0,{sqrt(3)});

\coordinate[rotate around={\aQ:(CQ)}] (Q) at (P);


\coordinate (PP) at (0,{-sqrt(3)});

\coordinate (P1) at (\xPOne, \yPOne);
\coordinate (P2) at (\xPTwo, \yPTwo);

\coordinate (PP1) at (\xPPOne,\yPPOne);
\coordinate (PP2) at (\xPPTwo,\yPPTwo);

\coordinate (P3) at (\xPThree, \yPThree);
\coordinate (P4) at (\xPFour, \yPFour);

\coordinate (PP3) at (\xPPThree,\yPPThree);
\coordinate (PP4) at (\xPPFour,\yPPFour);

\coordinate (X) at (0.25*\xPTwo, \yPTwo);

\node at (P3) {$\cdot$};




	(P1) -- (PP1)
	(P2) -- (PP2)
;

\arcThroughThreePoints[dashed]{P2}{PP}{P};



\arcThroughThreePoints[]{P2}{PP1}{P1};

\arcThroughThreePoints[]{P3}{PP3}{P4};

\arcThroughThreePoints[dashed, gray!50]{X}{PP3}{P3};

\arcThroughThreePoints[dashed, gray!50]{X}{PP}{P};

\arcThroughThreePoints[dashed]{P1}{PP}{P3};

\draw[gray]
	(P2) -- (P3)
; 

\draw[gray!50, dashed]
	(P1) -- (P2)
;

\node [shape = circle, fill = black, minimum size = 0.12 cm, inner sep=0pt] at (P1) {};
\node [shape = circle, fill = black, minimum size = 0.12 cm, inner sep=0pt] at (P2) {};
\node [shape = circle, fill = black, minimum size = 0.12 cm, inner sep=0pt] at (P3) {};
\node [shape = circle, fill = black, minimum size = 0.12 cm, inner sep=0pt] at (P4) {};
\node [shape = circle, draw=gray, fill = white, minimum size = 0.12 cm, inner sep=0pt] at (X) {};

\draw[gray, dotted]
	(180:\R) -- (0:\R)
;

\draw[gray!50] ([shift=(90:\R)]212:0.15) arc (212:286:0.15)
;

\node at (90: 2.0) {\small $N$};
\node at (270: 2.0) {\small $S$};
\node at (155: 1.35) {\small $A_1$};
\node at (25: 1.35) {\small $A_2$};

\node at (205: 1.35) {\small $B_1$};
\node at (-25: 1.35) {\small $B_2$};

\node at ([shift={(Q)}]-0.4,-0.15) {\small \textcolor{gray}{$X$}};

\node at (60: 1.5) {\small $r$};
\node at (180: 1.2) {\small $h$};
\node at (270-30: 1.5) {\small \textcolor{gray}{$r$}};
\node at (75: 1.15) {\small \textcolor{gray}{$r$}};


\node at (95: 1.4) {\footnotesize \textcolor{gray}{$\xi$}};
\node at (125: 1.05) {\footnotesize $\theta'$};
\node at (155: 0.9) {\footnotesize $\theta$};
\node at (200: 0.9) {\footnotesize $\psi$};
\node at (42: 0.75) {\footnotesize $\varphi$};

\end{scope}

\end{tikzpicture}
\caption{Geometric realisation of tiles}
\label{Fig-a5-a4-Geom-albega-Lune}
\end{figure}
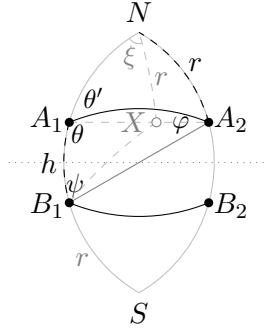

Let $\lambda (A_1A_2)$ denote segment of latitude defined by $A_1, A_2$. For $h<x$, we show that there is an $X_1 \in \lambda(A_1A_2)$ such that $B_1X_1= x$. By $r<\frac{1}{2}\pi$, we have $\theta' < \frac{1}{2}\pi$ and therefore $\theta > \frac{1}{2}\pi$. For $h<x<\pi$ and $A_2B_1<\pi$, the triangle $\triangle A_1A_2B_1$ is a standard triangle and hence $x>h$ implies $\psi>\phi$. Meanwhile, the arc $A_2B_1$ divides the lune into two identical halves. By the rotational symmetry of the subdivided lune, we know that $\psi < \frac{1}{2}\pi$. Then we have $\theta>\frac{1}{2}\pi>\psi>\phi$ and hence $A_2B_1 > x > h$. For $X \in \lambda(A_1, A_2)$, by $h=\pi-2r$ and cosine law on $\triangle B_1NX$, we get $\cos B_1X = \cos(h+r)\cos r + \sin(h+r) \sin r \cos \xi = - \cos^2 r + \sin^2 r \cos \xi$. For each fixed $r$, the length of $B_1X$ is a continuous function of $\xi$ such that $B_1X=h$ when $\xi=0$ and $B_1X = B_1A_2$ when $\xi=\frac{2}{m}\pi$ where $m=5$. Intermediate Value Theorem implies that there is a point $X_1 \in \lambda (A_1A_2)$ (given by $\xi_1 \in (0, \frac{2}{m}\pi)$) such that $B_1X_1=x$. By repeating the same process for each pair of vertices $B_i, B_{i+1}$ for the pentagon in the lower hemisphere, we locate the vertices $X_1, ..., X_5$ for the pentagon in the upper hemisphere and the rhombi are defined by edges $B_1X_1, B_2X_2, ..., B_5X_5$. Since the interiors of each lune are disjoint, the $B_iX_i$'s do not intersect. Therefore we can in fact obtain simple rhombi.

Note that the same construction works for all regular polygons with gonality $m\ge3$. Cosine law on $\triangle NA_1A_2$ gives $\cos x = \cos^2 r + \sin^2 r \cos \frac{2}{m}\pi$. For $\pi-2r = h < x$, it further implies $r > \cot^{-1} \sin \frac{1}{m}\pi$. Hence for each $m$, the construction works for any choice of $r \in (\cot^{-1} \sin \frac{\pi}{m}, \frac{1}{2}\pi)$.
\end{proof}

\begin{prop}\label{Prop-a5-a4-al2ga} There is no dihedral tiling with vertex $\alpha^2\gamma$.
\end{prop}

\begin{proof} By $\alpha>\gamma$ and $\alpha^2\gamma$, we get $\alpha>\frac{2}{3}\pi>\gamma$. Lemma \ref{Lem-albe-alga} implies that $\alpha\beta\cdots$ is a vertex. As a consequence of Proposition \ref{Prop-a5-a4-albega} (see the remark), we have $\beta > \alpha>\gamma$. Combined with $\alpha>\frac{2}{3}\pi$, we have $\alpha+\beta+\gamma> 2\alpha + \gamma = 2\pi$, contradicting \eqref{Eq-al+be+ga<=2pi}. 
\end{proof}

\begin{prop}\label{Prop-a5-a4-al3} There is no dihedral tiling with vertex $\alpha^3$.
\end{prop}

\begin{proof} If $\alpha^3$ is a vertex, then we have $\alpha = \frac{2}{3}\pi$ and $\alpha^2\beta$ is not a vertex. Combined with $\beta>\alpha>\gamma$ and $\beta+\gamma>\pi$, we have $R(\beta^2)<\alpha, 2\gamma$. So $\beta^2\cdots=\beta^2\gamma$. 

Suppose both $\alpha^3, \beta^2\gamma$ are vertices. We get
\begin{align*}
\alpha=\tfrac{2}{3}\pi, \quad
\beta = \pi - \tfrac{1}{2}\gamma. 
\end{align*}
Substituting the above into \eqref{Eq-cot-al-be-ga}, we get
\begin{align*}
\tan^2 \tfrac{1}{4}\gamma = 1 - \tfrac{2}{3}(4\cos \tfrac{2}{5}\pi + 1) < 0,
\end{align*}
a contradiction. Then $\beta^2\gamma$ is not a vertex. This further implies that $\beta^2\cdots$ is not a vertex.

By $\beta>\alpha>\gamma$ and $\alpha > \frac{3}{5}\pi$ and $\beta+\gamma>\pi$ and no $\beta^2\cdots$, we get $\beta\cdots=\beta\gamma^c, \alpha\beta\gamma^c$. Counting Lemma implies $\beta\cdots=\alpha\beta\gamma$, which is not a vertex, a contradiction. 
\end{proof}

\begin{prop}\label{Prop-a5-a4-be3} The dihedral tiling with vertex $\beta^3$ is the fourth tiling in Figure \ref{Fig-a5-a4-Archimedean-Tilings}.
\end{prop}

\begin{proof} By $\beta^3$, we get $\beta=\frac{2}{3}\pi$. Combined with $\alpha>\tfrac{3}{5}\pi$ and $\beta>\gamma$ and $\beta+\gamma>\pi$, we get the following,
\begin{align*}
\alpha> \tfrac{3}{5}\pi, \quad
\beta=\tfrac{2}{3}\pi, \quad
\gamma>\tfrac{1}{3}\pi.
\end{align*}

Lemma \ref{Lem-albe-alga} implies that $\alpha\beta\cdots$ is a vertex. The above inequalities imply $R(\alpha\beta)<\frac{11}{15}\pi<\alpha+\gamma, 3\gamma$. Meanwhile, by Proposition \ref{Prop-a5-a4-albega}, we may assume no $\alpha\beta\gamma$. Combined with $\beta>\alpha>\gamma$ and $\beta^3$, we get $\alpha\beta\cdots=\alpha\beta\gamma^2$. Then $\beta^3, \alpha\beta\gamma^2$ imply
\begin{align*}
\alpha + 2\gamma = \tfrac{4}{3}\pi, \quad
\beta = \tfrac{2}{3}\pi.
\end{align*}
Substituting the above into \eqref{Eq-cot-al-be-ga}, we get
\begin{align*}
\alpha \approx 0.61881\pi, \quad
\beta = \tfrac{2}{3}\pi, \quad
\gamma \approx 0.35726\pi, \quad
x \approx 0.12943\pi.
\end{align*}
The above angle values determine all vertices. Hence we get
\begin{align}\label{Eq-AVC-b3-albega2}
\AVC = \{ \beta^3, \alpha\beta\gamma^2 \}.
\end{align}
From the above, we have $\alpha\vert\gamma\cdots=\alpha\beta\gamma^2$. 

The arrangement of $\alpha\vert\gamma\cdots$ determines tiles $T_1, T_2$ in Figure \ref{Fig-a5-a4-AAD-albegac-alga}. Then $\alpha_1\beta\cdots=\alpha\beta\gamma^c$ determines $T_3$. Repeating the same argument we further determine $T_4,T_5, T_6$. Hence $\alpha\vert\gamma\cdots = \beta \vert \alpha \vert \gamma\cdots =\alpha\beta\gamma^2$. 

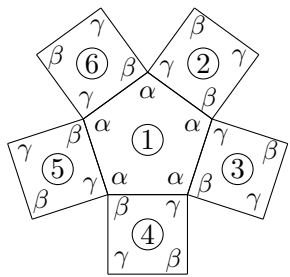
\begin{figure}[h!] 
\centering
\begin{tikzpicture}

\tikzmath{
\s=1;
\r=0.9;
\g=5;
\ph=360/\g;
\x=\r*cos(\ph/2);
\y=\r*sin(\ph/2);
\rr=2*\y/sqrt(2);
\h=4;
\th=360/\h;
\xx=\rr*cos(\th/2);
}

\begin{scope}

\foreach \p in {0,...,4} {

\draw[rotate=\p*\ph]
	(90-1*\ph:\r) -- (90:\r)
;

\foreach \q in {0,...,3} {

\tikzset{shift={(90-0.5*\ph+\p*\ph:\x+\xx)}}

\draw[rotate=-0.5*\ph+\p*\ph+\q*\th]
	(1.5*\th:\rr) -- (0.5*\th:\rr)
;

}

\foreach \q in {0,2} {

\tikzset{shift={(90-0.5*\ph+\p*\ph:\x+\xx)}}

\node at (1.5*\th+\q*\th-0.5*\ph+\p*\ph:0.65*\rr) {\footnotesize $\beta$};
\node at (0.5*\th+\q*\th-0.5*\ph+\p*\ph:0.65*\rr) {\footnotesize $\gamma$};

}

\node at (90+\p*\ph:0.7*\r) {\footnotesize $\alpha$};

}

\node[inner sep=1,draw,shape=circle] at (0,0) {\small $1$};
\node[inner sep=1,draw,shape=circle] at (90-0.5*\ph:\x+\xx) {\small $2$};
\node[inner sep=1,draw,shape=circle] at (90+3.5*\ph:\x+\xx) {\small $3$};
\node[inner sep=1,draw,shape=circle] at (90+2.5*\ph:\x+\xx) {\small $4$};
\node[inner sep=1,draw,shape=circle] at (90+1.5*\ph:\x+\xx) {\small $5$};
\node[inner sep=1,draw,shape=circle] at (90+0.5*\ph:\x+\xx) {\small $6$};

\end{scope}

\end{tikzpicture}
\caption{The arrangement of $\alpha\vert\gamma$}
\label{Fig-a5-a4-AAD-albegac-alga}
\end{figure}

By $\alpha\vert\gamma\cdots=\beta\vert\alpha\vert\gamma\cdots=\alpha\beta\gamma^2$ around the pentagon and $\gamma\vert\gamma\cdots$ implying $\beta\vert\beta\cdots=\beta^3$, we therefore determine the fourth tiling in Figure \ref{Fig-a5-a4-Archimedean-Tilings}.
\end{proof}

\begin{prop} \label{Prop-a5-a4-albe2} There are three dihedral tilings with vertex $\alpha\beta^2$ in Figures \ref{Fig-a5-a4-Tiling-albe2-albega2-Pen5albegaga}, \ref{Fig-a5-a4-Tilings-albe2-albega2-Pen3albegaga}. 
\end{prop}

Each of these tilings has $40$ rhombi and $12$ pentagons. They are the first three tilings of Figure \ref{Fig-a5-a4-Archimedean-Tilings}.

\begin{proof} By $\beta>\alpha>\gamma$ and $\alpha\beta^2$ and $\beta+\gamma>\pi$, we get $\pi > \beta > \frac{2}{3}\pi>\alpha > \frac{3}{5}\pi$ and $2\gamma>\alpha$. Then $ \beta  > \alpha > \frac{3}{5}\pi$ and $\alpha\beta^2$ imply $\frac{7}{10}\pi > \beta$. By $\alpha > \frac{3}{5}\pi$ and $2\gamma>\alpha$, we also get $\gamma>\frac{1}{2}\alpha>\frac{3}{10}\pi$. By $3\gamma>\frac{7}{10}\pi > \beta$ and $\alpha\beta^2$, we also have $3\alpha+\gamma, \alpha+\beta+3\gamma>2\pi$. This further implies that $\alpha^3\gamma\cdots$ is not a vertex and $\beta^2\cdots=\alpha\beta^2$ and $\alpha\beta\gamma\cdots=\alpha\beta\gamma^2$. We summarise the inequalities below,
\begin{align*}
\alpha > \tfrac{3}{5}\pi, \quad
\beta > \tfrac{2}{3}\pi, \quad
\gamma > \tfrac{3}{10}\pi.
\end{align*}

Lemma \ref{Lem-albe-alga} implies that $\alpha\gamma\cdots$ is a vertex. By the above inequalities and $\alpha\beta\gamma\cdots=\alpha\beta\gamma^2$ and no $\alpha^3\gamma\cdots$, we get $\alpha\gamma\cdots= \alpha^2\gamma^2, \alpha^2\gamma^3, \alpha\gamma^3, \alpha\gamma^4, \alpha\gamma^5, \alpha\beta\gamma^2$. Moreover, $\alpha\beta^2$ and one of $\alpha\gamma^3, \alpha\gamma^5, \alpha^2\gamma^3$ imply that there is no solution to \eqref{Eq-cot-al-be-ga} for $\alpha \in (\frac{3}{5}\pi, \frac{2}{3}\pi)$. So $\alpha\gamma^3, \alpha\gamma^5, \alpha^2\gamma^3$ is not a vertex. Hence we have
\begin{align*} 
\alpha\gamma\cdots= \alpha^2\gamma^2, \alpha\gamma^4, \alpha\beta\gamma^2.
\end{align*}
Any vertex above combined with $\alpha\beta^2$ uniquely determine angle values via \eqref{Eq-cot-al-be-ga} below
\begin{align*}
&\alpha^2\gamma^2:&
&\alpha \approx 0.63636\pi,&
&\beta \approx 0.68182\pi,&
&\gamma \approx 0.36364\pi.& \\
&\alpha\gamma^4 / \alpha\beta\gamma^2:&
&\alpha \approx 0.62526\pi,&
&\beta \approx 0.68737\pi,&
&\gamma \approx 0.34369\pi.&
\end{align*}
Note that either one of $\alpha\gamma^4, \alpha\beta\gamma^2$ combined with $\alpha\beta^2$ yield the same angle values. The corresponding angle values above further imply that the vertices are one of the following combinations,
\begin{align*}
\AVC &= \{ \alpha\beta^2, \alpha^2\gamma^2 \}; \\
\AVC &= \{ \alpha\beta^2, \alpha\gamma^4, \alpha\beta\gamma^2 \}. 
\end{align*}


We further divide the discussion into the cases below.

\begin{case*}[$\AVC = \{ \alpha\beta^2, \alpha^2\gamma^2 \}$] By $\AVC = \{ \alpha\beta^2, \alpha^2\gamma^2 \}$, we know that $\gamma^3\cdots$ are not vertices and $\alpha\beta\cdots=\alpha\beta^2$ and $\alpha^2\cdots=\alpha\gamma\cdots=\alpha^2\gamma^2$.

The arrangement of $\alpha\vert\alpha$ determines tiles $T_1, T_2$ in the first picture of Figure \ref{Fig-a5-a4-AAD-albe2-algac-gagaga/algaal-alal}. Up to mirror symmetry, $\alpha^2\cdots=\alpha^2\gamma^2$ determines $T_3, T_4$. Then $\alpha_2\beta_3\cdots, \alpha_2\beta_4\cdots=\alpha\beta^2$ determine $T_5, T_6$ respectively. This implies $\gamma_5 \vert \alpha_2 \vert \gamma_6$. It means that there is an $\alpha^2\gamma^2= \vert \alpha\vert \gamma \vert \alpha \vert \gamma \vert$.

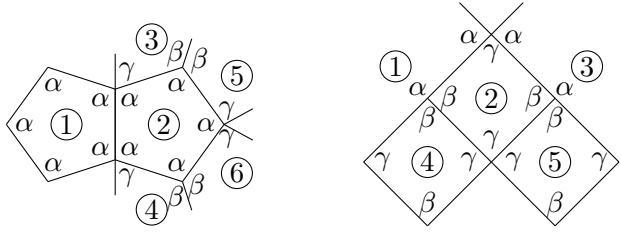
\begin{figure}[h!] 
\centering
\begin{tikzpicture}

\tikzmath{
\s=1;
}

\begin{scope}[yshift=0.5*\s cm] 

\tikzmath{
\r=0.8;
\g=5;
\ph=360/\g;
\x=\r*cos(\ph/2);
\y=\r*sin(\ph/2);
\rr=2*\y/sqrt(2);
\h=4;
\th=360/\h;
\xx=\rr*cos(\th/2);
}

\foreach \aa in {-1,1} {

\tikzset{shift={(\aa*\x,0)}, xscale=-\aa}

\foreach \a in {0,...,4} {

\draw[rotate=\a*\ph]
	(0.5*\ph:\r) -- (-0.5*\ph:\r)
;

\node at (0.5*\ph+\a*\ph:0.7*\r) {\small $\alpha$};

}

}

\foreach \a in {-1,1} {

\tikzset{shift={(\x,0)}}

\draw[rotate=\a*\ph]
	(0:\r) -- (0:1.5*\r)
;

}

\foreach \a in {0,1} {

\draw[rotate=\a*2*\th]
	(0,\y) -- (0, 2*\y)
;

}

\draw[shift={(\x,0)}]
	(0:\r) -- (10:1.5*\r)
	(0:\r) -- (-10:1.5*\r)
;

\node at (0.2*\r, 1.5*\y) {\small $\gamma$};
\node at (1*\r, 2*\y) {\small $\beta$};

\node at (1.4*\r, 1.8*\y) {\small $\beta$};
\node at (2.3*\x, 0.4*\y) {\small $\gamma$};

\node at (1.35*\r, -1.75*\y) {\small $\beta$};
\node at (2.3*\x, -0.4*\y) {\small $\gamma$};

\node at (0.2*\r, -1.5*\y) {\small $\gamma$};
\node at (1*\r, -2.1*\y) {\small $\beta$};

\node[inner sep=1,draw,shape=circle] at (-\x,0) {\small $1$};
\node[inner sep=1,draw,shape=circle] at (\x,0) {\small $2$};
\node[inner sep=1,draw,shape=circle] at (0.75*\x,1.75*\x) {\small $3$};
\node[inner sep=1,draw,shape=circle] at (0.75*\x,-1.75*\x) {\small $4$};
\node[inner sep=1,draw,shape=circle] at (2.5*\x,1*\x) {\small $5$};
\node[inner sep=1,draw,shape=circle] at (2.5*\x,-1*\x) {\small $6$};

\end{scope} 

\begin{scope}[xshift=5*\s cm] 

\tikzmath{
\r=1.2;
\gon=4;
\th=360/\gon;
\x=\r*cos(\th/2);
}

\foreach \a in {0,...,3} {
\draw[rotate=\th*\a] 
	(0:0) -- (0.5*\th:\r)
; 
}

\foreach \a in {0,...,2} {
\draw[rotate=\th*\a] 
	(0.5*\th:\r) -- (0:2*\x)
	(-0.5*\th:\r) -- (0:2*\x)
;
}

\draw[]
	(90:2*\x) -- ([shift=(90:2*\x)]45: 0.5*\r)
	(90:2*\x) -- ([shift=(90:2*\x)]135: 0.5*\r)
;

\node at (180:0.35*\x) {\small $\gamma$}; 
\node at (180:1.7*\x) {\small $\gamma$};
\node at (1.625*\th:0.85*\r) {\small $\beta$};
\node at (2.375*\th:0.85*\r) {\small $\beta$};

\node at (90:0.35*\x) {\small $\gamma$}; 
\node at (90:1.7*\x) {\small $\gamma$}; 
\node at (0.625*\th:0.85*\r) {\small $\beta$}; 
\node at (1.375*\th:0.85*\r) {\small $\beta$};

\node at (0:0.35*\x) {\small $\gamma$}; 
\node at (0:1.7*\x) {\small $\gamma$};
\node at (0.375*\th:0.85*\r) {\small $\beta$};
\node at (-0.375*\th:0.85*\r) {\small $\beta$};

\node at (100:2*\x) {\small $\alpha$}; 
\node at (1.5*\th:1.15*\r) {\small $\alpha$};

\node at (80:2*\x) {\small $\alpha$}; 
\node at (0.5*\th:1.15*\r) {\small $\alpha$};

\node[inner sep=1,draw,shape=circle] at (180:1*\x) {\small $4$};
\node[inner sep=1,draw,shape=circle] at (90:1*\x) {\small $2$};
\node[inner sep=1,draw,shape=circle] at (0:1*\x) {\small $5$};
\node[inner sep=1,draw,shape=circle] at (1.5*\th:1.5*\r) {\small $1$};
\node[inner sep=1,draw,shape=circle] at (0.5*\th:1.5*\r) {\small $3$};

\end{scope}

\end{tikzpicture}
\caption{The arrangements of $\gamma\vert\gamma\vert\gamma$ and $\alpha\vert\gamma\vert\alpha$ and $\alpha\vert\alpha$}
\label{Fig-a5-a4-AAD-albe2-algac-gagaga/algaal-alal}
\end{figure}

The angle arrangement $\alpha\vert\gamma\vert\alpha$ determines tiles $T_1, T_2, T_3$ in the second picture of Figure \ref{Fig-a5-a4-AAD-albe2-algac-gagaga/algaal-alal}. Then $\alpha_1\beta_2\cdots, \alpha_3\beta_2\cdots=\alpha\beta^2$ determine $T_4, T_5$. This implies $\gamma_2\gamma_4\gamma_5\cdots$, which is not a vertex, a contradiction. So $\alpha\vert\gamma\vert\alpha\cdots$ is not a vertex. So $\alpha\vert\alpha\cdots$ is also not a vertex. Hence $\alpha^2\gamma^2$ is not a vertex.
\end{case*}

\begin{case*}[$\AVC= \{ \alpha\beta^2, \alpha\gamma^4, \alpha\beta\gamma^2 \}$] We have $\alpha\gamma\cdots=\alpha\gamma^4$ and $\alpha^2\gamma\cdots$ is not a vertex and $\beta^2\cdots=\alpha\beta^2$. 

Reverse the deduction in the second picture of Figure \ref{Fig-a5-a4-AAD-albe2-algac-gagaga/algaal-alal}, the angle arrangement of $\gamma\vert\gamma\vert\gamma $ leads to $\alpha_1\alpha_3\gamma_2\cdots$, a contradiction. Hence $\alpha\gamma^4$ is not a vertex. So $\alpha\gamma\cdots=\alpha\beta\gamma^2$. 

The vertices $\alpha\beta^2, \alpha\beta\gamma^2$ imply 
\begin{align*}
\beta=\pi - \tfrac{1}{2}\alpha, \quad
\gamma = \tfrac{1}{2}\pi - \tfrac{1}{4}\alpha.
\end{align*}
Substituting the above into \eqref{Eq-cot-al-be-ga}, we get
\begin{align*}
\alpha \approx 0.62526\pi, \quad
\beta \approx 0.68737\pi, \quad
\gamma \approx 0.34369\pi, \quad
x \approx 0.14901\pi.
\end{align*}
With the above angle values and $\alpha\gamma\cdots=\alpha\beta\gamma^2$, one can determine all vertices satisfying \eqref{Eq-cot-al-be-ga}. Hence we get
\begin{align*}
\AVC = \{ \alpha\beta^2, \alpha\beta\gamma^2 \}.
\end{align*}


Suppose there is one pentagon with five $\vert\alpha\vert\beta\vert\gamma\vert\gamma\vert$'s, then we determine the tiling in Figure \ref{Fig-a5-a4-Tiling-albe2-albega2-Pen5albegaga}. 

\begin{figure}[h!] 
\centering
\begin{tikzpicture}

\tikzmath{
\r=0.6;
\th=360/5;
}

\begin{scope}[]

\fill[gray!50]
	(0,0) circle (3.25*\r)
;

\fill[white]
	(0,0) circle (3*\r)
;

\fill[gray!50]
	(90:0.5*\r) -- (90+\th:0.5*\r) -- (90+2*\th:0.5*\r) -- (90+3*\th:0.5*\r) -- (90+4*\th:0.5*\r)
;

\foreach \a in {0,...,4} {

\fill[gray!50, rotate=\a*\th]
	(90-\th:1*\r) -- (\th/6:1.5*\r) -- (0.5*\th:1.5*\r) -- (5*\th/6:1.5*\r)  -- (90-0.5*\th:1*\r) -- cycle
;

\fill[gray!50, rotate=\a*\th]
	(90-\th/12:2*\r) --  (5*\th/6:2*\r)  --  (0.5*\th:2*\r)  -- (42:2.5*\r) --  (90-\th/6:2.5*\r)  -- cycle
;

}

\foreach \a in {0,...,4}
{

\fill[rotate=\a*\th]
	(90:0.5*\r) circle (0.12*\r)
	(\th/12:3*\r) circle (0.12*\r)
;


\draw[rotate=\a*\th]
	(90:0.5*\r) -- (90-\th:0.5*\r)  
	(90-0.5*\th:1*\r) -- (90-\th:1*\r) 
	(\th/6:1.5*\r) -- (0.5*\th:1.5*\r)  
	(0.5*\th:1.5*\r) -- (5*\th/6:1.5*\r) 
	(90-0.5*\th:1*\r) -- (5*\th/6:1.5*\r) 
	(90:1*\r) -- (90-\th/12:1.5*\r) 
	%
	(90:0.5*\r) -- (90-0.5*\th:1*\r) 
	(90-\th:0.5*\r) -- (90-\th:1*\r) 
	(90:1*\r) -- (5*\th/6:1.5*\r) 
	%
	(-\th/6:2*\r) -- (\th/6:2*\r) 
	(0.5*\th:2*\r) -- (5*\th/6:2*\r) 
	(90-\th/6:2.5*\r) -- (42:2.5*\r) 
	(0.5*\th:2*\r) -- (42:2.5*\r) 
	(\th/6:2*\r) -- (\th/12:2.5*\r) 
	(90-\th/12:1.5*\r) --  (5*\th/6:2*\r) 
	(0.5*\th:1.5*\r) -- (0.5*\th:2*\r) 
	%
	(0.5*\th:1.5*\r) -- (\th/6:2*\r) 
	%
	(-\th/6:1.5*\r) -- (-\th/6:2*\r) 
	%
	%
	(0.5*\th:2*\r) -- (\th/12:2.5*\r) 
	(90-\th/6:2.5*\r) -- (90-\th/6:3*\r) 
	(42:2.5*\r) -- (\th/12:3*\r) 
;


	(54:0.6*\r) circle (0.05*\r)
	(30:0.8*\r) circle (0.05*\r)
	(6:0.8*\r) circle (0.05*\r)
	(-6:1.1*\r) circle (0.05*\r)
	(6:1.3*\r) circle (0.05*\r)
	(-6:1.6*\r) circle (0.05*\r)
	(-18:1.6*\r) circle (0.05*\r)
	(-30:1.8*\r) circle (0.05*\r)
	(6:1.8*\r) circle (0.05*\r)
	(18:1.6*\r) circle (0.05*\r)
	(30:1.8*\r) circle (0.05*\r)
	(15:2.1*\r) circle (0.05*\r)
	(30:2.2*\r) circle (0.05*\r)
	(12:2.7*\r) circle (0.05*\r)
	(0:2.7*\r) circle (0.05*\r)
	(-30:2.8*\r) circle (0.05*\r);
	
	(-18:0.6*\r) -- (6:0.8*\r) -- (30:0.8*\r) -- (54:0.6*\r)
	(6:0.8*\r) -- (-6:1.1*\r) -- (6:1.3*\r) -- (-6:1.6*\r) -- (6:1.8*\r) -- (18:1.6*\r) -- (30:1.8*\r) -- (42:1.8*\r) -- (54:1.6*\r) -- (66:1.6*\r)
	(30:1.8*\r)  -- (15:2.1*\r) -- (30:2.2*\r) -- (12:2.7*\r) -- (0:2.7*\r) -- (-30:2.8*\r) -- (-60:2.7*\r);


}

\draw
	(0,0) circle (3*\r)
;

\end{scope}

\end{tikzpicture}
\caption{The tilings with $\alpha\beta^2, \alpha\beta\gamma^2 $, where $\vert\alpha\vert\beta\vert\gamma\vert\gamma\vert$ denoted by a \quotes{\textbullet}}
\label{Fig-a5-a4-Tiling-albe2-albega2-Pen5albegaga}
\end{figure}
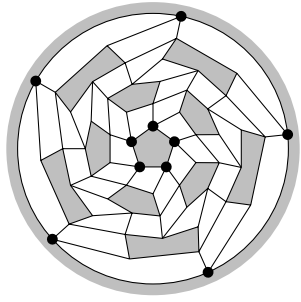

Next, we assume that there is no pentagon with five $\vert \alpha\vert \beta \vert \gamma \vert \gamma \vert$'s. The same deduction in Figure \ref{Fig-a5-a4-AAD-albegac-alga} and no $\alpha^2\cdots$ show that any pentagon without five $\vert \alpha\vert \beta \vert \gamma \vert \gamma \vert$'s has at least one $\alpha\beta^2$. The unique angle arrangement of $\alpha\beta^2$ determines tiles $T_1, T_2, T_3$ in Figure \ref{Fig-a5-a4-AAD-albe2-albega2-Pen-albe2}. By mirror symmetry and no $\alpha^2\cdots$, we know that $T_4$ is a rhombus and we may assume $T_4$ with $\beta_4, \gamma_4$ as configured. 

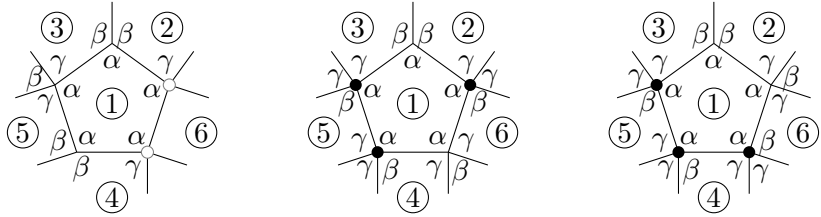
\begin{figure}[h!] 
\centering
\begin{tikzpicture}

\tikzmath{
\s=1;
\r=0.8;
\g=5;
\ph=360/\g;
\Ph=(\g-2)*180/\g;
\x=\r*cos(\ph/2);
\y=\r*sin(\ph/2);
\rr=2*\y/sqrt(2);
\h=4;
\th=360/\h;
\xx=\rr*cos(\th/2);
\R=sqrt(\r^2+(\r)^2 - 2*\r*\r*cos(180-\Ph/2));
\lxx = 2*\xx*cos(\Ph/2);
}

\begin{scope}[] 

\foreach \p in {0,...,4} {

\draw[rotate=\p*\ph]
	(90-1*\ph:\r) -- (90:\r)
;


%
\node at (90+\p*\ph:0.7*\r) {\small $\alpha$};

}

\foreach \p in {1,2,4} {

\draw[rotate=\p*\ph]
	(90+2*\ph:\r) -- ([shift={(90+2*\ph:\r)}]270:\lxx)
;

}

\foreach \p in {1,3,4} {

\draw[rotate=\p*\ph]
	(90+3*\ph:\r) -- ([shift={(90+3*\ph:\r)}]270:\lxx)
;

}

\draw[]
	(90:\r) -- (90:\r+\lxx)
;

\draw[]
	(90+3*\ph:\r) -- ([shift={(90+3*\ph:\r)}]270:\lxx)
;

\node at (90-0.15*\ph:1.15*\r) {\small $\beta$};
\node at (90-0.75*\ph:1.1*\r) {\small $\gamma$};

\node at (90+0.15*\ph:1.15*\r) {\small $\beta$};
\node at (90+0.75*\ph:1.1*\r) {\small $\gamma$};



\node at (90+1.25*\ph:1.1*\r) {\small $\gamma$};

\node at (90+\ph:1.35*\r) {\small $\beta$};



\node at (90+1.75*\ph:1.1*\r) {\small $\beta$};

\node at (90+2.15*\ph:1.2*\r) {\small $\beta$};
\node at (90+2.75*\ph:1.15*\r) {\small $\gamma$};

\node [shape = circle, draw=gray, fill = white, minimum size = 0.16 cm, inner sep=0pt] at (90-\ph:\r) {};
\node [shape = circle, draw=gray, fill = white, minimum size = 0.16 cm, inner sep=0pt] at (90-2*\ph:\r) {};

\node[inner sep=1,draw,shape=circle] at (0,0) {\small $1$};
\node[inner sep=1,draw,shape=circle] at (90-0.5*\ph:\x+0.75*\r) {\small $2$};
\node[inner sep=1,draw,shape=circle] at (90+0.5*\ph:\x+0.75*\r) {\small $3$};
\node[inner sep=1,draw,shape=circle] at (90+2.5*\ph:\x+0.75*\r) {\small $4$};
\node[inner sep=1,draw,shape=circle] at (90+1.5*\ph:\x+0.75*\r) {\small $5$};
\node[inner sep=1,draw,shape=circle] at (90-1.5*\ph:\x+0.75*\r) {\small $6$};

\end{scope}

\begin{scope}[xshift=4*\s cm] 

\foreach \p in {0,...,4} {

\draw[rotate=\p*\ph]
	(90-1*\ph:\r) -- (90:\r)
;


%
\node at (90+\p*\ph:0.7*\r) {\small $\alpha$};

}

\foreach \p in {1,2,4} {

\draw[rotate=\p*\ph]
	(90+2*\ph:\r) -- ([shift={(90+2*\ph:\r)}]270:\lxx)
;

}

\foreach \p in {1,3,4} {

\draw[rotate=\p*\ph]
	(90+3*\ph:\r) -- ([shift={(90+3*\ph:\r)}]270:\lxx)
;

}

\draw[]
	(90:\r) -- (90:\r+\lxx)
;

\draw[]
	(90+2*\ph:\r) -- ([shift={(90+2*\ph:\r)}]270:\lxx)
	(90+3*\ph:\r) -- ([shift={(90+3*\ph:\r)}]270:\lxx)
;

\node at (90-0.15*\ph:1.15*\r) {\small $\beta$};
\node at (90-0.75*\ph:1.1*\r) {\small $\gamma$};

\node at (90+0.15*\ph:1.15*\r) {\small $\beta$};
\node at (90+0.75*\ph:1.1*\r) {\small $\gamma$};

\node at (90-\ph:1.35*\r) {\small $\gamma$};

\node at (90-1.25*\ph:1.1*\r) {\small $\beta$};
\node at (90-1.75*\ph:1.1*\r) {\small $\gamma$};

\node at (90+1.25*\ph:1.1*\r) {\small $\beta$};

\node at (90+\ph:1.35*\r) {\small $\gamma$};

\node at (90+2*\ph:1.35*\r) {\small $\gamma$};

\node at (90-2.025*\ph:1.4*\r) {\small $\beta$};

\node at (90+1.75*\ph:1.1*\r) {\small $\gamma$};

\node at (90+2.25*\ph:1.2*\r) {\small $\beta$};
\node at (90+2.75*\ph:1.15*\r) {\small $\gamma$};

\node [shape = circle, fill = black, minimum size = 0.16 cm, inner sep=0pt] at (90-\ph:\r) {};
\node [shape = circle, fill = black, minimum size = 0.16 cm, inner sep=0pt] at (90+\ph:\r) {};
\node [shape = circle, fill = black, minimum size = 0.16 cm, inner sep=0pt] at (90+2*\ph:\r) {};

\node[inner sep=1,draw,shape=circle] at (0,0) {\small $1$};
\node[inner sep=1,draw,shape=circle] at (90-0.5*\ph:\x+0.75*\r) {\small $2$};
\node[inner sep=1,draw,shape=circle] at (90+0.5*\ph:\x+0.75*\r) {\small $3$};
\node[inner sep=1,draw,shape=circle] at (90+2.5*\ph:\x+0.75*\r) {\small $4$};
\node[inner sep=1,draw,shape=circle] at (90+1.5*\ph:\x+0.75*\r) {\small $5$};
\node[inner sep=1,draw,shape=circle] at (90-1.5*\ph:\x+0.75*\r) {\small $6$};

\end{scope}

\begin{scope}[xshift=8*\s cm] 

\foreach \p in {0,...,4} {

\draw[rotate=\p*\ph]
	(90-1*\ph:\r) -- (90:\r)
;


%
\node at (90+\p*\ph:0.7*\r) {\small $\alpha$};

}

\foreach \p in {1,2,4} {

\draw[rotate=\p*\ph]
	(90+2*\ph:\r) -- ([shift={(90+2*\ph:\r)}]270:\lxx)
;

}

\foreach \p in {1,3,4} {

\draw[rotate=\p*\ph]
	(90+3*\ph:\r) -- ([shift={(90+3*\ph:\r)}]270:\lxx)
;

}

\draw[]
	(90:\r) -- (90:\r+\lxx)
;

\draw[]
	(90+2*\ph:\r) -- ([shift={(90+2*\ph:\r)}]270:\lxx)
	(90+3*\ph:\r) -- ([shift={(90+3*\ph:\r)}]270:\lxx)
;

\node at (90-0.15*\ph:1.15*\r) {\small $\beta$};
\node at (90-0.75*\ph:1.1*\r) {\small $\gamma$};

\node at (90+0.15*\ph:1.15*\r) {\small $\beta$};
\node at (90+0.75*\ph:1.1*\r) {\small $\gamma$};

\node at (90-\ph:1.35*\r) {\small $\beta$};

\node at (90-1.25*\ph:1.1*\r) {\small $\gamma$};
\node at (90-1.75*\ph:1.1*\r) {\small $\beta$};

\node at (90+1.25*\ph:1.1*\r) {\small $\beta$};

\node at (90+\ph:1.35*\r) {\small $\gamma$};

\node at (90+2*\ph:1.35*\r) {\small $\gamma$};

\node at (90-2.025*\ph:1.4*\r) {\small $\gamma$};

\node at (90+1.75*\ph:1.1*\r) {\small $\gamma$};

\node at (90+2.25*\ph:1.2*\r) {\small $\beta$};
\node at (90+2.75*\ph:1.15*\r) {\small $\gamma$};

\node [shape = circle, fill = black, minimum size = 0.16 cm, inner sep=0pt] at (90-2*\ph:\r) {};
\node [shape = circle, fill = black, minimum size = 0.16 cm, inner sep=0pt] at (90+\ph:\r) {};
\node [shape = circle, fill = black, minimum size = 0.16 cm, inner sep=0pt] at (90+2*\ph:\r) {};

\node[inner sep=1,draw,shape=circle] at (0,0) {\small $1$};
\node[inner sep=1,draw,shape=circle] at (90-0.5*\ph:\x+0.75*\r) {\small $2$};
\node[inner sep=1,draw,shape=circle] at (90+0.5*\ph:\x+0.75*\r) {\small $3$};
\node[inner sep=1,draw,shape=circle] at (90+2.5*\ph:\x+0.75*\r) {\small $4$};
\node[inner sep=1,draw,shape=circle] at (90+1.5*\ph:\x+0.75*\r) {\small $5$};
\node[inner sep=1,draw,shape=circle] at (90-1.5*\ph:\x+0.75*\r) {\small $6$};

\end{scope}

\end{tikzpicture}
\caption{The five vertices at a pentagon with $\alpha\beta^2$}
\label{Fig-a5-a4-AAD-albe2-albega2-Pen-albe2}
\end{figure}

If there is another $\alpha\beta^2$, then $\alpha_1\beta_4\cdots=\alpha\beta^2$ determines $T_5$ in the first picture of Figure \ref{Fig-a5-a4-AAD-albe2-albega2-Pen-albe2}. The same argument also shows that $T_6$ is a rhombus. Then one of $\alpha_1\gamma_2\cdots, \alpha_1\gamma_4\cdots$ is $\vert \alpha \vert \beta \vert \gamma \vert \gamma\vert$.

If there is no more $\alpha\beta^2$, then $\alpha_1\beta_4\cdots=\alpha\beta\gamma^2$ determines $T_5$ in the second and third picture of Figure \ref{Fig-a5-a4-AAD-albe2-albega2-Pen-albe2}. As $T_6$ is a rhombus, the two angle configurations determine the tile in the second and the third picture respectively and we have exactly three $\vert \alpha \vert \beta \vert \gamma \vert \gamma\vert$'s in both pictures.

The second picture of Figure \ref{Fig-a5-a4-AAD-albe2-albega2-Pen-albe2} can be translated into the first picture of Figure \ref{Fig-a5-a4-AAD-albe2-albega2-one-albegaga}. We also know $T_7, T_8$ from $\alpha_1\beta_6\gamma_2\gamma_7$ and $\alpha_1\beta_5\gamma_3\gamma_8$ respectively. Then $\beta_2\beta_7\cdots$, $\beta_3\beta_8\cdots=\alpha\beta^2$ determine $T_9, T_{10}$, which implies $\gamma_2\gamma_3\cdots=\alpha\gamma^2$ or $\alpha^2\gamma^2\cdots$. None of them is a vertex. Hence the three $\vert\alpha\vert\beta\vert\gamma\vert\gamma\vert$'s at a pentagon must be consecutive.

If there is exactly one $\vert\alpha\vert\beta\vert\gamma\vert\gamma\vert$ at a pentagon, then by the symmetry of the first picture of Figure \ref{Fig-a5-a4-AAD-albe2-albega2-Pen-albe2}, we may assume $\alpha_1\gamma_2\cdots=\vert \alpha \vert \gamma \vert \beta \vert \gamma \vert$. Then focusing on $T_1, T_2, T_3$, we translating the picture into Figure \ref{Fig-a5-a4-AAD-albe2-albega2-one-albegaga}. Then $\alpha_1\beta_7\gamma_2\gamma_6, \alpha_1\beta_8\gamma_3\gamma_5$ determine $T_7, T_8$. The mirror symmetry of $T_1, T_2, T_3$ and $\gamma_2\gamma_3\cdots=\alpha\beta\gamma^2$ determine $T_9, T_{10}$. This implies that $T_{10}$ has at least two $\vert\alpha\vert\beta\vert\gamma\vert\gamma\vert$'s.

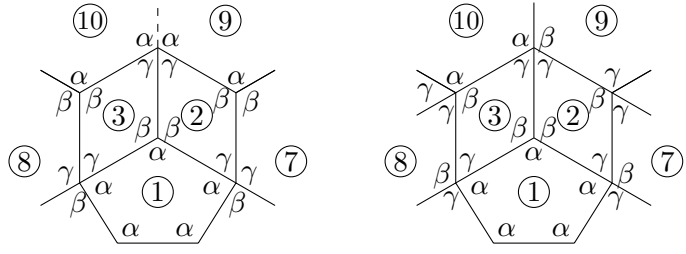
\begin{figure}[h!] 
\centering
\begin{tikzpicture}

\tikzmath{
\s=1;
\r=1.2;
\th=360/3;
\x=\r*cos(0.5*\th);
}

\begin{scope}

\foreach \a in {0,1,2} {

\draw[rotate=\th*\a]
	(0:0) -- (90:\r) 
;

}

\foreach \a in {0,1} {

\draw[rotate=\th*\a]
	(90:\r) -- (90-0.5*\th:2*\x)
	(90+2*\th:\r) -- (90+2.5*\th:2*\x)
;

}

\foreach \aa in {-1, 1} {

\tikzset{xscale=\aa}

\foreach \a in {0,1} {

\draw[rotate=\th*\a]
	(90-0.5*\th:2*\x) -- (90-0.5*\th:3*\x)
;

}

\node at (80-\th:1.15*\r) {\small $\beta$};

}

\draw[]
	%
	(90-1*\th:2*\x) -- (90-1*\th:3*\x)
	(90+1*\th:2*\x) -- (90+1*\th:3*\x)
	%
	%
	(90-\th:2*\x) -- (270+0.175*\th: 2.5*\x)
	(90+\th:2*\x) -- (270-0.175*\th: 2.5*\x)
	(270-0.175*\th: 2.5*\x) -- (270+0.175*\th: 2.5*\x)
;

\draw[dashed]
	(90:\r) -- (90:\r+\x)
;

\node at ([shift={(0,-\x)}]90:0.65*\x) {\small $\alpha$}; 
\node at ([shift={(0,-\x)}]-5:1.2*\x) {\small $\alpha$};
\node at ([shift={(0,-\x)}]185:1.2*\x) {\small $\alpha$};
\node at ([shift={(0,-\x)}]270+0.25*\th:1.2*\x) {\small $\alpha$};
\node at ([shift={(0,-\x)}]270-0.25*\th:1.2*\x) {\small $\alpha$};

\node at ([shift={(30:\x)}]30:0.625*\x) {\small $\beta$}; 
\node at ([shift={(30:\x)}]210:0.625*\x) {\small $\beta$};
\node at ([shift={(30:\x)}]120:1.2*\x) {\small $\gamma$};
\node at ([shift={(30:\x)}]-60:1.2*\x) {\small $\gamma$};

\node at ([shift={(150:\x)}]150:0.625*\x) {\small $\beta$}; 
\node at ([shift={(150:\x)}]330:0.625*\x) {\small $\beta$};
\node at ([shift={(150:\x)}]240:1.2*\x) {\small $\gamma$};
\node at ([shift={(150:\x)}]60:1.2*\x) {\small $\gamma$};

\node at (97.5:1.1*\r) {\small $\alpha$}; 
\node at (142.5:1.1*\r) {\small $\alpha$}; 

\node at (82.5:1.1*\r) {\small $\alpha$};  
\node at (37.5:1.1*\r) {\small $\alpha$};


\node at (-22.5:1.1*\r) {\small $\gamma$}; 
\node at (20:1.1*\r) {\small $\beta$};  

\node at (180+22.5:1.1*\r) {\small $\gamma$}; 
\node at (180-20:1.1*\r) {\small $\beta$};  

\node[inner sep=1,draw,shape=circle] at (270:0.6*\r) {\small $1$};
\node[inner sep=1,draw,shape=circle] at (30:0.5*\r) {\small $2$};
\node[inner sep=1,draw,shape=circle] at (150:0.5*\r) {\small $3$};

\node[inner sep=0.25,draw,shape=circle] at (120:1.5*\r) {\footnotesize $10$};
\node[inner sep=1,draw,shape=circle] at (60:1.5*\r) {\small $9$};
\node[inner sep=1,draw,shape=circle] at (-10:1.5*\r) {\small $7$};
\node[inner sep=1,draw,shape=circle] at (190:1.5*\r) {\small $8$};

\end{scope}

\begin{scope}[xshift=5*\s cm]

\foreach \a in {0,1,2} {

\draw[rotate=\th*\a]
	(0:0) -- (90:\r) 
;

}

\foreach \a in {0,1} {

\draw[rotate=\th*\a]
	(90:\r) -- (90-0.5*\th:2*\x)
	(90+2*\th:\r) -- (90+2.5*\th:2*\x)
;

}

\foreach \aa in {-1, 1} {

\tikzset{xscale=\aa}

\draw[]	
	(90-0.5*\th:2*\x) -- ([shift={(90-0.5*\th:2*\x)}]-30:\x)
;

\foreach \a in {0,1} {

\draw[rotate=\th*\a]
	(90-0.5*\th:2*\x) -- (90-0.5*\th:3*\x)
;

}

\node at (-22.5:1.1*\r) {\small $\beta$}; 
\node at (15:1*\r) {\small $\gamma$}; 
\node at (82.5-\th:1.15*\r) {\small $\gamma$};  

}

\draw[]
	(90:\r) -- (90:\r+\x)
	(90-1*\th:2*\x) -- (90-1*\th:3*\x)
	(90+1*\th:2*\x) -- (90+1*\th:3*\x)
	%
	%
	(90-\th:2*\x) -- (270+0.175*\th: 2.5*\x)
	(90+\th:2*\x) -- (270-0.175*\th: 2.5*\x)
	(270-0.175*\th: 2.5*\x) -- (270+0.175*\th: 2.5*\x)
;

\node at ([shift={(0,-\x)}]90:0.65*\x) {\small $\alpha$}; 
\node at ([shift={(0,-\x)}]-5:1.2*\x) {\small $\alpha$};
\node at ([shift={(0,-\x)}]185:1.2*\x) {\small $\alpha$};
\node at ([shift={(0,-\x)}]270+0.25*\th:1.2*\x) {\small $\alpha$};
\node at ([shift={(0,-\x)}]270-0.25*\th:1.2*\x) {\small $\alpha$};

\node at ([shift={(30:\x)}]30:0.625*\x) {\small $\beta$}; 
\node at ([shift={(30:\x)}]210:0.625*\x) {\small $\beta$};
\node at ([shift={(30:\x)}]120:1.2*\x) {\small $\gamma$};
\node at ([shift={(30:\x)}]-60:1.2*\x) {\small $\gamma$};

\node at ([shift={(150:\x)}]150:0.625*\x) {\small $\beta$}; 
\node at ([shift={(150:\x)}]330:0.625*\x) {\small $\beta$};
\node at ([shift={(150:\x)}]240:1.2*\x) {\small $\gamma$};
\node at ([shift={(150:\x)}]60:1.2*\x) {\small $\gamma$};

\node at (97.5:1.1*\r) {\small $\alpha$}; 
\node at (142.5:1.1*\r) {\small $\alpha$}; 

\node at (82.5:1.1*\r) {\small $\beta$};  
\node at (37.5:1.1*\r) {\small $\gamma$};


\node at (180-22.5:1.3*\r) {\small $\gamma$};  


\node[inner sep=1,draw,shape=circle] at (270:0.6*\r) {\small $1$};
\node[inner sep=1,draw,shape=circle] at (30:0.5*\r) {\small $2$};
\node[inner sep=1,draw,shape=circle] at (150:0.5*\r) {\small $3$};

\node[inner sep=0.25,draw,shape=circle] at (120:1.5*\r) {\footnotesize $10$};
\node[inner sep=1,draw,shape=circle] at (60:1.5*\r) {\small $9$};
\node[inner sep=1,draw,shape=circle] at (-10:1.5*\r) {\small $7$};
\node[inner sep=1,draw,shape=circle] at (190:1.5*\r) {\small $8$};

\end{scope}

\end{tikzpicture}
\caption{The arrangement of $\alpha\beta^2$}
\label{Fig-a5-a4-AAD-albe2-albega2-one-albegaga}
\end{figure}


The above discussion entails that there is a pentagon with exactly three consecutive $\vert\alpha\vert\beta\vert\gamma\vert\gamma\vert$'s.


The vertices $\alpha\beta^2, \alpha\beta\gamma^2$ imply $\beta=2\gamma$. Then every rhombus in the tilings can be subdivided into two equilateral triangles with three $\gamma$'s. This means that the underlying dihedral tiling is the snub dodecahedron. To obtain the remaining tilings, we merge the pairs of equilateral triangles in the snub dodecahedron by deleting their common edge in ways that there is a pentagon with exactly three $\vert\alpha\vert\beta\vert\gamma\vert\gamma\vert$'s.

In the snub dodecahedron in Figure \ref{Fig-a5-a4-Tilings-albe2-albega2-Pen3albegaga}, we denote the pentagons as $P_1, P_2, ..., P_{12}$. The edges $e_i$'s to be deleted are labeled by $i=1,2,..., 40$ in both pictures, so are the rhombus $Q_1, Q_2, ..., Q_{40}$. Up to symmetry, we may assume that the central pentagon $P_1$ in Figure \ref{Fig-a5-a4-Tilings-albe2-albega2-Pen3albegaga} has exactly three $\vert\alpha\vert\beta\vert\gamma\vert\gamma\vert$'s and we determine $Q_1, Q_2, ..., Q_6$. By $\alpha_4\beta_3\gamma_4\cdots=\alpha\beta\gamma^2$, we further determine $Q_7$. We repeat the same argument and determine $Q_8, ..., Q_{27}$. Then $P_{12}$ has two adjacent $\vert \alpha \vert \beta \vert \gamma \vert \gamma \vert$'s. This means that there is another $\vert \alpha \vert \beta \vert \gamma \vert \gamma \vert$ on either side of them. Suppose we have the third $\vert \alpha \vert \beta \vert \gamma \vert \gamma \vert$ in the first picture. Then we further determine $Q_{28}$. Repeat the same argument of $\alpha\beta\gamma\cdots=\alpha\beta\gamma^2$, we further determine $Q_{29}, ..., Q_{40}$ and obtain the tiling in the first picture. Suppose we have the third $\vert \alpha \vert \beta \vert \gamma \vert \gamma \vert$ in the second picture. Then we further determine $Q_{28}, ..., Q_{40}$ and obtain the tiling in the second picture.

\begin{figure}[h!]
\centering
\begin{tikzpicture}[>=latex,scale=1]

\tikzmath{
\s=1;
\r=1;
\rr=0.08*\r;
\th=360/5;
}

\begin{scope}[scale=1] 

\fill[gray!50]
	(0,0) circle (3.2*\r)
;

\fill[white]
	(0,0) circle (3*\r)
;

\fill[gray!50]
	(90:0.5*\r) -- (90+\th:0.5*\r) -- (90+2*\th:0.5*\r) -- (90+3*\th:0.5*\r) -- (90+4*\th:0.5*\r)
;

\foreach \a in {0,...,4} {

\fill[gray!50, rotate=\a*\th]
	(90-\th:1*\r) -- (\th/6:1.5*\r) -- (0.5*\th:1.5*\r) -- (5*\th/6:1.5*\r)  -- (90-0.5*\th:1*\r) -- cycle
;

\fill[gray!50, rotate=\a*\th]
	(90-\th/12:2*\r) --  (5*\th/6:2*\r)  --  (0.5*\th:2*\r)  -- (42:2.5*\r) --  (90-\th/6:2.5*\r)  -- cycle
;

}

\draw[gray!20, line width=1.05]
	(90-\th:0.5*\r) -- (90-0.5*\th:1*\r) 
	(90:0.5*\r) -- (90+0.5*\th:1*\r) 
	(90+\th:0.5*\r) -- (90+1.5*\th:1*\r) 
	(90+2*\th:0.5*\r) -- (90+2*\th:1*\r) 
	(90-2*\th:0.5*\r) -- (90-2.5*\th:1*\r) 
	(90-2*\th:0.5*\r) -- (90-1.5*\th:1*\r) 
	(90:1*\r) -- (90-0.5*\th:1*\r) 
	(90-\th:1*\r) -- (90-1.5*\th:1*\r) 
	(90+\th:1*\r) -- (90+0.5*\th:1*\r) 
	(90-2*\th:1*\r) -- (5*\th/6-2*\th:1.5*\r) 
	(90+2*\th:1*\r) -- (5*\th/6+2*\th:1.5*\r) 
	(\th/6:1.5*\r) -- (-\th/6:1.5*\r) 
	(\th/6+\th:1.5*\r) -- (-\th/6+\th:1.5*\r) 
	(\th/6+2*\th:1.5*\r) -- (-\th/6+2*\th:1.5*\r) 
	(5*\th/6-1*\th:2*\r) -- (0.5*\th-1*\th:1.5*\r) 
	(\th/6:1.5*\r) -- (\th/6:2*\r) 
	(5*\th/6:2*\r) -- (0.5*\th:1.5*\r) 
	(\th/6:2*\r) -- (0.5*\th:2*\r) 
	(5*\th/6+\th:2*\r) -- (0.5*\th+\th:1.5*\r) 
	(\th/6+\th:1.5*\r) -- (\th/6+\th:2*\r) 
	(\th/6+\th:2*\r) -- (0.5*\th+\th:2*\r) 
	(\th/6+2*\th:1.5*\r) -- (\th/6+2*\th:2*\r) 
	(\th/6+2*\th:2*\r) -- (0.5*\th+2*\th:2*\r) 
	(5*\th/6+2*\th:2*\r) -- (0.5*\th+2*\th:1.5*\r) 
	(90-\th/12+2*\th:1.5*\r) --  (5*\th/6+2*\th:2*\r) 
	(0.5*\th+3*\th:1.5*\r) -- (\th/6+3*\th:2*\r) 
	(5*\th/6+3*\th:2*\r) -- (0.5*\th+3*\th:1.5*\r) 
	(90-\th/12+3*\th:1.5*\r) --  (5*\th/6+3*\th:2*\r) 
	(0.5*\th-1*\th:1.5*\r) -- (\th/6-1*\th:2*\r) 
	(42:2.5*\r) -- (\th/12:2.5*\r) 
	(42+\th:2.5*\r) -- (\th/12+\th:2.5*\r) 
	(42+2*\th:2.5*\r) -- (\th/12+2*\th:2.5*\r) 
	(0.5*\th+3*\th:2*\r) -- (\th/12+3*\th:2.5*\r) 
	(0.5*\th-1*\th:2*\r) -- (\th/12-1*\th:2.5*\r) 
	(90-\th/6-2*\th:2.5*\r) -- (90-\th/6-2*\th:3*\r) 
	(42:2.5*\r) -- (90-\th/6:3*\r) 
	(42+\th:2.5*\r) -- (90-\th/6+\th:3*\r) 
	(42+2*\th:2.5*\r) -- (90-\th/6+2*\th:3*\r) 
	(42-\th:2.5*\r) -- (90-\th/6-\th:3*\r) 
	(42+3*\th:2.5*\r) -- (\th/12+3*\th:3*\r) 
;

\node at ($(90-\th:0.5*\r) !1/2! (90-0.5*\th:1*\r)$) {\footnotesize $1$};
\node at ($(90:0.5*\r) !1/2! (90+0.5*\th:1*\r)$) {\footnotesize $2$};
\node at ($(90+\th:0.5*\r) !1/2! (90+1.5*\th:1*\r)$) {\footnotesize $3$};
\node at ($(90+2*\th:0.5*\r) !1/2! (90+2*\th:1*\r)$) {\footnotesize $4$};
\node at ($(90-2*\th:0.5*\r) !1/2! (90-2.5*\th:1*\r)$) {\footnotesize $5$};
\node at ($(90-2*\th:0.5*\r) !1/2! (90-1.5*\th:1*\r)$) {\footnotesize $6$};
\node at ($(90-\th:1*\r) !1/2! (90-1.5*\th:1*\r)$) {\footnotesize $9$};
\node at ($(90+\th:1*\r) !1/2! (90+0.5*\th:1*\r)$) {\footnotesize $18$};
\node at ($(90+2*\th:1*\r) !1/2! (5*\th/6+2*\th:1.5*\r)$) {\footnotesize $7$};
\node at ($(90-2*\th:1*\r) !1/2! (5*\th/6-2*\th:1.5*\r)$) {\footnotesize $8$};
\node at ($(\th/6:1.5*\r) !1/2! (-\th/6:1.5*\r)$) {\footnotesize $10$};
\node at ($(\th/6+\th:1.5*\r) !1/2! (-\th/6+\th:1.5*\r)$) {\footnotesize $14$};
\node at ($(\th/6+2*\th:1.5*\r) !1/2! (-\th/6+2*\th:1.5*\r)$) {\footnotesize $19$};

\node at ($(\th/6:1.5*\r) !1/2! (\th/6:2*\r)$) {\footnotesize $12$};
\node at ($(90:1*\r) !1/2! (90-0.5*\th:1*\r)$) {\footnotesize $13$};
\node at ($(5*\th/6-1*\th:2*\r) !1/2! (0.5*\th-1*\th:1.5*\r)$) {\footnotesize $11$};
\node at ($(5*\th/6:2*\r) !1/2! (0.5*\th:1.5*\r)$) {\footnotesize $15$};
\node at ($(\th/6:2*\r) !1/2! (0.5*\th:2*\r)$) {\footnotesize $16$};

\node at ($(5*\th/6+\th:2*\r) !1/2! (0.5*\th+\th:1.5*\r) $) {\footnotesize $20$};
\node at ($(\th/6+\th:1.5*\r) !1/2! (\th/6+\th:2*\r)$) {\footnotesize $21$};
\node at ($(\th/6+\th:2*\r) !1/2! (0.5*\th+\th:2*\r)$) {\footnotesize $22$};
\node at ($(\th/6+2*\th:1.5*\r) !1/2! (\th/6+2*\th:2*\r)$) {\footnotesize $27$};
\node at ($(\th/6+2*\th:2*\r) !1/2! (0.5*\th+2*\th:2*\r)$) {\footnotesize $29$};

\node at ($(5*\th/6+2*\th:2*\r) !1/2! (0.5*\th+2*\th:1.5*\r)$) {\footnotesize $30$};
\node at ($(90-\th/12+2*\th:1.5*\r) !1/2! (5*\th/6+2*\th:2*\r)$) {\footnotesize $31$};
\node at ($(0.5*\th+3*\th:1.5*\r) !1/2! (\th/6+3*\th:2*\r)$) {\footnotesize $32$};
\node at ($(0.5*\th+3*\th:2*\r) !1/2! (\th/12+3*\th:2.5*\r) $) {\footnotesize $33$};
\node at ($(5*\th/6+3*\th:2*\r) !1/2! (0.5*\th+3*\th:1.5*\r)$) {\footnotesize $34$};
\node at ($(90-\th/12+3*\th:1.5*\r) !1/2! (5*\th/6+3*\th:2*\r)$) {\footnotesize $35$};
\node at ($(0.5*\th-1*\th:1.5*\r) !1/2! (\th/6-1*\th:2*\r)$) {\footnotesize $36$};
\node at ($(0.5*\th-1*\th:2*\r) !1/2! (\th/12-1*\th:2.5*\r) $) {\footnotesize $37$};

\node at ($(42:2.5*\r) !1/2! (\th/12:2.5*\r)$) {\footnotesize $17$};
\node at ($(42+\th:2.5*\r) !1/2! (\th/12+\th:2.5*\r)$) {\footnotesize $23$};
\node at ($(42+2*\th:2.5*\r) !1/2! (\th/12+2*\th:2.5*\r)$) {\footnotesize $28$};
\node at ($(90-\th/6-2*\th:2.5*\r) !1/2! (90-\th/6-2*\th:3*\r)$) {\footnotesize $38$};

\node at ($(42-\th:2.5*\r) !1/2! (90-\th/6-\th:3*\r)$) {\footnotesize $26$};
\node at ($(42:2.5*\r) !1/2! (90-\th/6:3*\r)$) {\footnotesize $24$};
\node at ($(42+\th:2.5*\r) !1/2! (90-\th/6+\th:3*\r)$) {\footnotesize $25$};
\node at ($(42+3*\th:2.5*\r) !1/2! (\th/12+3*\th:3*\r) $) {\footnotesize $39$};
\node at ($(42+2*\th:2.5*\r) !1/2! (90-\th/6+2*\th:3*\r)$) {\footnotesize $40$};


\draw[]
	(90+2*\th:0.5*\r) -- (90+2.5*\th:1*\r) 
	(42+3*\th:2.5*\r) -- (90-\th/6+3*\th:3*\r)
;

\foreach \a in {0,1,2,3} {

\draw[rotate=\a*\th]
	(90-2*\th:0.5*\r) -- (90-2*\th:1*\r)
	(90-\th:0.5*\r) -- (90-1.5*\th:1*\r)
	(90-\th/6-\th:2.5*\r) -- (90-\th/6-\th:3*\r)
	(42-\th:2.5*\r) -- (\th/12-\th:3*\r)
;

}

\foreach \a in {0,1,2} {

\draw[rotate=\a*\th]
	(90-\th:1*\r) -- (5*\th/6-\th:1.5*\r)
	(90-\th/12-\th:1.5*\r) --  (5*\th/6-\th:2*\r) 
	(0.5*\th:1.5*\r) -- (\th/6:2*\r) 
	(0.5*\th:2*\r) -- (\th/12:2.5*\r) 
;

}

\foreach \a in {0,1} {

\draw[rotate=\a*\th]
	(90+2*\th:1*\r) -- (90+1.5*\th:1*\r) 
	(\th/6+3*\th:2*\r) -- (0.5*\th+3*\th:2*\r) 
	(\th/6+3*\th:1.5*\r) -- (\th/6+3*\th:2*\r) 
	(\th/6+3*\th:1.5*\r) -- (-\th/6+3*\th:1.5*\r)
	(42+3*\th:2.5*\r) -- (\th/12+3*\th:2.5*\r) 
;

}

\foreach \a in {0,...,4}
{
\begin{scope}[rotate=\a*\th]

\draw
	(90:0.5*\r) -- (90-\th:0.5*\r)  
	(90-0.5*\th:1*\r) -- (90-\th:1*\r) 
	(\th/6:1.5*\r) -- (0.5*\th:1.5*\r) 
	(0.5*\th:1.5*\r) -- (5*\th/6:1.5*\r)  
	(90-0.5*\th:1*\r) -- (5*\th/6:1.5*\r) 
	(90:1*\r) -- (90-\th/12:1.5*\r) 
	(0.5*\th:1.5*\r) -- (0.5*\th:2*\r) 
	(-\th/6:2*\r) -- (\th/6:2*\r) 
	(0.5*\th:2*\r) -- (5*\th/6:2*\r) 
	(90-\th/6:2.5*\r) -- (42:2.5*\r) 
	(0.5*\th:2*\r) -- (42:2.5*\r) 
	(\th/6:2*\r) -- (\th/12:2.5*\r) 
	(-\th/6:1.5*\r) -- (-\th/6:2*\r) 
;

\begin{scope}[red!50]

	(54:0.6*\r) circle (0.05*\r)
	(30:0.8*\r) circle (0.05*\r)
	(6:0.8*\r) circle (0.05*\r)
	(-6:1.1*\r) circle (0.05*\r)
	(6:1.3*\r) circle (0.05*\r)
	(-6:1.6*\r) circle (0.05*\r)
	(-18:1.6*\r) circle (0.05*\r)
	(-30:1.8*\r) circle (0.05*\r)
	(6:1.8*\r) circle (0.05*\r)
	(18:1.6*\r) circle (0.05*\r)
	(30:1.8*\r) circle (0.05*\r)
	(15:2.1*\r) circle (0.05*\r)
	(30:2.2*\r) circle (0.05*\r)
	(12:2.7*\r) circle (0.05*\r)
	(0:2.7*\r) circle (0.05*\r)
	(-30:2.8*\r) circle (0.05*\r);
	
	(-18:0.6*\r) -- (6:0.8*\r) -- (30:0.8*\r) -- (54:0.6*\r)
	(6:0.8*\r) -- (-6:1.1*\r) -- (6:1.3*\r) -- (-6:1.6*\r) -- (6:1.8*\r) -- (18:1.6*\r) -- (30:1.8*\r) -- (42:1.8*\r) -- (54:1.6*\r) -- (66:1.6*\r)
	(30:1.8*\r)  -- (15:2.1*\r) -- (30:2.2*\r) -- (12:2.7*\r) -- (0:2.7*\r) -- (-30:2.8*\r) -- (-60:2.7*\r);

\end{scope}

\end{scope}
}

\draw
	(0,0) circle (3*\r)
;

\fill 
	(90:0.5*\r) circle (\rr)
	(90-\th:0.5*\r) circle (\rr)
	(90+\th:0.5*\r) circle (\rr)
	(90+1.5*\th:1*\r) circle (\rr)
	(5*\th/6+2*\th:1.5*\r) circle (\rr)
	(0.5*\th+2*\th:1.5*\r) circle (\rr)
	(-\th/6:2*\r) circle (\rr)
	(0.5*\th-1*\th:2*\r) circle (\rr)
	(42-\th:2.5*\r) circle (\rr)
	(\th/12:3*\r) circle (\rr)
	(\th/12+\th:3*\r) circle (\rr)
	(\th/12+2*\th:3*\r) circle (\rr)
;

\node at (0,0) {\footnotesize $P_1$};
\node at (0.5*\th:1.2*\r) {\footnotesize $P_2$};
\node at (0.5*\th+\th:1.2*\r) {\footnotesize $P_3$};
\node at (0.5*\th+2*\th:1.2*\r) {\footnotesize $P_4$};
\node at (0.5*\th+3*\th:1.2*\r) {\footnotesize $P_5$};
\node at (0.5*\th-1*\th:1.2*\r) {\footnotesize $P_6$};

\node at (0.8*\th:2.2*\r) {\footnotesize $P_7$};
\node at (0.8*\th+\th:2.2*\r) {\footnotesize $P_8$};
\node at (0.8*\th+2*\th:2.2*\r) {\footnotesize $P_9$};
\node at (0.85*\th+3*\th:2.2*\r) {\footnotesize $P_{10}$};
\node at (0.725*\th-\th:2.2*\r) {\footnotesize $P_{11}$};

\node at (90:3.25*\r) {\footnotesize $P_{12}$};

\end{scope} 

\begin{scope}[xshift=7.5*\s cm] 

\fill[gray!50]
	(0,0) circle (3.2*\r)
;

\fill[white]
	(0,0) circle (3*\r)
;

\fill[gray!50]
	(90:0.5*\r) -- (90+\th:0.5*\r) -- (90+2*\th:0.5*\r) -- (90+3*\th:0.5*\r) -- (90+4*\th:0.5*\r)
;

\foreach \a in {0,...,4} {

\fill[gray!50, rotate=\a*\th]
	(90-\th:1*\r) -- (\th/6:1.5*\r) -- (0.5*\th:1.5*\r) -- (5*\th/6:1.5*\r)  -- (90-0.5*\th:1*\r) -- cycle
;

\fill[gray!50, rotate=\a*\th]
	(90-\th/12:2*\r) --  (5*\th/6:2*\r)  --  (0.5*\th:2*\r)  -- (42:2.5*\r) --  (90-\th/6:2.5*\r)  -- cycle
;

}

\draw[gray!20, line width=1.05]
	(90-\th:0.5*\r) -- (90-0.5*\th:1*\r) 
	(90:0.5*\r) -- (90+0.5*\th:1*\r) 
	(90+\th:0.5*\r) -- (90+1.5*\th:1*\r) 
	(90+2*\th:0.5*\r) -- (90+2*\th:1*\r) 
	(90-2*\th:0.5*\r) -- (90-2.5*\th:1*\r) 
	(90-2*\th:0.5*\r) -- (90-1.5*\th:1*\r) 
	(90:1*\r) -- (90-0.5*\th:1*\r) 
	(90-\th:1*\r) -- (90-1.5*\th:1*\r) 
	(90+\th:1*\r) -- (90+0.5*\th:1*\r) 
	(90-2*\th:1*\r) -- (5*\th/6-2*\th:1.5*\r) 
	(90+2*\th:1*\r) -- (5*\th/6+2*\th:1.5*\r) 
	(\th/6:1.5*\r) -- (-\th/6:1.5*\r) 
	(\th/6+\th:1.5*\r) -- (-\th/6+\th:1.5*\r) 
	(\th/6+2*\th:1.5*\r) -- (-\th/6+2*\th:1.5*\r) 
	(5*\th/6-1*\th:2*\r) -- (0.5*\th-1*\th:1.5*\r) 
	(\th/6:1.5*\r) -- (\th/6:2*\r) 
	(5*\th/6:2*\r) -- (0.5*\th:1.5*\r) 
	(\th/6:2*\r) -- (0.5*\th:2*\r) 
	(42:2.5*\r) -- (\th/12:2.5*\r) 
	(5*\th/6+\th:2*\r) -- (0.5*\th+\th:1.5*\r) 
	(\th/6+\th:1.5*\r) -- (\th/6+\th:2*\r) 
	(\th/6+\th:2*\r) -- (0.5*\th+\th:2*\r) 
	(42+\th:2.5*\r) -- (\th/12+\th:2.5*\r) 
	(42:2.5*\r) -- (90-\th/6:3*\r) 
	(42+\th:2.5*\r) -- (90-\th/6+\th:3*\r) 
	(42-\th:2.5*\r) -- (90-\th/6-\th:3*\r) 
	(\th/6+2*\th:1.5*\r) -- (\th/6+2*\th:2*\r) 
	(42+2*\th:2.5*\r) -- (\th/12+2*\th:3*\r) 
	(0.5*\th+2*\th:2*\r) -- (\th/12+2*\th:2.5*\r) 
	(0.5*\th+2*\th:1.5*\r) -- (0.5*\th+2*\th:2*\r) 
	(90-\th/6+2*\th:2.5*\r) -- (90-\th/6+2*\th:3*\r) 
	(-\th/6+3*\th:1.5*\r) -- (-\th/6+3*\th:2*\r) 
	(\th/6+3*\th:1.5*\r) -- (\th/6+3*\th:2*\r)  
	(0.5*\th+3*\th:1.5*\r) -- (0.5*\th+3*\th:2*\r) 
	(0.5*\th+3*\th:2*\r) -- (\th/12+3*\th:2.5*\r) 
	(-\th/6-1*\th:1.5*\r) -- (-\th/6-1*\th:2*\r) 
	(\th/6-1*\th:1.5*\r) -- (\th/6-1*\th:2*\r)  
	(\th/6-\th:2*\r) -- (0.5*\th-\th:2*\r) 
	(42-1*\th:2.5*\r) -- (\th/12-1*\th:2.5*\r)  
	(42+3*\th:2.5*\r) -- (90-\th/6+3*\th:3*\r) 
;

\node at ($(90-\th:0.5*\r) !1/2! (90-0.5*\th:1*\r)$) {\footnotesize $1$};
\node at ($(90:0.5*\r) !1/2! (90+0.5*\th:1*\r)$) {\footnotesize $2$};
\node at ($(90+\th:0.5*\r) !1/2! (90+1.5*\th:1*\r)$) {\footnotesize $3$};
\node at ($(90+2*\th:0.5*\r) !1/2! (90+2*\th:1*\r)$) {\footnotesize $4$};
\node at ($(90-2*\th:0.5*\r) !1/2! (90-2.5*\th:1*\r)$) {\footnotesize $5$};
\node at ($(90-2*\th:0.5*\r) !1/2! (90-1.5*\th:1*\r)$) {\footnotesize $6$};
\node at ($(90-\th:1*\r) !1/2! (90-1.5*\th:1*\r)$) {\footnotesize $9$};
\node at ($(90+\th:1*\r) !1/2! (90+0.5*\th:1*\r)$) {\footnotesize $18$};
\node at ($(90+2*\th:1*\r) !1/2! (5*\th/6+2*\th:1.5*\r)$) {\footnotesize $7$};
\node at ($(90-2*\th:1*\r) !1/2! (5*\th/6-2*\th:1.5*\r)$) {\footnotesize $8$};
\node at ($(\th/6:1.5*\r) !1/2! (-\th/6:1.5*\r)$) {\footnotesize $10$};
\node at ($(\th/6+\th:1.5*\r) !1/2! (-\th/6+\th:1.5*\r)$) {\footnotesize $14$};
\node at ($(\th/6+2*\th:1.5*\r) !1/2! (-\th/6+2*\th:1.5*\r)$) {\footnotesize $19$};

\node at ($(\th/6:1.5*\r) !1/2! (\th/6:2*\r)$) {\footnotesize $12$};
\node at ($(90:1*\r) !1/2! (90-0.5*\th:1*\r)$) {\footnotesize $13$};
\node at ($(5*\th/6-1*\th:2*\r) !1/2! (0.5*\th-1*\th:1.5*\r)$) {\footnotesize $11$};
\node at ($(5*\th/6:2*\r) !1/2! (0.5*\th:1.5*\r)$) {\footnotesize $15$};
\node at ($(\th/6:2*\r) !1/2! (0.5*\th:2*\r)$) {\footnotesize $16$};

\node at ($(5*\th/6+\th:2*\r) !1/2! (0.5*\th+\th:1.5*\r) $) {\footnotesize $20$};
\node at ($(\th/6+\th:1.5*\r) !1/2! (\th/6+\th:2*\r)$) {\footnotesize $21$};
\node at ($(\th/6+\th:2*\r) !1/2! (0.5*\th+\th:2*\r)$) {\footnotesize $22$};
\node at ($(\th/6+2*\th:1.5*\r) !1/2! (\th/6+2*\th:2*\r)$) {\footnotesize $27$};

\node at ($(0.5*\th+2*\th:2*\r) !1/2! (\th/12+2*\th:2.5*\r) $) {\footnotesize $29$};
\node at ($(0.5*\th+2*\th:1.5*\r) !1/2! (0.5*\th+2*\th:2*\r)$) {\footnotesize $30$};
\node at ($(-\th/6+3*\th:1.5*\r) !1/2! (-\th/6+3*\th:2*\r)$) {\footnotesize $32$};
\node at ($(90-\th/6+2*\th:2.5*\r) !1/2! (90-\th/6+2*\th:3*\r)$) {\footnotesize $31$};

\node at ($(\th/6+3*\th:1.5*\r) !1/2! (\th/6+3*\th:2*\r) $) {\footnotesize $33$};
\node at ($(0.5*\th+3*\th:1.5*\r) !1/2! (0.5*\th+3*\th:2*\r)$) {\footnotesize $34$};
\node at ($(0.5*\th+3*\th:2*\r) !1/2! (\th/12+3*\th:2.5*\r)$) {\footnotesize $35$};
\node at ($(-\th/6-1*\th:1.5*\r) !1/2! (-\th/6-1*\th:2*\r)$) {\footnotesize $36$};
\node at ($(\th/6-1*\th:1.5*\r) !1/2! (\th/6-1*\th:2*\r)$) {\footnotesize $37$};

\node at ($(42:2.5*\r) !1/2! (\th/12:2.5*\r)$) {\footnotesize $17$};
\node at ($(42+\th:2.5*\r) !1/2! (\th/12+\th:2.5*\r)$) {\footnotesize $23$};
\node at ($(42+2*\th:2.5*\r) !1/2! (\th/12+2*\th:3*\r)$) {\footnotesize $28$};
\node at ($(\th/6-\th:2*\r) !1/2! (0.5*\th-\th:2*\r)$) {\footnotesize $38$};

\node at ($(42-\th:2.5*\r) !1/2! (90-\th/6-\th:3*\r)$) {\footnotesize $26$};
\node at ($(42:2.5*\r) !1/2! (90-\th/6:3*\r)$) {\footnotesize $24$};
\node at ($(42+\th:2.5*\r) !1/2! (90-\th/6+\th:3*\r)$) {\footnotesize $25$};
\node at ($(42-1*\th:2.5*\r) !1/2! (\th/12-1*\th:2.5*\r)$) {\footnotesize $39$};
\node at ($(42+3*\th:2.5*\r) !1/2! (90-\th/6+3*\th:3*\r)$) {\footnotesize $40$};

\draw[]
	(90+2*\th:0.5*\r) -- (90+2.5*\th:1*\r) 
	(42+2*\th:2.5*\r) -- (90-\th/6+2*\th:3*\r)
;

\foreach \a in {0,1,2,3} {

\draw[rotate=\a*\th]
	(90-2*\th:0.5*\r) -- (90-2*\th:1*\r)
	(90-\th:0.5*\r) -- (90-1.5*\th:1*\r)
	(90-\th/6-2*\th:2.5*\r) -- (90-\th/6-2*\th:3*\r)
	(42-2*\th:2.5*\r) -- (\th/12-2*\th:3*\r) 
;

}

\foreach \a in {0,1,2} {

\draw[rotate=\a*\th]
	(90-\th:1*\r) -- (5*\th/6-\th:1.5*\r)
	(90-\th/12-\th:1.5*\r) --  (5*\th/6-\th:2*\r) 
	(0.5*\th-\th:1.5*\r) -- (0.5*\th-\th:2*\r) 
	(0.5*\th:1.5*\r) -- (\th/6:2*\r) 
	(-\th/6:1.5*\r) -- (-\th/6:2*\r) 
	(0.5*\th-\th:2*\r) -- (\th/12-\th:2.5*\r)
;

}

\foreach \a in {0,1} {

\draw[rotate=\a*\th]
	(90+2*\th:1*\r) -- (90+1.5*\th:1*\r) 
	%
	%
	(90-\th/12+2*\th:1.5*\r) --  (5*\th/6+2*\th:2*\r)
	(0.5*\th+3*\th:1.5*\r) -- (\th/6+3*\th:2*\r) 
	(\th/6+3*\th:1.5*\r) -- (-\th/6+3*\th:1.5*\r)
	(5*\th/6+2*\th:2*\r) -- (0.5*\th+2*\th:1.5*\r)
	(\th/6+2*\th:2*\r) -- (0.5*\th+2*\th:2*\r) 
	(42+2*\th:2.5*\r) -- (\th/12+2*\th:2.5*\r) 
;

}

\foreach \a in {0,...,4}
{
\begin{scope}[rotate=\a*\th]

\draw
	(90:0.5*\r) -- (90-\th:0.5*\r)  
	(90-0.5*\th:1*\r) -- (90-\th:1*\r) 
	(\th/6:1.5*\r) -- (0.5*\th:1.5*\r) 
	(0.5*\th:1.5*\r) -- (5*\th/6:1.5*\r)  
	(90-0.5*\th:1*\r) -- (5*\th/6:1.5*\r) 
	(90:1*\r) -- (90-\th/12:1.5*\r) 
	(-\th/6:2*\r) -- (\th/6:2*\r) 
	(0.5*\th:2*\r) -- (5*\th/6:2*\r) 
	(90-\th/6:2.5*\r) -- (42:2.5*\r) 
	(0.5*\th:2*\r) -- (42:2.5*\r) 
	(\th/6:2*\r) -- (\th/12:2.5*\r) 
;

\begin{scope}[red!50]

	(54:0.6*\r) circle (0.05*\r)
	(30:0.8*\r) circle (0.05*\r)
	(6:0.8*\r) circle (0.05*\r)
	(-6:1.1*\r) circle (0.05*\r)
	(6:1.3*\r) circle (0.05*\r)
	(-6:1.6*\r) circle (0.05*\r)
	(-18:1.6*\r) circle (0.05*\r)
	(-30:1.8*\r) circle (0.05*\r)
	(6:1.8*\r) circle (0.05*\r)
	(18:1.6*\r) circle (0.05*\r)
	(30:1.8*\r) circle (0.05*\r)
	(15:2.1*\r) circle (0.05*\r)
	(30:2.2*\r) circle (0.05*\r)
	(12:2.7*\r) circle (0.05*\r)
	(0:2.7*\r) circle (0.05*\r)
	(-30:2.8*\r) circle (0.05*\r);
	
	(-18:0.6*\r) -- (6:0.8*\r) -- (30:0.8*\r) -- (54:0.6*\r)
	(6:0.8*\r) -- (-6:1.1*\r) -- (6:1.3*\r) -- (-6:1.6*\r) -- (6:1.8*\r) -- (18:1.6*\r) -- (30:1.8*\r) -- (42:1.8*\r) -- (54:1.6*\r) -- (66:1.6*\r)
	(30:1.8*\r)  -- (15:2.1*\r) -- (30:2.2*\r) -- (12:2.7*\r) -- (0:2.7*\r) -- (-30:2.8*\r) -- (-60:2.7*\r);

\end{scope}

\end{scope}
}

\draw
	(0,0) circle (3*\r)
;

\fill 
	(90:0.5*\r) circle (\rr)
	(90-\th:0.5*\r) circle (\rr)
	(90+\th:0.5*\r) circle (\rr)
	(90-2*\th:1*\r) circle (\rr)
	(0.5*\th-1*\th:1.5*\r) circle (\rr)
	(90-\th/12-2*\th:1.5*\r) circle (\rr)
	(-\th/6+2*\th:2*\r) circle (\rr)
	(\th/6+2*\th:2*\r) circle (\rr) 
	(\th/12+2*\th:2.5*\r) circle (\rr)
	(\th/12:3*\r) circle (\rr)
	(\th/12+\th:3*\r) circle (\rr)
	(\th/12-\th:3*\r) circle (\rr)
;

\node at (0,0) {\footnotesize $P_1$};
\node at (0.5*\th:1.2*\r) {\footnotesize $P_2$};
\node at (0.5*\th+\th:1.2*\r) {\footnotesize $P_3$};
\node at (0.5*\th+2*\th:1.2*\r) {\footnotesize $P_4$};
\node at (0.5*\th+3*\th:1.2*\r) {\footnotesize $P_5$};
\node at (0.5*\th-1*\th:1.2*\r) {\footnotesize $P_6$};

\node at (0.8*\th:2.2*\r) {\footnotesize $P_7$};
\node at (0.8*\th+\th:2.2*\r) {\footnotesize $P_8$};
\node at (0.8*\th+2*\th:2.2*\r) {\footnotesize $P_9$};
\node at (0.85*\th+3*\th:2.2*\r) {\footnotesize $P_{10}$};
\node at (0.725*\th-\th:2.2*\r) {\footnotesize $P_{11}$};

\node at (90:3.25*\r) {\footnotesize $P_{12}$};

\end{scope} 

\end{tikzpicture}
\caption{Dihedral tilings with $\alpha\beta^2, \alpha\beta\gamma^2$ and pentagons with exactly three vertices of angle arrangement $\vert \alpha \vert \beta \vert \gamma \vert \gamma \vert$, the trio $\vert \alpha \vert \beta \vert \gamma \vert \gamma \vert$'s are denoted by \quotes{$\bullet$}}
\label{Fig-a5-a4-Tilings-albe2-albega2-Pen3albegaga}
\end{figure}
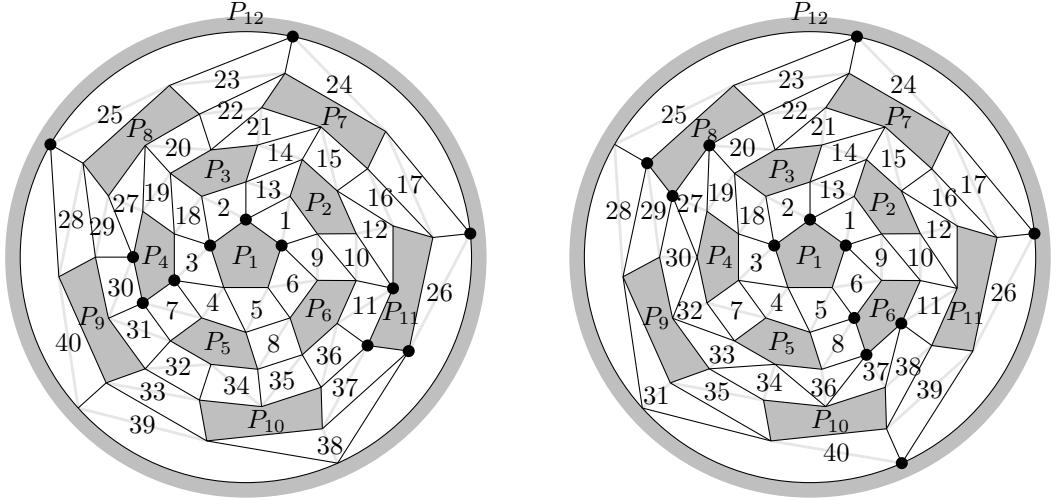

\end{case*}

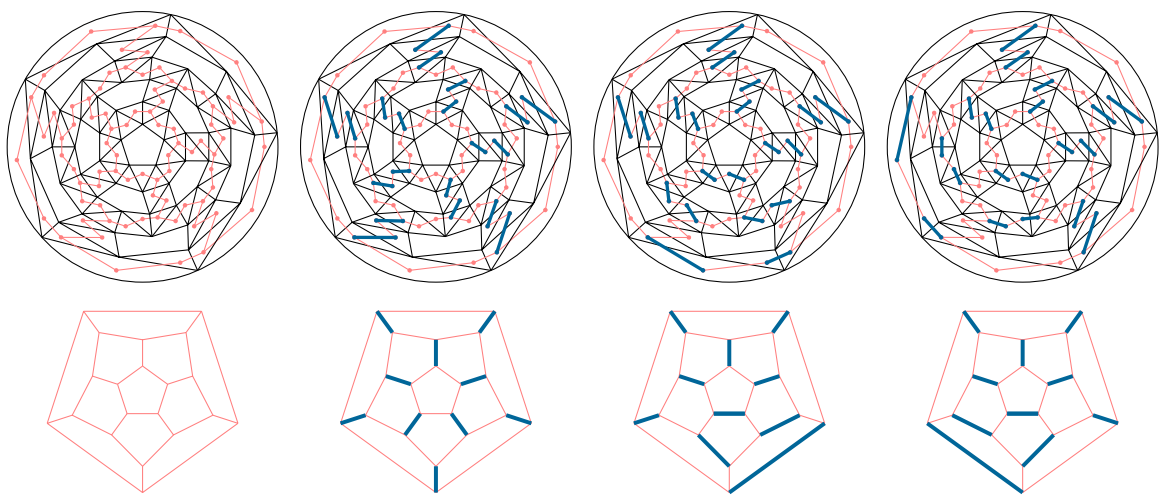
\begin{figure}[h!]
\centering
\begin{tikzpicture}[>=latex,scale=1]

\tikzmath{
\s=0.975;
}

\begin{scope} 

\tikzmath{
\r=0.6;
\th=360/5;
\x=\r*cos(\th/2);
}

\begin{scope}

	(0,0) circle (3.2*\r)
;

	(0,0) circle (3*\r)
;

	(90:0.5*\r) -- (90+\th:0.5*\r) -- (90+2*\th:0.5*\r) -- (90+3*\th:0.5*\r) -- (90+4*\th:0.5*\r)
;

\foreach \a in {0,...,4} {

	(90-\th:1*\r) -- (\th/6:1.5*\r) -- (0.5*\th:1.5*\r) -- (5*\th/6:1.5*\r)  -- (90-0.5*\th:1*\r) -- cycle
;

	(90-\th/12:2*\r) --  (5*\th/6:2*\r)  --  (0.5*\th:2*\r)  -- (42:2.5*\r) --  (90-\th/6:2.5*\r)  -- cycle
;

}

\foreach \a in {0,...,4}
{
\begin{scope}[rotate=\a*\th]

\draw
	(90:0.5*\r) -- (90-\th:0.5*\r)  
	(90:0.5*\r) -- (90-0.5*\th:1*\r) 
	(90-0.5*\th:1*\r) -- (90-\th:0.5*\r) 
	(90-\th:0.5*\r) -- (90-\th:1*\r) 
	(90-0.5*\th:1*\r) -- (90-\th:1*\r) 
	(90-\th:1*\r) -- (90-1.5*\th:1*\r) 
	(\th/6:1.5*\r) -- (0.5*\th:1.5*\r)  
	(0.5*\th:1.5*\r) -- (5*\th/6:1.5*\r) 
	%
	(90-0.5*\th:1*\r) -- (5*\th/6:1.5*\r) 
	(90:1*\r) -- (5*\th/6:1.5*\r) 
	(90:1*\r) -- (90-\th/12:1.5*\r) 
	(90-\th/12:1.5*\r) --  (5*\th/6:2*\r) 
	(-\th/6:2*\r) -- (\th/6:2*\r) 
	(\th/6:2*\r) -- (0.5*\th:2*\r) 
	(0.5*\th:2*\r) -- (5*\th/6:2*\r) 
	(0.5*\th:1.5*\r) -- (0.5*\th:2*\r) 
	(5*\th/6:2*\r) -- (0.5*\th:1.5*\r) 
	(0.5*\th:1.5*\r) -- (\th/6:2*\r) 
	(\th/6:1.5*\r) -- (\th/6:2*\r) 
	(\th/6:1.5*\r) -- (-\th/6:1.5*\r) 
	(-\th/6:1.5*\r) -- (-\th/6:2*\r) 
	(90-\th/6:2.5*\r) -- (42:2.5*\r) 
	(42:2.5*\r) -- (\th/12:2.5*\r) 
	(0.5*\th:2*\r) -- (42:2.5*\r) 
	(0.5*\th:2*\r) -- (\th/12:2.5*\r) 
	(\th/6:2*\r) -- (\th/12:2.5*\r) 
	(90-\th/6:2.5*\r) -- (90-\th/6:3*\r) 
	(42:2.5*\r) -- (90-\th/6:3*\r) 
	(42:2.5*\r) -- (\th/12:3*\r) 
;

\begin{scope}[red!50]

\fill
	(54:0.6*\r) circle (0.05*\r)
	(30:0.8*\r) circle (0.05*\r)
	(6:0.8*\r) circle (0.05*\r)
	(-6:1.1*\r) circle (0.05*\r)
	(6:1.3*\r) circle (0.05*\r)
	(-6:1.6*\r) circle (0.05*\r)
	(-18:1.6*\r) circle (0.05*\r)
	(-30:1.8*\r) circle (0.05*\r)
	(6:1.8*\r) circle (0.05*\r)
	(18:1.6*\r) circle (0.05*\r)
	(30:1.8*\r) circle (0.05*\r)
	(15:2.1*\r) circle (0.05*\r)
	(30:2.2*\r) circle (0.05*\r)
	(12:2.7*\r) circle (0.05*\r)
	(0:2.7*\r) circle (0.05*\r)
	(-30:2.8*\r) circle (0.05*\r);
	
\draw
	(-18:0.6*\r) -- (6:0.8*\r) -- (30:0.8*\r) -- (54:0.6*\r)
	(6:0.8*\r) -- (-6:1.1*\r) -- (6:1.3*\r) -- (-6:1.6*\r) -- (6:1.8*\r) -- (18:1.6*\r) -- (30:1.8*\r) -- (42:1.8*\r) -- (54:1.6*\r) -- (66:1.6*\r)
	(30:1.8*\r)  -- (15:2.1*\r) -- (30:2.2*\r) -- (12:2.7*\r) -- (0:2.7*\r) -- (-30:2.8*\r) -- (-60:2.7*\r);

\end{scope}

\end{scope}
}

\draw
	(0,0) circle (3*\r)
;

\end{scope}

\begin{scope}[xshift=4*\s cm]

	(0,0) circle (3.2*\r)
;

	(0,0) circle (3*\r)
;

	(90:0.5*\r) -- (90+\th:0.5*\r) -- (90+2*\th:0.5*\r) -- (90+3*\th:0.5*\r) -- (90+4*\th:0.5*\r)
;

\foreach \a in {0,...,4} {

	(90-\th:1*\r) -- (\th/6:1.5*\r) -- (0.5*\th:1.5*\r) -- (5*\th/6:1.5*\r)  -- (90-0.5*\th:1*\r) -- cycle
;

	(90-\th/12:2*\r) --  (5*\th/6:2*\r)  --  (0.5*\th:2*\r)  -- (42:2.5*\r) --  (90-\th/6:2.5*\r)  -- cycle
;

}

\foreach \a in {0,...,4}
{
\begin{scope}[rotate=\a*\th]

\draw
	(90:0.5*\r) -- (90-\th:0.5*\r)  
	(90:0.5*\r) -- (90-0.5*\th:1*\r) 
	(90-0.5*\th:1*\r) -- (90-\th:0.5*\r) 
	(90-\th:0.5*\r) -- (90-\th:1*\r) 
	(90-0.5*\th:1*\r) -- (90-\th:1*\r) 
	(90-\th:1*\r) -- (90-1.5*\th:1*\r) 
	(\th/6:1.5*\r) -- (0.5*\th:1.5*\r)  
	(0.5*\th:1.5*\r) -- (5*\th/6:1.5*\r) 
	%
	(90-0.5*\th:1*\r) -- (5*\th/6:1.5*\r) 
	(90:1*\r) -- (5*\th/6:1.5*\r) 
	(90:1*\r) -- (90-\th/12:1.5*\r) 
	(90-\th/12:1.5*\r) --  (5*\th/6:2*\r) 
	(-\th/6:2*\r) -- (\th/6:2*\r) 
	(\th/6:2*\r) -- (0.5*\th:2*\r) 
	(0.5*\th:2*\r) -- (5*\th/6:2*\r) 
	(0.5*\th:1.5*\r) -- (0.5*\th:2*\r) 
	(5*\th/6:2*\r) -- (0.5*\th:1.5*\r) 
	(0.5*\th:1.5*\r) -- (\th/6:2*\r) 
	(\th/6:1.5*\r) -- (\th/6:2*\r) 
	(\th/6:1.5*\r) -- (-\th/6:1.5*\r) 
	(-\th/6:1.5*\r) -- (-\th/6:2*\r) 
	(90-\th/6:2.5*\r) -- (42:2.5*\r) 
	(42:2.5*\r) -- (\th/12:2.5*\r) 
	(0.5*\th:2*\r) -- (42:2.5*\r) 
	(0.5*\th:2*\r) -- (\th/12:2.5*\r) 
	(\th/6:2*\r) -- (\th/12:2.5*\r) 
	(90-\th/6:2.5*\r) -- (90-\th/6:3*\r) 
	(42:2.5*\r) -- (90-\th/6:3*\r) 
	(42:2.5*\r) -- (\th/12:3*\r) 
;

\begin{scope}[red!50]

\fill
	(54:0.6*\r) circle (0.05*\r)
	(30:0.8*\r) circle (0.05*\r)
	(6:0.8*\r) circle (0.05*\r)
	(-6:1.1*\r) circle (0.05*\r)
	(6:1.3*\r) circle (0.05*\r)
	(-6:1.6*\r) circle (0.05*\r)
	(-18:1.6*\r) circle (0.05*\r)
	(-30:1.8*\r) circle (0.05*\r)
	(6:1.8*\r) circle (0.05*\r)
	(18:1.6*\r) circle (0.05*\r)
	(30:1.8*\r) circle (0.05*\r)
	(15:2.1*\r) circle (0.05*\r)
	(30:2.2*\r) circle (0.05*\r)
	(12:2.7*\r) circle (0.05*\r)
	(0:2.7*\r) circle (0.05*\r)
	(-30:2.8*\r) circle (0.05*\r);
	
\draw
	(-18:0.6*\r) -- (6:0.8*\r) -- (30:0.8*\r) -- (54:0.6*\r)
	(6:0.8*\r) -- (-6:1.1*\r) -- (6:1.3*\r) -- (-6:1.6*\r) -- (6:1.8*\r) -- (18:1.6*\r) -- (30:1.8*\r) -- (42:1.8*\r) -- (54:1.6*\r) -- (66:1.6*\r)
	(30:1.8*\r)  -- (15:2.1*\r) -- (30:2.2*\r) -- (12:2.7*\r) -- (0:2.7*\r) -- (-30:2.8*\r) -- (-60:2.7*\r);

\end{scope}

\end{scope}
}

\draw
	(0,0) circle (3*\r)
;

\foreach \a in {0,...,4} {

\fill[teal!80!blue,rotate=\a*\th] 
	(6+\th:0.8*\r) circle (0.05*\r)
	(-6+\th:1.1*\r) circle (0.05*\r)
	(6+\th:1.3*\r) circle (0.05*\r)
	(-6+\th:1.6*\r) circle (0.05*\r)
	(30+1*\th:1.8*\r) circle (0.05*\r)
	(15+1*\th:2.1*\r) circle (0.05*\r)
	(30+1*\th:2.2*\r) circle (0.05*\r) 
	(12+1*\th:2.7*\r) circle (0.05*\r)
;

\draw[teal!80!blue, line width=1.25, rotate=\a*\th]
	(6+\th:0.8*\r) -- (-6+\th:1.1*\r) 
	(6+\th:1.3*\r) -- (-6+\th:1.6*\r) 
	(30+1*\th:1.8*\r) -- (15+1*\th:2.1*\r) 
	(30+1*\th:2.2*\r) -- (12+1*\th:2.7*\r) 
;

}

\end{scope}

\begin{scope}[xshift=8*\s cm]

	(0,0) circle (3.2*\r)
;

	(0,0) circle (3*\r)
;

	(90:0.5*\r) -- (90+\th:0.5*\r) -- (90+2*\th:0.5*\r) -- (90+3*\th:0.5*\r) -- (90+4*\th:0.5*\r)
;

\foreach \a in {0,...,4} {

	(90-\th:1*\r) -- (\th/6:1.5*\r) -- (0.5*\th:1.5*\r) -- (5*\th/6:1.5*\r)  -- (90-0.5*\th:1*\r) -- cycle
;

	(90-\th/12:2*\r) --  (5*\th/6:2*\r)  --  (0.5*\th:2*\r)  -- (42:2.5*\r) --  (90-\th/6:2.5*\r)  -- cycle
;

}

\foreach \a in {0,...,4}
{
\begin{scope}[rotate=\a*\th]

\draw
	(90:0.5*\r) -- (90-\th:0.5*\r)  
	(90:0.5*\r) -- (90-0.5*\th:1*\r) 
	(90-0.5*\th:1*\r) -- (90-\th:0.5*\r) 
	(90-\th:0.5*\r) -- (90-\th:1*\r) 
	(90-0.5*\th:1*\r) -- (90-\th:1*\r) 
	(90-\th:1*\r) -- (90-1.5*\th:1*\r) 
	(\th/6:1.5*\r) -- (0.5*\th:1.5*\r)  
	(0.5*\th:1.5*\r) -- (5*\th/6:1.5*\r) 
	%
	(90-0.5*\th:1*\r) -- (5*\th/6:1.5*\r) 
	(90:1*\r) -- (5*\th/6:1.5*\r) 
	(90:1*\r) -- (90-\th/12:1.5*\r) 
	(90-\th/12:1.5*\r) --  (5*\th/6:2*\r) 
	(-\th/6:2*\r) -- (\th/6:2*\r) 
	(\th/6:2*\r) -- (0.5*\th:2*\r) 
	(0.5*\th:2*\r) -- (5*\th/6:2*\r) 
	(0.5*\th:1.5*\r) -- (0.5*\th:2*\r) 
	(5*\th/6:2*\r) -- (0.5*\th:1.5*\r) 
	(0.5*\th:1.5*\r) -- (\th/6:2*\r) 
	(\th/6:1.5*\r) -- (\th/6:2*\r) 
	(\th/6:1.5*\r) -- (-\th/6:1.5*\r) 
	(-\th/6:1.5*\r) -- (-\th/6:2*\r) 
	(90-\th/6:2.5*\r) -- (42:2.5*\r) 
	(42:2.5*\r) -- (\th/12:2.5*\r) 
	(0.5*\th:2*\r) -- (42:2.5*\r) 
	(0.5*\th:2*\r) -- (\th/12:2.5*\r) 
	(\th/6:2*\r) -- (\th/12:2.5*\r) 
	(90-\th/6:2.5*\r) -- (90-\th/6:3*\r) 
	(42:2.5*\r) -- (90-\th/6:3*\r) 
	(42:2.5*\r) -- (\th/12:3*\r) 
;

\begin{scope}[red!50]

\fill
	(54:0.6*\r) circle (0.05*\r)
	(30:0.8*\r) circle (0.05*\r)
	(6:0.8*\r) circle (0.05*\r)
	(-6:1.1*\r) circle (0.05*\r)
	(6:1.3*\r) circle (0.05*\r)
	(-6:1.6*\r) circle (0.05*\r)
	(-18:1.6*\r) circle (0.05*\r)
	(-30:1.8*\r) circle (0.05*\r)
	(6:1.8*\r) circle (0.05*\r)
	(18:1.6*\r) circle (0.05*\r)
	(30:1.8*\r) circle (0.05*\r)
	(15:2.1*\r) circle (0.05*\r)
	(30:2.2*\r) circle (0.05*\r)
	(12:2.7*\r) circle (0.05*\r)
	(0:2.7*\r) circle (0.05*\r)
	(-30:2.8*\r) circle (0.05*\r);
	
\draw
	(-18:0.6*\r) -- (6:0.8*\r) -- (30:0.8*\r) -- (54:0.6*\r)
	(6:0.8*\r) -- (-6:1.1*\r) -- (6:1.3*\r) -- (-6:1.6*\r) -- (6:1.8*\r) -- (18:1.6*\r) -- (30:1.8*\r) -- (42:1.8*\r) -- (54:1.6*\r) -- (66:1.6*\r)
	(30:1.8*\r)  -- (15:2.1*\r) -- (30:2.2*\r) -- (12:2.7*\r) -- (0:2.7*\r) -- (-30:2.8*\r) -- (-60:2.7*\r);

\end{scope}

\end{scope}
}

\draw
	(0,0) circle (3*\r)
;

\foreach \a in {-1,0,1} {

\fill[teal!80!blue,rotate=\a*\th] 
	(6+\th:0.8*\r) circle (0.05*\r)
	(-6+\th:1.1*\r) circle (0.05*\r)
	(6+\th:1.3*\r) circle (0.05*\r)
	(-6+\th:1.6*\r) circle (0.05*\r)
	(30+1*\th:1.8*\r) circle (0.05*\r)
	(15+1*\th:2.1*\r) circle (0.05*\r)
	(30+1*\th:2.2*\r) circle (0.05*\r) 
	(12+1*\th:2.7*\r) circle (0.05*\r)
;

\draw[teal!80!blue, line width=1.25, rotate=\a*\th]
	(6+\th:0.8*\r) -- (-6+\th:1.1*\r) 
	(6+\th:1.3*\r) -- (-6+\th:1.6*\r) 
	(30+1*\th:1.8*\r) -- (15+1*\th:2.1*\r) 
	(30+1*\th:2.2*\r) -- (12+1*\th:2.7*\r) 
;

}

\draw[teal!80!blue, line width=1.25]
	(6+3*\th:0.8*\r) -- (30+3*\th:0.8*\r)
	(54+3*\th:0.6*\r) -- (6+4*\th:0.8*\r)
	(-6+3*\th:1.6*\r) -- (6+3*\th:1.8*\r) 
	(18+3*\th:1.6*\r) -- (30+3*\th:1.8*\r)
	(-6+4*\th:1.6*\r) -- (6+4*\th:1.8*\r)
	(18+4*\th:1.6*\r) -- (30+4*\th:1.8*\r)
	(0+4*\th:2.7*\r) -- (12+4*\th:2.7*\r)
	(-30+4*\th:2.8*\r) -- (-60+4*\th:2.7*\r)
;

\fill[teal!80!blue] 
	(6+3*\th:0.8*\r) circle (0.05*\r)
	(30+3*\th:0.8*\r) circle (0.05*\r)
	(54+3*\th:0.6*\r) circle (0.05*\r)
	(6+4*\th:0.8*\r) circle (0.05*\r)
	(-6+4*\th:1.6*\r) circle (0.05*\r)
	(6+4*\th:1.8*\r) circle (0.05*\r)
	(18+4*\th:1.6*\r) circle (0.05*\r)
	(30+4*\th:1.8*\r) circle (0.05*\r)
	(-6+3*\th:1.6*\r) circle (0.05*\r)
	(6+3*\th:1.8*\r) circle (0.05*\r)
	(18+3*\th:1.6*\r) circle (0.05*\r)
	(30+3*\th:1.8*\r) circle (0.05*\r)
	(0+4*\th:2.7*\r) circle (0.05*\r)
	(12+4*\th:2.7*\r) circle (0.05*\r)
	(-30+4*\th:2.8*\r) circle (0.05*\r) 
	(-60+4*\th:2.7*\r) circle (0.05*\r)
;

\end{scope} 

\begin{scope}[xshift=12*\s cm]

	(0,0) circle (3.2*\r)
;

	(0,0) circle (3*\r)
;

	(90:0.5*\r) -- (90+\th:0.5*\r) -- (90+2*\th:0.5*\r) -- (90+3*\th:0.5*\r) -- (90+4*\th:0.5*\r)
;

\foreach \a in {0,...,4} {

	(90-\th:1*\r) -- (\th/6:1.5*\r) -- (0.5*\th:1.5*\r) -- (5*\th/6:1.5*\r)  -- (90-0.5*\th:1*\r) -- cycle
;

	(90-\th/12:2*\r) --  (5*\th/6:2*\r)  --  (0.5*\th:2*\r)  -- (42:2.5*\r) --  (90-\th/6:2.5*\r)  -- cycle
;

}

\foreach \a in {0,...,4}
{
\begin{scope}[rotate=\a*\th]

\draw
	(90:0.5*\r) -- (90-\th:0.5*\r)  
	(90:0.5*\r) -- (90-0.5*\th:1*\r) 
	(90-0.5*\th:1*\r) -- (90-\th:0.5*\r) 
	(90-\th:0.5*\r) -- (90-\th:1*\r) 
	(90-0.5*\th:1*\r) -- (90-\th:1*\r) 
	(90-\th:1*\r) -- (90-1.5*\th:1*\r) 
	(\th/6:1.5*\r) -- (0.5*\th:1.5*\r)  
	(0.5*\th:1.5*\r) -- (5*\th/6:1.5*\r) 
	%
	(90-0.5*\th:1*\r) -- (5*\th/6:1.5*\r) 
	(90:1*\r) -- (5*\th/6:1.5*\r) 
	(90:1*\r) -- (90-\th/12:1.5*\r) 
	(90-\th/12:1.5*\r) --  (5*\th/6:2*\r) 
	(-\th/6:2*\r) -- (\th/6:2*\r) 
	(\th/6:2*\r) -- (0.5*\th:2*\r) 
	(0.5*\th:2*\r) -- (5*\th/6:2*\r) 
	(0.5*\th:1.5*\r) -- (0.5*\th:2*\r) 
	(5*\th/6:2*\r) -- (0.5*\th:1.5*\r) 
	(0.5*\th:1.5*\r) -- (\th/6:2*\r) 
	(\th/6:1.5*\r) -- (\th/6:2*\r) 
	(\th/6:1.5*\r) -- (-\th/6:1.5*\r) 
	(-\th/6:1.5*\r) -- (-\th/6:2*\r) 
	(90-\th/6:2.5*\r) -- (42:2.5*\r) 
	(42:2.5*\r) -- (\th/12:2.5*\r) 
	(0.5*\th:2*\r) -- (42:2.5*\r) 
	(0.5*\th:2*\r) -- (\th/12:2.5*\r) 
	(\th/6:2*\r) -- (\th/12:2.5*\r) 
	(90-\th/6:2.5*\r) -- (90-\th/6:3*\r) 
	(42:2.5*\r) -- (90-\th/6:3*\r) 
	(42:2.5*\r) -- (\th/12:3*\r) 
;

\begin{scope}[red!50]

\fill
	(54:0.6*\r) circle (0.05*\r)
	(30:0.8*\r) circle (0.05*\r)
	(6:0.8*\r) circle (0.05*\r)
	(-6:1.1*\r) circle (0.05*\r)
	(6:1.3*\r) circle (0.05*\r)
	(-6:1.6*\r) circle (0.05*\r)
	(-18:1.6*\r) circle (0.05*\r)
	(-30:1.8*\r) circle (0.05*\r)
	(6:1.8*\r) circle (0.05*\r)
	(18:1.6*\r) circle (0.05*\r)
	(30:1.8*\r) circle (0.05*\r)
	(15:2.1*\r) circle (0.05*\r)
	(30:2.2*\r) circle (0.05*\r)
	(12:2.7*\r) circle (0.05*\r)
	(0:2.7*\r) circle (0.05*\r)
	(-30:2.8*\r) circle (0.05*\r);
	
\draw
	(-18:0.6*\r) -- (6:0.8*\r) -- (30:0.8*\r) -- (54:0.6*\r)
	(6:0.8*\r) -- (-6:1.1*\r) -- (6:1.3*\r) -- (-6:1.6*\r) -- (6:1.8*\r) -- (18:1.6*\r) -- (30:1.8*\r) -- (42:1.8*\r) -- (54:1.6*\r) -- (66:1.6*\r)
	(30:1.8*\r)  -- (15:2.1*\r) -- (30:2.2*\r) -- (12:2.7*\r) -- (0:2.7*\r) -- (-30:2.8*\r) -- (-60:2.7*\r);

\end{scope}

\end{scope}
}

\draw
	(0,0) circle (3*\r)
;

\foreach \a in {-1,0,1} {

\fill[teal!80!blue,rotate=\a*\th] 
	(6+\th:0.8*\r) circle (0.05*\r)
	(-6+\th:1.1*\r) circle (0.05*\r)
	(6+\th:1.3*\r) circle (0.05*\r)
	(-6+\th:1.6*\r) circle (0.05*\r)
	(30:1.8*\r) circle (0.05*\r)
	(15:2.1*\r) circle (0.05*\r)
	(30:2.2*\r) circle (0.05*\r) 
	(12:2.7*\r) circle (0.05*\r)
;

\draw[teal!80!blue, line width=1.25, rotate=\a*\th]
	(6+\th:0.8*\r) -- (-6+\th:1.1*\r) 
	(6+\th:1.3*\r) -- (-6+\th:1.6*\r) 
	(30:1.8*\r) -- (15:2.1*\r) 
	(30:2.2*\r) -- (12:2.7*\r) 
;

}

\draw[teal!80!blue, line width=1.25]
	(6+3*\th:0.8*\r) -- (30+3*\th:0.8*\r)
	(54+3*\th:0.6*\r) -- (6+4*\th:0.8*\r)
	(-6+3*\th:1.6*\r) -- (54+2*\th:1.6*\r)
	(42+2*\th:1.8*\r) -- (30+2*\th:1.8*\r) 
	(-6+4*\th:1.6*\r) -- (54+3*\th:1.6*\r)
	(42+3*\th:1.8*\r) -- (30+3*\th:1.8*\r) 
	(0+3*\th:2.7*\r) -- (12+3*\th:2.7*\r)
	(-30+3*\th:2.8*\r) -- (-60+3*\th:2.7*\r)
;

\fill[teal!80!blue] 
	(6+3*\th:0.8*\r) circle (0.05*\r)
	(30+3*\th:0.8*\r) circle (0.05*\r)
	(54+3*\th:0.6*\r) circle (0.05*\r)
	(6+4*\th:0.8*\r) circle (0.05*\r)
	(-6+4*\th:1.6*\r) circle (0.05*\r)
	(54+3*\th:1.6*\r) circle (0.05*\r)
	(42+3*\th:1.8*\r) circle (0.05*\r)
	(30+3*\th:1.8*\r) circle (0.05*\r)
	(-6+3*\th:1.6*\r) circle (0.05*\r)
	(54+2*\th:1.6*\r) circle (0.05*\r)
	(42+2*\th:1.8*\r) circle (0.05*\r)
	(30+2*\th:1.8*\r) circle (0.05*\r)
	(0+3*\th:2.7*\r) circle (0.05*\r)
	(12+3*\th:2.7*\r) circle (0.05*\r)
	(-30+3*\th:2.8*\r) circle (0.05*\r) 
	(-60+3*\th:2.7*\r) circle (0.05*\r)
;

\end{scope}

\end{scope}

\begin{scope}[yshift=-3.35*\s cm] 

\tikzmath{
\r=0.35;
\th=360/5;
\x=\r*cos(\th/2);
}

\begin{scope}[]

\foreach \a in {0,...,4}{

\draw[red!50, rotate=\a*\th]
	(90:\r) -- (90-\th:\r)
	(90:\r) -- (90:2*\r)
	(90:2*\r) -- (90-0.5*\th:\x+2*\r) 
	(90:2*\r) -- (90+0.5*\th:\x+2*\r)
	(90-0.5*\th:\x+2*\r) -- (90-0.5*\th:\x+3*\r)
	(90-0.5*\th:\x+3*\r) --(90+0.5*\th:\x+3*\r)
;

}

\foreach \aa in {-1, 1} {

	(90-\th:\r) -- (90-\th:2*\r)
	(90-2*\th:\r) -- (90+2*\th:\r)
	(90-0.5*\th:\x+2*\r) -- (90-0.5*\th:\x+3*\r) 
;

}

	(90:\r) -- (90:2*\r)
	(90-2*\th:2*\r) -- (90-1.5*\th:\x+2*\r)
	(90+2*\th:2*\r) -- (270:\x+2*\r) 
	(90+1.5*\th:\x+2*\r) -- (90+1.5*\th:\x+3*\r)
	(90-1.5*\th:\x+3*\r) -- (270:\x+3*\r) 
;

\end{scope}

\begin{scope}[xshift=4*\s cm]

\foreach \a in {0,...,4}{

\draw[red!50, rotate=\a*\th]
	(90:\r) -- (90-\th:\r)
	(90:\r) -- (90:2*\r)
	(90:2*\r) -- (90-0.5*\th:\x+2*\r) 
	(90:2*\r) -- (90+0.5*\th:\x+2*\r)
	(90-0.5*\th:\x+2*\r) -- (90-0.5*\th:\x+3*\r)
	(90-0.5*\th:\x+3*\r) --(90+0.5*\th:\x+3*\r)
;

}

\foreach \a in {0,...,4} {

\draw[teal!80!blue, line width=1.5, rotate=\a*\th]
	(90:\r) -- (90:2*\r)
	(90-0.5*\th:\x+2*\r) -- (90-0.5*\th:\x+3*\r)
;

}

\end{scope}

\begin{scope}[xshift=8*\s cm]

\foreach \a in {0,...,4}{

\draw[red!50, rotate=\a*\th]
	(90:\r) -- (90-\th:\r)
	(90:\r) -- (90:2*\r)
	(90:2*\r) -- (90-0.5*\th:\x+2*\r) 
	(90:2*\r) -- (90+0.5*\th:\x+2*\r)
	(90-0.5*\th:\x+2*\r) -- (90-0.5*\th:\x+3*\r)
	(90-0.5*\th:\x+3*\r) --(90+0.5*\th:\x+3*\r)
;

}

\foreach \aa in {-1, 1} {

\draw[teal!80!blue, xscale=\aa, line width=1.5]
	(90-\th:\r) -- (90-\th:2*\r)
	(90-2*\th:\r) -- (90+2*\th:\r)
	(90-0.5*\th:\x+2*\r) -- (90-0.5*\th:\x+3*\r) 
;

}

\draw[teal!80!blue, line width=1.5]
	(90:\r) -- (90:2*\r)
	(90-2*\th:2*\r) -- (90-1.5*\th:\x+2*\r)
	(90+2*\th:2*\r) -- (270:\x+2*\r) 
	(90+1.5*\th:\x+2*\r) -- (90+1.5*\th:\x+3*\r)
	(90-1.5*\th:\x+3*\r) -- (270:\x+3*\r) 
;

\end{scope}

\begin{scope}[xshift=12*\s cm]

\foreach \a in {0,...,4}{

\draw[red!50, rotate=\a*\th]
	(90:\r) -- (90-\th:\r)
	(90:\r) -- (90:2*\r)
	(90:2*\r) -- (90-0.5*\th:\x+2*\r) 
	(90:2*\r) -- (90+0.5*\th:\x+2*\r)
	(90-0.5*\th:\x+2*\r) -- (90-0.5*\th:\x+3*\r)
	(90-0.5*\th:\x+3*\r) --(90+0.5*\th:\x+3*\r)
;

}

\foreach \aa in {-1, 1} {

\draw[teal!80!blue, xscale=\aa, line width=1.5]
	(90-\th:\r) -- (90-\th:2*\r)
	(90-2*\th:\r) -- (90+2*\th:\r)
	(90-0.5*\th:\x+2*\r) -- (90-0.5*\th:\x+3*\r) 
;

}

\draw[teal!80!blue, line width=1.5]
	(90:\r) -- (90:2*\r)
	(90+2*\th:2*\r) -- (90+1.5*\th:\x+2*\r)
	(90-2*\th:2*\r) -- (270:\x+2*\r) 
	(90-1.5*\th:\x+2*\r) -- (90-1.5*\th:\x+3*\r)
	(90+1.5*\th:\x+3*\r) -- (270:\x+3*\r) 
;

\end{scope}

\end{scope}

\end{tikzpicture}
\caption{The correspondence between the dihedral tilings and the perfect matchings of the dodecahedron, \quotes{\textbf{\fontsize{25}{15}\selectfont \textcolor{teal!80!blue}{-}\textcolor{red!50}{-}\textcolor{teal!80!blue}{-}}} and \quotes{\textbf{\fontsize{25}{15}\selectfont \textcolor{red!50}{-}\textcolor{teal!80!blue}{-}\textcolor{red!50}{-}}} in the first row respectively corresponds to \quotes{\textbf{\fontsize{25}{15}\selectfont \textcolor{teal!80!blue}{-}}} and \quotes{\textbf{\fontsize{25}{15}\selectfont \textcolor{red!50}{-}}} in the second row}
\label{Fig-a5-a4-dodeca-matching}
\end{figure}

We sum up that the dihedral tilings with $\alpha\beta^2$ are the three tilings with underlying tiling of the snub dodecahedron as illustrated in the first picture of Figure \ref{Fig-a5-a4-dodeca-matching}. The underlying network with vertices at \quotes{the centre} of each triangle corresponds to the graph of the dodecahedron with each edge subdivided into three. As a quadrilateral is formed via deleting the common edge between two triangles, we highlight the adjacent pairs of such edges where the two non-adjacent end vertices of the \quotes{thick} edges have degree $3$. 

As seen in the second, third, and the fourth picture, such a pair uniquely corresponds to an edge of the dodecahedron and the collection of all such pairs form a perfect matching in the dodecahedron (see the second row). Note that an unmarked edges in the dodecahedron correspond to three consecutive edges in the network (two end vertices have degree $3$) where the middle segment is to be deleted. This means that the three tilings can be derived from the perfect matching of the dodecahedron.
\end{proof}

\begin{prop}\label{Prop-a5-a4-al2be} There is no dihedral tiling with vertex $\alpha^2\beta$.
\end{prop}

\begin{proof} By Propositions \ref{Prop-a5-a4-albega}, \ref{Prop-a5-a4-al3}, \ref{Prop-a5-a4-be3}, \ref{Prop-a5-a4-albe2}, we may assume that the only degree $3$ vertices are $\alpha^2\beta, \beta^2\gamma$. Suppose both $\alpha^2\beta, \beta^2\gamma$ are vertices. By $\beta>\alpha>\gamma$ and $\alpha^2\beta$, we get $\beta>\frac{2}{3}\pi > \alpha$. Meanwhile, the two vertices imply
\begin{align*}
\beta =2\pi - 2\alpha, \quad
\gamma = 4\alpha- 2\pi.
\end{align*}
Substituting the above into \eqref{Eq-cot-al-be-ga}, one can show that there is no solution for $\alpha \in (0, \pi)$. 

So $\alpha^2\beta$ being a vertex implies that $\beta^2\gamma$ is not a vertex. Combined with $\beta>\alpha>\gamma$ and the rhombus angle sum, we have $R(\beta^2)<R(\alpha\beta)=\alpha$ and $R(\beta^2)<2\gamma$. Then $\beta^2\cdots$ is not a vertex. By no $\beta^2\cdots$, we also know that $\gamma \vert \gamma\cdots$ is not a vertex. So $\gamma^c, \alpha\gamma^c, \beta\gamma^c$ are not vertices. 

On the other hand, by $\alpha^2\beta$ and $\beta>\alpha>\gamma$, we get $\beta>\frac{2}{3}\pi>\alpha$. Then $\alpha>\frac{3}{5}\pi$ and $\beta>\frac{2}{3}\pi$ and no $\gamma^c, \alpha\gamma^c, \beta\gamma^c, \beta^2\cdots$ give the vertices 
\begin{align*}
\AVC = \{ \alpha^2\beta, \alpha^{3}\gamma^c, \alpha^{2}\gamma^c, \alpha\beta\gamma^c \}.
\end{align*}
By no $\gamma\vert\gamma\cdots$, we have $\alpha^3\gamma^c=\vert\alpha\vert\alpha\vert\alpha\vert\gamma\vert, \vert\alpha\vert\alpha\vert \gamma \vert\alpha\vert\gamma\vert, \vert\alpha\vert\gamma\vert\alpha\vert\gamma\vert\alpha\vert \gamma \vert =\alpha^3\gamma, \alpha^3\gamma^2$, $\alpha^3\gamma^3$ and $\alpha^2\gamma^c=\vert\alpha\vert\gamma\vert\alpha\vert\gamma\vert$ and $\alpha\beta\gamma^c=\vert\alpha\vert\gamma\vert\beta\vert\gamma\vert=\alpha\beta\gamma^2$. Hence the $\AVC$ is updated as follows,
\begin{align*}
\AVC = \{ \alpha^2\beta, \alpha^{3}\gamma, \alpha^{3}\gamma^2, \alpha^{3}\gamma^3, \alpha^{2}\gamma^2, \alpha\beta\gamma^2 \}.
\end{align*}
Then we have $\alpha\gamma\cdots=\alpha^{3}\gamma, \alpha^{3}\gamma^2, \alpha^{3}\gamma^3, \alpha^{2}\gamma^2, \alpha\beta\gamma^2$. Combining one of them with $\alpha^2\beta$, we get the angle formulae in terms of $\alpha$. Substituting the angle formulae into \eqref{Eq-cot-al-be-ga}, the only two vertices with solutions for $\alpha \in (\frac{3}{5}\pi, \frac{2}{3}\pi)$ are $\alpha^2\gamma^2, \alpha\beta\gamma^2$. By $\beta>\alpha>\gamma$, the vertices $\alpha^2\gamma^2, \alpha\beta\gamma^2$ are mutually exclusive. Hence we get the updated vertices below,
\begin{align*}
\AVC &= \{ \alpha^2\beta, \alpha^{2}\gamma^2  \}; \\
\AVC &= \{ \alpha^2\beta, \alpha\beta\gamma^2 \}.
\end{align*}
For both $\AVC$'s, by no $\gamma \vert \gamma\cdots$, we have $\alpha^2\gamma^2 = \vert \alpha \vert \gamma \vert \alpha \vert \gamma \vert$ and $\alpha\beta\gamma^2 = \vert \alpha \vert \gamma \vert \beta \vert \gamma \vert$. It implies $\alpha\vert \alpha \cdots, \alpha\vert\beta\cdots=\alpha^2\beta$ and both $\alpha^2\gamma^2, \alpha\beta\gamma^2$ have angle arrangement $\gamma \vert \alpha \vert \gamma$.

The angle arrangement $\gamma \vert \alpha \vert \gamma$ determines tiles $T_1, T_2, T_3$ in Figure \ref{Fig-a5-a4-AAD-al2be-al2ga2/albega2-alal/gaalga}. Then $\alpha_2\beta_1\cdots=\alpha^2\beta$ determines the angles in $T_4$. By mirror symmetry, we also get the angles in $T_5$. By $\alpha_2 \vert \alpha_4\cdots, \alpha_2 \vert \alpha_5\cdots=\alpha^2\beta$, we get two adjacent $\beta$'s in $T_6$, a contradiction. 

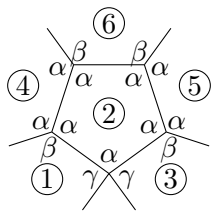
\begin{figure}[h!] 
\centering
\begin{tikzpicture}

\tikzmath{
\s=1;
\r=0.8;
\g=5;
\ph=360/\g;
\x=\r*cos(\ph/2);
\y=\r*sin(\ph/2);
\rr=2*\y/sqrt(2);
\h=4;
\th=360/\h;
\xx=\rr*cos(\th/2);
}

\begin{scope}[] 

\foreach \p in {0,...,4} {

\draw[rotate=\p*\ph]
	(270:\r) -- (270-1*\ph:\r) 
;

\node at (270+\p*\ph:0.7*\r) {\small $\alpha$};

}

\foreach \p in {1,...,4} {

\draw[rotate=\p*\ph]
	(270:\r) -- (270:1.75*\r)
;

}

\draw[]
	(270:\r) -- ([shift={(270:\r)}]270-0.5*\ph: 0.75*\r)
	(270:\r) -- ([shift={(270:\r)}]270+0.5*\ph: 0.75*\r)
;

\foreach \a in {-1, 1} {

\tikzset{xscale=\a}

\node at (1*\r, -0.6*\r) {\small $\beta$};
\node at (0.3*\r, -1.1*\r) {\small $\gamma$};

\node at (1.125*\r, -0.15*\r) {\small $\alpha$};
\node at (0.85*\r, 0.75*\r) {\small $\alpha$};

\node at (0.5*\r, 1*\r) {\small $\beta$};

}

\node[inner sep=1,draw,shape=circle] at (0,0) {\small $2$};
\node[inner sep=1,draw,shape=circle] at (270-0.6*\ph:1.5*\r) {\small $1$};
\node[inner sep=1,draw,shape=circle] at (270+0.6*\ph:1.5*\r) {\small $3$};
\node[inner sep=1,draw,shape=circle] at (90+1*\ph:1.5*\r) {\small $4$};
\node[inner sep=1,draw,shape=circle] at (90-1*\ph:1.5*\r) {\small $5$};
\node[inner sep=1,draw,shape=circle] at (90:1.5*\r) {\small $6$};

\end{scope} 

\end{tikzpicture}
\caption{The angle arrangement $ \gamma \vert \alpha\vert \gamma $}
\label{Fig-a5-a4-AAD-al2be-al2ga2/albega2-alal/gaalga}
\end{figure}



Hence there is no dihedral tiling for $\AVC = \{ \alpha^2\beta, \alpha^2\gamma^2 \}, \{ \alpha^2\beta, \alpha\beta\gamma^2 \}$.
\end{proof}

\begin{prop}\label{Prop-a5-a4-be2ga} The infinite family of dihedral tilings with vertex $\beta^2\gamma$ is in Figure \ref{Fig-a5-a4-Tilings-be2ga-albegac}.
\end{prop}

Each member of the infinite family has $2$ pentagons and $4(2c-1)$ rhombi, where $c\ge2$ is the number of $\gamma$'s in the vertex $\alpha\beta\gamma^c$. They are illustrated in Figure \ref{Fig-a5-a4-EMTs}.

\begin{proof} By Proposition \ref{Prop-a5-a4-al2be}, we know that $\alpha^2\beta$ is not a vertex. As $\gamma$ is the smallest angle, by $\beta^2\gamma$ we get $\beta^2\cdots=\beta^2\gamma$ and $\beta>\frac{2}{3}\pi$. By $\alpha>\frac{3}{5}\pi$ and $\beta>\frac{2}{3}\pi$ and $\beta>\alpha>\gamma$, we get all the vertices below,
\begin{align}\label{Eq-AVC-be2ga-full}
\AVC = \{ \beta^2\gamma, \gamma^c, \alpha^{a\le3}\gamma^c, \beta\gamma^c, \alpha\beta\gamma^c \}.
\end{align}
From the above, we know that $\alpha^2\beta\cdots$ is not a vertex and $\alpha\beta\cdots = \alpha\beta\gamma^c$. Lemma \ref{Lem-albe-alga} implies that $\alpha\beta\gamma^c$ is a vertex.

We will first show that $\gamma^c, \alpha^a\gamma^c$ are not vertices for dihedral tilings. 

By $\alpha\beta\cdots=\alpha\beta\gamma^c$, the same deduction on $\alpha\vert\gamma\cdots$ in Figure \ref{Fig-a5-a4-AAD-albegac-alga} determines $\alpha\vert\gamma\cdots = \beta \vert \alpha \vert \gamma\cdots =\alpha\beta\gamma^c$. By $\alpha\vert\gamma\cdots = \beta \vert \alpha \vert \gamma\cdots$, we further know that $\alpha^a\gamma^c$ is not a vertex.

If $\gamma^c$ is a vertex, then by $\beta^2\cdots=\beta^2\gamma$, the angle arrangement $\gamma\vert\gamma$ determines three tiles as illustrated in Figure \ref{Fig-a4-be2-gac-Timezones}. This process continues at $\gamma^c$ and we determine a monohedral earth map tiling for each fixed $c\ge3$. Further details about the earth map tilings can be seen in \cite{cly}. Hence $\gamma^c$ is not a vertex for dihedral tilings.

\begin{figure}[h!]
\centering
\begin{tikzpicture}[>=latex,scale=1]

\tikzmath{ 
\s=1;
\r=0.8;
\x = sqrt(2*\r^2);
\y = sqrt(\r^2 - (\x/2)^2);
\yy = 2*\r + \y;
\l=2*\r*sin(67.5);
\L=sqrt(\r^2+\l^2-2*\r*\l*cos(135+22.5));
\A=acos((\L^2+\l^2-\r^2)/(2*\L*\l));
}

\begin{scope}[xshift=0 cm] 

	(\x,0) -- (\x,\r) -- (90-22.5:\l) -- (90-22.5+\A:\L) -- (-0.5*\x, 2*\r+\y) -- (-0.5*\x, \r+\y) -- (0*\x, \r) -- (0,0) -- (\x, 0)
;

\tikzmath{ 
\tz=1;
\tzz=\tz-1;
\tzzz=\tzz-1;
}

\draw[]
	(90:\r) -- (90+22.5:\l)
	(90+22.5:\l) -- (90+22.5-\A:\L)
	(1*\x,0) -- (1*\x, \r)
;

\foreach \a in {0,...,\tz} {

\draw[xshift=\x*\a cm]
	(0:0) -- (90:\r)
	(90:\r) -- (90-22.5:\l)
	(90-22.5:\l) -- (90-22.5+\A:\L)
;
	
}

\foreach \a in {0,...,\tzz} {

\draw[xshift=\x*\a cm]
	(90-22.5:\l) -- (\x,\r)
;

}

\foreach \c in {-1,0,...,\tzz} {

\node at (1*\x+\c*\x, 2*\r + 0.75*\y) {\small $\gamma$}; 
\node at (1*\x+\c*\x, 1*\r + 0.5*\y) {\small $\gamma$}; 
\node at (0.65*\x+\c*\x, 1.1*\r + 1*\y) {\small $\beta$}; 
\node at (1.35*\x+\c*\x, 1.1*\r + 1*\y) {\small $\beta$}; 

}

\foreach \c in {0,...,\tzz} {

\node at (0.5*\x+\c*\x, 0.1*\r) {\small $\gamma$}; 
\node at (0.5*\x+\c*\x, 1*\r + 0.5*\y) {\small $\gamma$}; 
\node at (0.15*\x+\c*\x, 0.9*\r) {\small $\beta$}; 
\node at (0.85*\x+\c*\x, 0.9*\r) {\small $\beta$}; 

}

\node at (\y,-0.75*\r) {\small $\mathcal{T}$};

\node[inner sep=1,draw,shape=circle] at (0,\r+0.75*\x) {\small $1$};
\node[inner sep=1,draw,shape=circle] at (\x,\r+0.75*\x) {\small $2$};
\node[inner sep=1,draw,shape=circle] at (0.5*\x,0.5*\x) {\small $3$};

\end{scope} 

\end{tikzpicture}
\caption{A timezone given by $T_1, T_3$ and $\mathcal{T}=c-\frac{1}{2}$ timezones, $c=2$}
\label{Fig-a4-be2-gac-Timezones}
\end{figure}
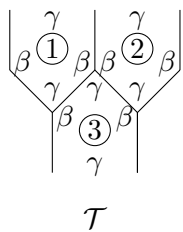

For $c\ge2$, we use $\mathcal{T}$ to denote a block of $c-\frac{1}{2}$ timezones which consist of $c-1$ timezones and one extra tile. Examples with $c=2,3$ are illustrated in Figure \ref{Fig-a4-be2-gac-Timezones} and the first picture of Figure \ref{Fig-a5-a4-EMTs} respectively.

By no $\gamma^c, \alpha^a\gamma^c$, we get
\begin{align}\label{Eq-AVC-be2ga-begac-albegac}
\AVC = \{ \beta^2\gamma, \beta\gamma^c, \alpha\beta\gamma^c \}.
\end{align} 
Since $\alpha$ appears at some vertex, we know that $\alpha\beta\gamma^c$ is a vertex. By the deduction in Figure \ref{Fig-a5-a4-AAD-albegac-alga}, we get the central square in Figure \ref{Fig-a5-a4-Tilings-be2ga-albegac}. For any fixed $c\ge2$, we can fill the remaining $\gamma$'s of $\alpha\beta\gamma^c$ at the vertices $\alpha_1\beta_2\gamma_3\cdots, \alpha_1\beta_3\gamma_4\cdots$, $\alpha_1\beta_4\gamma_5\cdots, \alpha_1\beta_5\gamma_6\cdots$, $\alpha_1\beta_6\gamma_2\cdots$ of Figure \ref{Fig-a5-a4-AAD-albegac-alga}. Then the $\gamma^c$-part of each $\alpha\beta\gamma^c$ determine $\mathcal{T}$ in Figure \ref{Fig-a5-a4-Tilings-be2ga-albegac}. The tiling is then completed by another square in the exterior of the \quotes{circular boundary} in Figure \ref{Fig-a5-a4-Tilings-be2ga-albegac}.

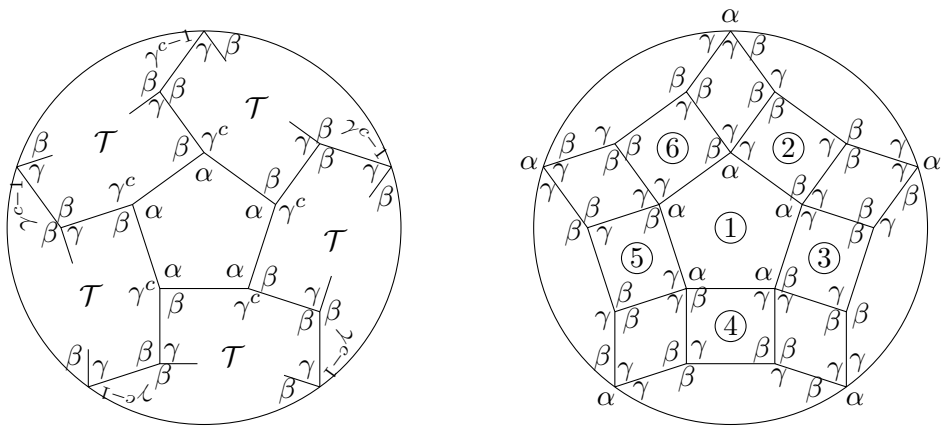
\begin{figure}[h!] 
\centering
\begin{tikzpicture}

\tikzmath{
\s=1;
\r=1;
\g=5;
\ph=360/\g;
\x=\r*cos(\ph/2);
\y=\r*sin(\ph/2);
\rr=2*\y/sqrt(2);
\h=4;
\th=360/\h;
\xx=\rr*cos(\th/2);
\R=sqrt(2*\r^2 - 2*\r^2*cos(180-\ph/2));
\bR = \ph/4;
\l = \r*cos(\ph/2);
}

\begin{scope}

\foreach \p in {0,...,4} {

\draw[rotate=\p*\ph]
	(90-1*\ph:\r) -- (90:\r)
	(90:\r) -- (90+\bR:\R)
	(90+\bR:\R) -- (90:\r+2*\l)
	(90+\bR+2*\ph:\R) -- ([shift={(90+\bR+2*\ph:\R)}]0:\r/2)
	(90:\r+2*\l) -- ([shift={(90:\r+2*\l)}]270+\ph/2:\r/2)
;

\node at (90+\p*\ph:0.7*\r) {\small $\alpha$};

\node at (90-0.1*\ph+\p*\ph:1.2*\r) {\small $\gamma^c$};
\node at (90-0.775*\ph+\p*\ph:1.1*\r) {\small $\beta$};
\node at (90+0.6*\bR+\p*\ph:0.975*\R) {\small $\beta$};
\node at (90-0.7*\ph+\p*\ph:0.9*\R) {\small $\gamma$};

\node at (90+1.15*\bR+\p*\ph:1.08*\R) {\small $\beta$};

\node at (90+\p*\ph: \r+1.65*\l) {\small $\gamma$};
\node at (90-0.125*\ph+\p*\ph: \r+1.78*\l) {\small $\beta$};

\node[rotate=\ph/4+\p*\ph] at (90+0.15*\ph+\p*\ph:\r+1.75*\l) {\footnotesize $\gamma^{c-1}$};

\node at (90-0.325*\ph+\p*\ph:1.75*\r) {\small $\mathcal{T}$};

}

\draw (0,0) circle (\r+2*\l);

\end{scope}

\begin{scope}[xshift=7*\s cm]

\foreach \p in {0,...,4} {

\draw[rotate=\p*\ph]
	(90-1*\ph:\r) -- (90:\r)
	(90:\r) -- (90+\bR:\R)
	(90+\bR:\R) -- (90:\r+2*\l)
	(90:\r) -- (90-\bR:\R)
	(90-\bR:\R) -- (90:\r+2*\l)
	(90+\bR:\R) -- (90-\bR+\ph:\R) 
;

\node at (90+\p*\ph:0.75*\r) {\small $\alpha$};

\node at (90+\p*\ph:1.25*\r) {\small $\gamma$};
\node at (90+\p*\ph:0.75*\r+2*\l) {\small $\gamma$};
\node at (90-0.15*\ph+\p*\ph:0.95*\R) {\small $\beta$};
\node at (90+0.15*\ph+\p*\ph:0.975*\R) {\small $\beta$};

\node at (90-0.175*\ph+\p*\ph:1.05*\r) {\small $\gamma$};
\node at (90+0.175*\ph+\p*\ph:1.05*\r) {\small $\beta$};
\node at (90-1.2*\bR+\p*\ph:0.88*\R) {\small $\beta$};
\node at (90+1.2*\bR+\p*\ph:0.88*\R) {\small $\gamma$};

\node at (90-\bR+\p*\ph:1.1*\R) {\small $\gamma$};
\node at (90+1.1*\bR+\p*\ph:1.1*\R) {\small $\beta$};
\node at (90-0.12*\ph+\p*\ph:0.85*\r+1.95*\l) {\small $\beta$};
\node at (90+0.1*\ph+\p*\ph:0.85*\r+2*\l) {\small $\gamma$};

\node at (90+\p*\ph:1.2*\r+2*\l) {\small $\alpha$};

}

\draw (0,0) circle (\r+2*\l);

\node[inner sep=1,draw,shape=circle] at (0,0) {\small $1$};
\node[inner sep=1,draw,shape=circle] at (90-0.5*\ph:\x+\r/2) {\small $2$};
\node[inner sep=1,draw,shape=circle] at (90+3.5*\ph:\x+\r/2) {\small $3$};
\node[inner sep=1,draw,shape=circle] at (90+2.5*\ph:\x+\r/2) {\small $4$};
\node[inner sep=1,draw,shape=circle] at (90+1.5*\ph:\x+\r/2) {\small $5$};
\node[inner sep=1,draw,shape=circle] at (90+0.5*\ph:\x+\r/2) {\small $6$};

\end{scope}

\end{tikzpicture}
\caption{The infinite family of tilings with $\beta^2\gamma, \alpha\beta\gamma^c $}
\label{Fig-a5-a4-Tilings-be2ga-albegac}
\end{figure}

For $c=2$, we illustrate the minimal member of the family in the second picture of Figure \ref{Fig-a5-a4-Tilings-be2ga-albegac}. 

\subsubsection*{Geometric Realisation}

Lastly, we show the geometric existence of these tilings. By $\beta^2\gamma, \alpha\beta\gamma^c$, we have
\begin{align} \label{Eq-be2ga-be-al}
\alpha = \pi - (c-\tfrac{1}{2})\gamma, \quad
\beta= \pi - \tfrac{1}{2}\gamma.
\end{align}
The geometric existence of the tilings means $\pi > \beta > \alpha > \frac{3}{5}\pi > \gamma$ and \eqref{Eq-cot-al-be-ga} is satisfied and $0 < \cos x < 1$. By the second identity of \eqref{Eq-be2ga-be-al}, the right hand side of \eqref{Eq-cot-al-be-ga} gives $\cos x = \frac{1}{2} (1 - \tan^2 \frac{1}{4}\gamma) \in (0, 1)$ for $\gamma \in (0, \pi)$. It suffices to check the first two conditions.

By $\gamma > 0$ and $c\ge2$ and \eqref{Eq-be2ga-be-al}, we have $\beta>\alpha$. Meanwhile, for $\alpha=\alpha(\gamma) \in (0, \pi)$, by the second equation in \eqref{Eq-be2ga-be-al} and \eqref{Eq-cot-al-be-ga} we get 
\begin{align*}
\alpha(\gamma) = 2 \sin^{-1} \frac{2 \cos \frac{1}{5}\pi}{ ( 3 -\tan^2 \frac{1}{4}\gamma )^{\frac{1}{2}} }.
\end{align*}
Then by the first equation of \eqref{Eq-be2ga-be-al}, we get
\begin{align} \label{Eq-be2ga-albegac-c(ga)}
c=c(\gamma) = \tfrac{1}{\gamma}( \pi - \alpha(\gamma) ) + \tfrac{1}{2}.
\end{align}
For $\alpha > \frac{3}{5}\pi$ and $c\ge2$, the first equation of \eqref{Eq-be2ga-be-al} also implies $\gamma \in (0, \frac{2}{5}\pi)$. The inequality $\alpha(\gamma)>\frac{3}{5}\pi$ is equivalent to 
\begin{align*}
\sin^{-1} \frac{2 \cos \frac{1}{5}\pi}{ (3 -\tan^2 \frac{1}{4}\gamma)^{\frac{1}{2}} }>\tfrac{3}{10}\pi. 
\end{align*}
As $\sin \theta$ is strictly increasing on $(0,\frac{1}{2}\pi)$ and $\cos \frac{1}{5}\pi = \sin \frac{3}{10}\pi = \frac{1}{4}(1+\sqrt{5})$, it is equivalent to $2 >( 3 -\tan^2 \frac{1}{4}\gamma )^{\frac{1}{2}}$, which is true for $\gamma \in (0, \pi)$. Hence we always have $\alpha(\gamma)>\gamma$. 

It remains to show that, for any integer $c\ge2$, there is a $\gamma_c \in (0, \frac{2}{5}\pi)$ satisfying $c(\gamma_c)=c$. One can show that \eqref{Eq-be2ga-albegac-c(ga)} is continuous and decreasing on $(0, \frac{2}{5}\pi)$ and $c(\frac{2}{5}\pi) =1 < 2$ and 
\begin{align*}
\lim_{\gamma \to 0^+} c(\gamma)= + \infty.
\end{align*}
Then Intermediate Value Theorem implies that the desired $\gamma_c$ exists and is unique.

Therefore we conclude that the infinite family of tilings in Figure \ref{Fig-a5-a4-Tilings-be2ga-albegac} exist.
\end{proof}

\section{Tilings by Rhombi and $m$-gons with $m\ge6$} \label{Sec-a6+-a4-tilings}

\begin{prop}\label{Prop-am-a4-be2ga} There is no dihedral tiling with angle sum $2\beta+\gamma=2\pi$ for $m\ge6$.
\end{prop}

\begin{proof} By $2\beta+\gamma=2\pi$, we get
\begin{align*}
\beta = \pi - \tfrac{1}{2}\gamma.
\end{align*} 
Substituting the above into \eqref{Eq-cot-al-be-ga}, we get
\begin{align} \label{Eq-am-a4-be2ga-cosx}
\cos x = \tfrac{1}{2}(1 - \tan^2 \tfrac{1}{4}\gamma).
\end{align}

An $m$-gon with maximum perimeter is the circumference. By $\alpha<\pi$ and $m \ge 6$, this implies $x < \frac{2}{m}\pi \le \frac{1}{3}\pi$. Since $\cos \theta$ is strictly decreasing on $(0, \frac{1}{2}\pi)$, we further get $\cos x > \cos \frac{2}{m}\pi$. Combined with \eqref{Eq-am-a4-be2ga-cosx}, we get
\begin{align*}
\tfrac{1}{2} > \tfrac{1}{2}(1 - \tan^2 \tfrac{1}{4}\gamma) > \cos \tfrac{2}{m}\pi,
\end{align*} 
which implies $m<6$, a contradiction. Hence there is no tiling.
\end{proof}

\begin{prop} There is no dihedral tiling with vertex $\alpha^2\gamma$ for $m\ge6$.
\end{prop}

\begin{proof} By $\alpha^2\gamma$, we have $\alpha<\pi$ and $\alpha+\gamma>\pi$. Then $\alpha+\gamma,\beta+\gamma>\pi$ implies $R(\alpha\beta)<2\gamma$. By $\alpha,\beta>\gamma$, we get $\alpha\beta\cdots=\alpha\beta\gamma$. Lemma \ref{Lem-albe-alga} implies that $\alpha\beta\cdots=\alpha\beta\gamma$ is a vertex. Then $\alpha^2\gamma, \alpha\beta\gamma$ implies $\alpha=\beta$. Hence we have $2\beta+\gamma$. Therefore Proposition \ref{Prop-am-a4-be2ga} implies that there is no dihedral tiling.
\end{proof}

\begin{prop}\label{Prop-am-a4-albega} The dihedral tilings with vertex $\alpha\beta\gamma$ and $m\ge6$ are in Figure \ref{Fig-a5-a4-Tilings-albega}.
\end{prop}

\begin{proof} By $\beta+\gamma > \pi$, we get $R(\beta^2)<2\gamma$. Then $\alpha,\beta>\gamma$ implies $\beta^2\cdots=\alpha\beta^2, \beta^3, \beta^2\gamma$. By \eqref{Eq-am-a4-deg3-list}, we know that $\alpha\beta^2, \beta^3$ are not vertices and by Proposition \ref{Prop-am-a4-be2ga} we also know that $ \beta^2\gamma$ is not a vertex. Hence $\beta^2\cdots$ is not a vertex.

By no $\beta^2\cdots$, we know that $\gamma\vert\gamma\cdots$ is not a vertex. Then $\beta\gamma^c$ is not a vertex. By $\alpha\beta\gamma$ and no $\beta^2\cdots, \beta\gamma^c$, we get $\beta\cdots=\beta\gamma\cdots=\alpha\beta\gamma\cdots=\alpha\beta\gamma$. Counting Lemma implies $\gamma\cdots=\alpha\beta\gamma$. So the other vertices consist of only $\alpha$'s. For $m\ge6$, we have $\alpha>(1-\frac{2}{m})\pi \ge \frac{2}{3}\pi$ and $\alpha^a$ is not a vertex. Hence we have the vertices below,
\begin{align*}
\AVC = \{ \alpha\beta\gamma \}.
\end{align*}
Starting at an $\alpha\beta\gamma$, by $\alpha\beta\cdots=\alpha\gamma\cdots=\beta\gamma\cdots=\alpha\beta\gamma$ we determine a tiling in the first picture of Figure \ref{Fig-a5-a4-Tilings-albega}. The tiling for $m=6$ is given in the second picture.

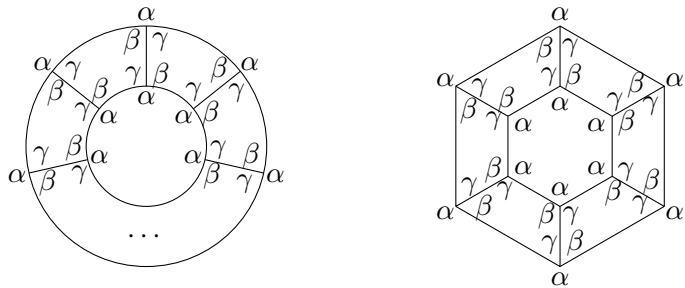
\begin{figure}[h!] 
\centering
\begin{tikzpicture}

\tikzmath{
\s=1;
\r=0.8;
\g=6;
\ph=360/\g;
\x=\r*cos(\ph/2);
\y=\r*sin(\ph/2);
\rr=2*\y/sqrt(2);
\h=4;
\th=360/\h;
\xx=\rr*cos(\th/2);
\ps=360/7;
}

\begin{scope}[] 

\foreach \a in {0,1,2,5,6} {

\draw[rotate=\a*\ps]
	(90:\r) -- (90:2*\r)
;

\node at (90+\a*\ps:0.8*\r) {\small $\alpha$};

\node at (90-0.225*\ps+\a*\ps:1.2*\r) {\small $\beta$};
\node at (90+0.2*\ps+\a*\ps:1.2*\r) {\small $\gamma$};

\node at (90-0.15*\ps+\a*\ps:1.75*\r) {\small $\gamma$};
\node at (90+0.15*\ps+\a*\ps:1.75*\r) {\small $\beta$};

\node at (90+\a*\ps:2.2*\r) {\small $\alpha$};

}

\draw (0,0) circle (\r);

\draw (0,0) circle (2*\r);

\node at (270:1.5*\r) {\small $\cdots$};

\end{scope} 

\begin{scope}[xshift=5.5*\s cm]

\foreach \p in {0,...,5} {

\draw[rotate=\p*\ph]
	(90-1*\ph:\r) -- (90:\r)
	(90:\r) -- (90:2*\r)
	(90-1*\ph:2*\r) -- (90:2*\r)
;

\node at (90+\p*\ph:0.7*\r) {\small $\alpha$};

\node at (90-0.18*\ph+\p*\ph:1.175*\r) {\small $\beta$};
\node at (90+0.15*\ph+\p*\ph:1.175*\r) {\small $\gamma$};

\node at (90-0.1*\ph+\p*\ph:1.65*\r) {\small $\gamma$};
\node at (90+0.15*\ph+\p*\ph:1.625*\r) {\small $\beta$};

\node at (90+\p*\ph:2.2*\r) {\small $\alpha$};

}

\end{scope}

\end{tikzpicture}
\caption{The dihedral tilings by rhombi and $m$-gon ($m\ge6$) with $\alpha\beta\gamma$}
\label{Fig-a5-a4-Tilings-albega}
\end{figure}

The existence of these tilings has been discussed in the Geometric Realisation of Proposition \ref{Prop-a5-a4-albega}.
\end{proof}

\end{document}